\documentclass[12pt, a4paper]{amsart} 

\usepackage[english]{babel}
\usepackage{amsfonts,amsmath,amsbsy,amstext,amssymb,amsopn,amsfonts,latexsym,amsthm,stmaryrd,mathptmx}
\usepackage{comment} 
\usepackage[usenames,dvipsnames]{color}
\usepackage{hyperref} \hypersetup{%
  colorlinks=true, 
  linkcolor=Maroon, 
  citecolor=NavyBlue, 
  filecolor=magenta, 
  urlcolor=cyan 
}

\newcommand{\cyan}{\textcolor{cyan}}

\newcommand{\red}{\textcolor{red}}

\numberwithin{equation}{section} \newtheorem{theorem}{Theorem}[section]
\newtheorem{lemma}[theorem]{Lemma}
\newtheorem{proposition}[theorem]{Proposition}
\newtheorem{corollary}[theorem]{Corollary}
 \theoremstyle{definition}
\newtheorem{definition}[theorem]{Definition}
\newtheorem{notation}[theorem]{Notation}
 
\newtheorem{remark}[theorem]{Remark}

\newcommand{\numberfield}[1]{\mathbb{#1}}
\newcommand{\R}{\numberfield{R}} 
\newcommand{\Z}{\numberfield{Z}} 
\newcommand{\N}{\numberfield{N}} 
\newcommand{\CC}{\numberfield{C}} 
\newcommand{\T}{{\numberfield{T}}} 

\newcommand{\D}{{\mathrm d}} 
\newcommand{\dist}{{\mathrm{dist}}} 
\newcommand{\Vol}{{\mathrm{Vol}}} 
\newcommand{\tr}{{\mathrm{tr}}} 
\newcommand{\hei}{\mathfrak{h}^g} 
\newcommand{\Hei}{{\mathsf H}_{\text{\it red}}^{g}} 

\newcommand{\Torus}{{\mathsf T}} 
\newcommand{\Heis}{{\mathsf H}^{g}} 
\newcommand{\Lattice}{{{\mathsf \Gamma}}} 
\newcommand{\Latt}{{{\mathsf \Lambda}}} 
\newcommand{\Subgroup}{{\mathsf P}} 
\newcommand{\bigAuto}{\operatorname{Aut}_0(\Heis)} 
\newcommand{\Auto}{\Symp_{2g}(\R)} 
\newcommand{\Mod}{\mathfrak M_g} 
\newcommand{\GMod}{\Symp_{2g}(\Z)} 

\newcommand{\Birkhoff}{\mathcal P} 
\newcommand{\step}{\tau} 
\newcommand{\radius}{T} 

\newcommand{\Cartan}{\mathsf{A} } 
\newcommand{\cartan}{\frak{a} } 

\def\G{\mathsf{G}} 
\newcommand{\M}{\mathsf{M}} 
\newcommand{\Center}{\mathsf{Z}} 
\newcommand{\SL}{\mathrm{SL}} 
\newcommand{\GL}{\mathrm{GL}} 
\newcommand{\Uni}{\mathsf K } 
\newcommand{\Symp}{\mathrm{Sp}} 
\newcommand{\symp}{\mathfrak{sp}} 
\newcommand{\SO}{\mathrm{SO} } 
\newcommand{\subLap}{{H}} 
\newcommand{\Lap}{{\Delta}} 
\newcommand{\Int}{\mathcal{I}} 
\newcommand{\Prim}{\mathcal{P}} 
\newcommand{\Egauss}{\mathcal{E}} 
\newcommand{\Khom}{\mathcal{K}} 

\newcommand{\Sym}{\mathrm{Sym} } 
\newcommand{\Siegel}{\mathfrak H} 
\newcommand{\fq}{\mathcal{Q}} 
\newcommand{\fl}{\ell} 

\newcommand{\Sobolev}{W} 
\newcommand{\Sob}{\mathcal K} 
\newcommand{\SB}{\mathcal S} 
\newcommand{\Sc}{\rho} 
\newcommand{\Sch}{\rho} 
\newcommand{\sch}{\rho} 
\newcommand{\Bundle}{\mathfrak W} 
\newcommand{\norm}{\interleave} 

\newcommand{\lin}{\left\langle} 
  \newcommand{\rin}{\right\rangle} 
\def\paragraph#1{\subsubsection*{{\bf{#1}}}}

\def\height{\operatorname{\daleth}}
\def\Height{\operatorname{Hgt}} \def\detY{\operatorname{hgt}}

\def\formbundle#1#2#3{A^{#1}(\mathfrak p^{#2},\Bundle^{#3})}
\def\currentbundle#1#2#3{A_{#1}(\mathfrak p^{#2},\Bundle^{#3})}
\def\closedbundle#1#2#3{Z_{#1}(\mathfrak p^{#2},\Bundle^{#3})}
\def\restbundle#1#2#3{R_{#1}(\mathfrak p^{#2},\Bundle^{#3})}
\def\class#1{[[#1]]}
\def\dcart{\widehat{\delta}}
\def\isotropic{isotropic\ }

\begin{document}

\author{Salvatore Cosentino}
\address{Centro de Matem\' atica,  Universidade do Minho, Campus de Gualtar, 4710-057 Braga,    PORTUGAL.}  \email{  scosentino@math.uminho.pt } 
      
 \author{Livio Flaminio}
\address{Unit\' e Mixte de Recherche CNRS 8524, Unit\' e de Formation et Recherche    de Math\' ematiques, Universit\' e de Lille 1, F59655 Villeneuve  d'Asq CEDEX, FRANCE. } 
\email{livio.flaminio@math.univ-lille1.fr}

\title[Equidistribution on  Heisenberg nilmanifolds]{Equidistribution for higher-rank Abelian actions on  Heisenberg nilmanifolds}


\date{\today}


\maketitle

\begin{abstract}
  We prove quantitative equidistribution results for actions of
  Abelian subgroups of the $2g+1$ dimensional Heisenberg group acting
  on compact $2g+1$-dimensional homogeneous nilmanifolds.  The results
  are based on the study of the $C^\infty$-cohomology of the action of
  such groups, on tame estimates of the associated cohomological
  equations and on a renormalisation method initially applied by Forni
  to surface flows and by Forni and the second author to other
  parabolic flows.  As an application we obtain bounds for finite Theta sums
  defined by real quadratic forms in $g$ variables, generalizing the
  classical results of Hardy and Littlewood \cite{MR1555099,
    MR1555214} and the optimal result of  Fiedler, Jurkat and
  K\"orner \cite{MR0563894} to higher dimension.
\end{abstract}

\tableofcontents
\section{Introduction}

In the analysis of the time evolution of a dynamical system many problems
reduce to the study of the \emph{cohomological equation}; in the case, for
example, of a smooth vector field $X$ on a connected compact manifold $M$ this
means finding a function $u$ on~$M$ solution of the equation
\begin{equation}
  \label{eq:cohoeq}
  Xu = f\,,
\end{equation}
where $f$ is a given function on~$M$.

For a detailed discussion of the cohomological equation for flows and
tranformations in ergodic theory the reader may consult \cite{MR2008435}.

For higher dimensional Lie groups, the study of the cohomology of their actions
(or the related cohomology of lattice sub-groups) plays a fundamental r\^ole in
many works; to cite just a few, we recall R.~Zimmer's cocycle super-rigidity
theorem (\cite{MR776417}), and the numerous works of A.~Katok \emph{et al.}\ on
abelian actions (\cite{MR1336707, MR1632177, MR2726100, MR2600682, MR2726100,
  MR2798364}, \dots).

\paragraph{Cohomology in Heisenberg manifolds} In this article we study the cohomology of the action of an abelian
subgroup~$\Subgroup$ of the $(2g+1)$-dimensional Heisenberg group $\Heis$ on
the algebra of smooth functions on a homogeneous manifold
$\Heis / \Lattice$. The linearity of the problem and the fact that the
unitary dual of $\Heis$ is classical knowledge make the use of harmonic
analysis particularly suitable to our goal, as it was the case in the works of
L.~Flaminio and G.~Forni~\cite{MR2003124, MR2218767, MR2261071}.  As a
consequence, our results on the cohomology of $\Subgroup$ also apply to more
general $\Heis$-modules, those for which the action of the center of $\Heis$
has a spectral gap.

Before stating ours results let us  fix some notation.

Let $\mathsf G$ be a connected Lie group of Lie algebra $\mathfrak g$, and let $\M=
\mathsf G/\Lattice$ be a compact homogeneous space of $\mathsf G$. Then $\mathsf G$ acts by left translations on $C^\infty(\M)$
via
\begin{equation}
  \label{eq:1}
  (h. f) (m) = f(h^{-1} m), \quad h\in \mathsf G, \quad f\in C^\infty(\M).
\end{equation} 
Let $F$ be a closed $\mathsf G$-invariant subspace of $C^\infty(\M)$.
The space $F$ is a tame graded Fr\'echet space
(\cite[Def. II.1.3.2]{MR656198}) topologized by the family of
increasing Sobolev norms~$\| \cdot \|_s$, defining $L^2$ Sobolev
spaces $\Sobolev^s(\M)$.

For any connected Lie subgroup $\Subgroup< \mathsf G$ with Lie algebra
$\mathfrak p$, the action by translations of $\Subgroup$ on $\mathsf G/\Lattice$
turns $F$ into a $\mathfrak p$-module.  Therefore we may consider the
Chevalley-Eilenberg cochain complex $A^\ast (\mathfrak p ,
F):=\Lambda^*\mathfrak p' \otimes F$ of $F$-valued alternating forms on
$\mathfrak p$, endowed with the usual differential ``$\D{}$''.
By \emph{cohomology
  of the  $\mathfrak p$-module  $ F$ } we simply mean the
Lie-algebra cohomology $ H^{*} (\mathfrak p, F)$ of this cochain
complex. When $F=C^\infty(\M)$ we also refer to this cohomology as
\emph{the cohomology of the action of $\mathsf P$ on $\M$}.

A natural question that arises when we consider a Lie group or Lie
algebra cohomology with values in a topological module, is whether the
\emph{reduced} cohomology coincides with ordinary cohomology; that is
whether the spaces $B^\ast(\mathfrak p, F)$ of co-boundaries are
closed in the spaces $Z^\ast(\mathfrak p, F)$ of cocycles. Following
A. Katok ~\cite{MR1858535}, we give the following definition. 
 
\begin{definition}
  \label{def:sect1:1}
  The  $\mathfrak p$-module  $ F$  is \emph{cohomologically $C^\infty$-stable
    in degree~$k$} if the space $B ^k (\mathfrak p , F)$ of $F$-valued
  co-boundaries of degree~$k$ is closed in the $C^\infty$ topology.
\end{definition}
  
Let $Z_k (\mathfrak p , F)$ denote the space of closed currents of
dimension $k$, that is the space of all continuous linear functionals
on $ A^k (\mathfrak p , M)$ vanishing on $B ^k (\mathfrak p , F)$. By
the Hahn-Banach Theorem, $B ^k (\mathfrak p ,F)$ is a closed subspace
of $A^k (\mathfrak p , F)$ if and only if it is equal to the
intersection of the kernels of all $D\in Z_k (\mathfrak p , F)$.

We recall that a tame linear map $\phi: F_1 \to F_2$ between tame
graded Fr\'echet spaces satisfies a tame estimates of degree $r$ with
base $b$ if, denoting by $\Vert\cdot\Vert_s$ the norms defining the
grading, we have $\Vert \phi(f) \Vert_{s} \le C \Vert f \Vert_{s+r}$
for all $s\ge b$ and $f\in F_1$; the constant $C$ may depends on $s$.

The tame grading of $F$ implies that  $A^\ast (\mathfrak p , F)$ is a  tame
graded Fr\'echet cochain complex and that the differentials are  tame maps of
degree~$1$. Thus, besides $C^\infty$-stability, another question that
arises naturally is whether, for a given a co-boundary $\omega$, there
exists a primitive $\Omega$ whose norm is tamely estimated by the norm of
$\omega$.

\begin{definition}
  \label{def:sect3:2}
  We say that the $\mathfrak p$-module $ F$ is \emph{tamely
    cohomologically $C^\infty$-stable in degree~$k\ge 1$} if there exists a
  tame map $\D_{-1}\colon B ^k (\mathfrak p , F)\to A^{k-1} (\mathfrak
  p , F)$ assigning to every co-boundary $\omega\in B ^k (\mathfrak p
  , F) $ a primitive of $\omega$.
\end{definition}

A related question, which is fundamental in perturbation theory, is
whether the chochain complex $A^k (\mathfrak p , F)$ has a tame
splitting~\cite{MR656198}  (see \cite{MR2183300, katok-damianovich-parabolic}). 
Recall that a
graded Fr\'echet space $F_1$ is a tame summand of a graded Fr\'echet
space $F_2$ if there are tame maps $L:F_1\to F_2$ and $M\colon F_2\to
F_1$ such that $M\circ L$ is the identity map of $F_1$
\cite[Def. II.1.3.1]{MR656198}. In this situation we also say that the
short exact sequence $0\to F_1\to F_2\to F_2/L(F_1)\to 0$ splits
tamely.

\begin{definition}
  \label{def:sect3:3}
  We say that the $\mathfrak p$-module $ F$ has \emph{tame splitting}
  in degree $k$ if the space $B^{k}(\mathfrak p, F)$ is a tame direct
  summand of $A^{k}(\mathfrak p, F)$.
\end{definition}

Let $\Heis$ be the Heisenberg group of dimension $2g +1$. Any
compact homogeneous space $\M=\Heis/\Lattice$ is a circle bundle $p: \M
\to \Heis / ( \Lattice \, \Center (\Heis))$ over the $2g$-dimensional torus
$\Torus ^{2g} =\Heis / ( \Lattice \, \Center (\Heis))$, with fibers given by
the orbits of the center $ \Center (\Heis)$ of $\Heis$. The space of
$C^\infty$ functions on $\M$ splits as a direct sum of
$\Heis$-invariant subspace  
$\pi^* (C^\infty(\Torus ^{2g} ))$ and the $\Heis$-invariant subspace
$F_0=C^\infty_0(\M)$ formed by the smooth functions on $\M$ having zero
average on the fibers of the fibration~$p$. The following theorem is a
particular case of Theorem~\ref{thm:sect3:2} below.

\begin{definition}
  \label{def:sect1:2}
  A  connected Abelian subgroup of
  $\Heis$ without central elements  will be called \emph{an \isotropic
    subgroup of  $\Heis$}. A \emph{Legendrian} subgroup of  $\Heis$ is
  an \isotropic  subgroup of  $\Heis$ of maximal dimension~$g$.
\end{definition}
\begin{theorem}
  \label{thm:sect1:2}
  Let $\Subgroup$ be a $d$-dimensional \isotropic subgroup of
  $\Heis$  with Lie algebra $\mathfrak p$.
  The $\mathfrak p$-module $F_0$ is tamely cohomologically
  $C^\infty$-stable in all degrees. In fact, for all $k=1,\dots,d$
  there are linear maps
  \[\D{}_{-1}\colon B^k(\mathfrak p, F_0)\to
  A^{k-1}(\mathfrak p, F_0)\] associating to each $\omega\in B^k(\mathfrak
  p, F_0)$ a primitive of $\omega$ and satisfying tame estimates of degree
  $(k+1)/2+\varepsilon$ for any $\varepsilon >0$.

  We have $H^{k} (\mathfrak p, F_0)=0 $ for $k<d$; in degree $d$, we
  have that $H^{d} (\mathfrak p, F_0) $ is infinite dimensional if $d <g$ 
  or one-dimensional if
  $d=g$ (i.e. if $\mathfrak p$ is a Legendrian subspace) 
  in each irreducible
  $\mathfrak p$-sub-module of $F_0$.

  The $\mathfrak p$-module $ F_0$ has \emph{tame splitting} in all
  degrees: for $k=0, \dots, d$ and any $\varepsilon>0$, there exist a
  constant $C$ and linear maps
  \[
  M^k: A^{k}(\mathfrak p, F_0) \to B^k ( \mathfrak p, F_0)
  \]
  such that he restriction of $M^k$ to $B^k ( \mathfrak p,F_0)$ is the
  identity map and the following estimate hold
  \[
  \| M^k \omega\|_s \le C \|\omega\|_{s+w}, \quad \forall \omega \in
  A^{k}(\mathfrak p, F_0)
  \]
  where $w=(k+3)/2 + \varepsilon$, if $k<d$ and $w=d/2 + \varepsilon$ if
  $k=d$. 
\end{theorem}

Let $\Subgroup<\Heis $ be a subgroup as in the theorem above and let
$\bar\Subgroup$ be group obtained by projecting $\Subgroup$ on the
$\Heis/ \Center (\Heis)\approx\R^{2g}$. 
As before we set $\Torus ^{2g} =\Heis / ( \Lattice \, \Center (\Heis))$. The
$\Subgroup$-module $\pi^* (C^\infty(\Torus^{2g} ))$ is naturally
isomorphic to the $\bar\Subgroup$-module $C^\infty(\Torus ^{2g}
)$. It should be considered as folklore that the cohomology of the action
of a subgroup $\bar\Subgroup$ on a torus depends on the Diophantine
properties of $\bar\Subgroup$, considered as vector space.  
The Diophantine condition $\bar{\mathfrak p}\in \hbox{DC}_\tau(\bar
\Lattice)$ mentioned in the  theorem below will be precised in section
\ref{sect:3.1}.
\begin{theorem}
  \label{thm:sect1:1}
  Let\/ $\Subgroup$ be an \isotropic subgroup of\/ ${\Heis}$, let
  $\M:=\Heis/\Lattice$ be a compact homogeneous space and let
  $F:=C^\infty(\M)$.  Let $\bar \Subgroup$ be the projection of
  $\Subgroup$ into $\Heis/ \Center (\Heis)\approx \R^{2g}$, let
  $\bar{\mathfrak p}$ its Lie algebra, and let $\bar \Lattice =
  \Lattice/(\Lattice \cap \Center (\Heis))\approx \Z^{2g}$.  Then
  action of\/~$\Subgroup$ on $\M$ is tamely cohomologically
  $C^\infty$-stable and has a tame splitting in all degrees if and
  only if $\bar{\mathfrak p}\in \hbox{DC}_\tau(\bar \Lattice)$ for
  some $\tau >0$.  In this case we have
  \[
  \begin{split}
    H^{k} (\mathfrak p, F) = \Lambda^k\mathfrak p \text{ if } k < \dim
    \mathfrak p, \quad H^{k} (\mathfrak p, F) = \Lambda^k\mathfrak p
    \oplus  H^{k} (\mathfrak p, F_0) \text{ if } k = \dim \mathfrak p
  \end{split}
  \]
\end{theorem}


\paragraph{Equidistribution of isotropic subgroups on Heisenberg manifolds.}

In their fundamental 1914 paper~\cite{MR1555099} Hardy and Littlewood
introduced a renormalization formula to study the exponential sums
$\sum_{n=0}^N e(n^2x/2 +\xi n)$, usually called \emph{finite theta
  sums}, where $N\in \N$ and $e(t) :=\exp(2\pi i t)$. Their algorithm
provided optimal bounds for these sums when $x$ is of bounded type.

Since then, Hardy and Littlewood's renormalization method has been
applied or improved by several authors obtaining finer estimates on 
finite theta sums (Berry and Goldberg~\cite{MR928946}, Coutsias and
Kazarinoff~\cite{MR1443869}, Fedotov and
Klopp \cite{2009arXiv0909.3079F}).  Optimal estimates have obtained by
Fiedler, Jurkat and K\"orner~\cite{MR0563894}. Differently from the
previously quoted authors, who relied heavily on the continued
fractions properties of the real number $x$, Fiedler, Jurkat and
K\"orner's method was based on an approximation of $x$ by rational with
denominators bounded by $4N$.

In this paper we consider the $g$-dimensional generalization, the  finite
theta sums
\begin{equation}
  \label{def_theta_sums}
\sum_{ n \in \Z ^g \cap [0, N]^g} e \left( \fq [n] + \fl(n)  \right)
\end{equation}
where $\mathcal Q [x] := x^\top \fq x$ is the quadratic form defined 
by a symmetric $g\times g$ real matrix $\fq$, and $\fl (x) := \fl^\top x$   is the linear form 
defined by a vector $\fl \in \R^g$.  
In the spirit of Flaminio and Forni~\cite{MR2218767}, our method consists
into reducing the sum~\eqref{def_theta_sums} to a Birkhoff sum along the an
orbit (depending on $\fl$) of some subgroup (depending on $\fq$) of a
standard ($2g+1$)-dimensional Heisenberg nilmanifold and then using a
more general quantitative equidistribution result of some Abelian
group action on standard Heisenberg nilmanifolds.

The occurrence of Heisenberg nilmanifolds is not a surprise: in fact
the connection between the Heisenberg group and the theta series is
well known and very much exploited~\cite{MR0414785, MR0466409,
  MR487050, MR2218767, MR2352717, MR2307769}.

Let $\M = \Heis / \Lattice$ be the standard Heisenberg nilmanifold
(see Section~\ref{sec:heis-group-sieg} for details on the definitions
and notations). Let $(X_1, \dots, X_g,\Xi_1,\dots, \Xi_g ,T)$ be a
fixed  rational basis of $\hei = \mathrm{Lie}  (\Heis)$ satisfying
the canonical commutation relations. Then the symplectic group $ \Symp
_{2g} (\R)$ acts on $\Heis$ by automorphisms\footnote{by acting on the
  left on the components of elements of $\hei$ in the given
  basis.}. For $1\leq d \leq d$, let $\Subgroup ^d$ be the subgroup generated by $(X_1,
\dots, X_d)$ and, for any $\alpha \in \Auto$, set $X_i^\alpha := \alpha
^{-1} (X_i)$, $1\leq i \leq d$. We define a parametrization of the
subgroup $\alpha^{-1} ( \Subgroup ^d)$ according to
\[
\Subgroup_x^{d,\alpha} := \exp (x_1 X_1^\alpha+\dots+x_d
X_d^\alpha), \qquad x=(x_1,\dots,x_d) \in \R^d.
\]
Given a Jordan region $U \subset \R^d$ and a point $m \in \M$, we
define a $d$-dimensional $\mathfrak p$-current
$\Birkhoff_{U}^{d,\alpha}m$ by
\begin{equation}
  \label{eq:sect1:1}
\lin \Birkhoff^{d,\alpha}_{U}m, \omega \rin := \int _{U} f (
\Subgroup^{d,\alpha}_x m ) \, \D x
\end{equation}
for any
degree $d$ $\mathfrak p$-form $\omega = f\, \D X_1^\alpha \wedge
\dots \wedge \D X_d^\alpha$, with $f \in C_0^\infty (\M)$ (here
$C_0^\infty (\M)$ denote the space of smooth functions with zero
average along the fibers of the central fibration of $\M$).

It is well-known that the Diophantine properties of a real number may
be formulated in terms of the speed of excursion, into the cusp of the
modular surface, of a geodesic ray having that number as limit point on
the boundary of hyperbolic space. This observation allows us to define
the Diophantine properties of the subgroup $\Subgroup^{d,\alpha}$ in
terms of bounds on the \emph{height} of the projection, in the Siegel
modular variety $\Sigma_g=\Uni_g\backslash \Auto/\GMod$, of the orbit
of $\alpha$ under the action of some one-parameter semi-group of the
Cartan subgroup of $\Symp _{2g} (\R)$ (here $\Uni_g$ denotes the
maximal compact subgroup of $\Auto$). 
We refer to
Section~\ref{subsection:Diophantine} for the definition of height
function. 

Let $\exp t\dcart(d)$ be the Cartan subgroup of $\Symp _{2g}
(\R)$ defined by $\exp (t\dcart(d)) X_i= e^t X_i$, for $i=1,\dots, d$
and $\exp (t\dcart(d)) X_i= X_i$, for $i=d+1,\dots, g$.  Roughly, the
definition~\ref{def:Diophantine_properties} states that
$\alpha\in \Auto$
satisfies a $\dcart(d)$-Diophantine condition of type $\sigma$,
if the height
of the projection of $\exp (-t\dcart(d))\alpha$ in the Siegel modular
variety $\Sigma_g$ is bounded by $e^{2td(1-\sigma)}$; if,
for any $\varepsilon>0 $, the height considered above is bounded by
$e^{2td\varepsilon}$, then we say that $\alpha\in \Auto$ satisfies a
$\dcart(d)$-Roth condition; finally we say that  $\alpha$ is of bounded
type if the height  of $\exp (-\dcart ) \alpha$,  
stays bounded as $\dcart$ ranges in a positive cone $\cartan ^+ $ in  the
Cartan algebra  of diagonal symplectic matrices (see Def.~\ref{def:Diophantine_properties}).

As the height function is defined on the Siegel modular variety
$\Sigma_g$, the Diophantine properties of $\alpha$ depend only on its
class $[\alpha]$ in the quotient space $\Mod= \Auto/\GMod$.

The definitions above agree with the usual definitions in the $g=1$
case. Several authors (Lagarias \cite{MR662052}, Dani \cite{MR794799},
Kleinbock and Margulis \cite{MR1719827},
Chevallier~\cite{zbMATH06315600}) proposed, in different contexts,
various generalizations of the $g=1$ case: we postpone to
Remark~\ref{rem:Diophantine} the discussion of these
generalizations.
  
We may now state our main equi-distribution result.

\begin{theorem} \label{equidistribution_intro} Let $\Subgroup ^d <
  \Heis$ be an \isotropic subgroup of dimension $d\leq g$.  Set
  $Q(T)=[0,T]^d$. For any $s>\tfrac{1}{4}d (d+11)+g+1/2$ and any
  $\varepsilon >0$ there exists a constant $C=C(\Subgroup,\alpha,s,g,
  \varepsilon)>0$ such that, for all $\radius \gg 1$ and all test
  $\mathfrak p$-forms $\omega \in \Lambda^d\mathfrak p\otimes
  \Sobolev^s_0 (M)$,

  \begin{itemize}
   \item there exists a full measure set  $\Omega_g (w_d)\subset \Mod$ such that if $[\alpha] \in \Omega _g(w_d)$  then 
     \[
    \left| \lin \Birkhoff^{g,\alpha}_{Q(\radius)}m, \omega \rin \right| \leq
    C \, ( \log \radius) ^{d+1/(2g+2)+\varepsilon} \, \radius ^{d/2} \, \| \omega \| _{s}
    \]
 \item if $[\alpha ] \in \Mod$ satisfies a $\dcart(d)$-Diophantine  condition of exponent $\sigma>0$ then
     \[ \left| \lin \Birkhoff^{d,\alpha}_{Q(\radius)}m, \omega \rin \right| \leq
    C \,  \radius ^{d(1-\sigma'/2)} \, \| \omega \| _{s} \, ,
    \]
for all $\sigma ' < \sigma$; 
  \item if $[\alpha ] \in \Mod$ satisfies a $\dcart(d)$-Roth condition, then
    \[ \left| \lin \Birkhoff^{d,\alpha}_{Q(\radius)}m, \omega \rin \right| \leq
    C \, \radius ^{d/2+\varepsilon} \, \| \omega \| _{s} \, ,
    \]
  \item if $[ \alpha ]  \in \Mod$ is of bounded type, then
    \[
    \left| \lin \Birkhoff^{d,\alpha}_{Q(\radius)}m, \omega \rin \right| \leq
    C \, \radius ^{d/2} \, \| \omega \| _{s}
    \]
     \end{itemize}
\end{theorem}

The exponent of the logarithmic factor in the first case is certainly
not optimal.  When $d=1$, a more precise result is stated in
Proposition \ref{result_one_dim_cor} which coincides with the optimal
classical result for $d=g=1$ (Fiedler, Jurkat and
K\"orner~\cite{MR0563894}).

The method of proof is, to our knowledge, the first generalization of
the methods of renormalization of Forni
(\cite{MR1888794}) and of Flaminio and Forni
(\cite{MR2218767, MR2261071}) to actions of higher dimensional
Lie groups. 

A limitation of the inductive scheme that we adopted is that we are
limited to consider averages on cubes $Q(T)$ (the generalization to
pluri-rectangles is however feasible, but more cumbersome to
state). For more general regions, growing by homotheties, we can
obtain weak estimates where the power $T^{d/2}$ is replaced by
$T^{d-1}$. However, N.~Shah's ideas \cite{shah2009} suggest that
equi-distributions estimates as strong as those stated above are valid
for general regions with smooth boundary.


The application to $g$-dimensional finite theta sums
\eqref{def_theta_sums} is the following corollary of Theorem
\ref{pretheta_sums}.

\begin{corollary}\label{theta_sums}
  Let $\fq [x]= x^\top \fq x $ be the quadratic form defined by  the symmetric $g \times g$ real matrix $\fq$, 
  let $\alpha = \left( \begin{smallmatrix} I & 0 \\ \fq & I \end{smallmatrix}
  \right) \in \Symp_{2g}(\R)$, and let $\fl (x)= \fl ^\top x $ be the linear form defined by  $\fl \in \R^g$. Set
  \[
  \Theta (\fq, \fl ; N ) := N^{-g/2}\sum_{ n \in \Z ^g \cap [0, N]^g} e\left(  \fq [n] + \fl(n) \right)  \, . 
  \]
  \begin{itemize}
  \item There exists a full measure set $\Omega_g \subset \Mod$ such
    that if $[\alpha] \in \Omega _g$ and $\varepsilon >0$ then
    \[
    \Theta (\fq , \fl ; N )= \mathcal O \left( (\log N
      )^{g+1/(2g+2)+\varepsilon }  \right)
    \]
  \item If $[\alpha] \in \Mod$ satisfies a $\dcart(g)$-Roth condition, then
    for any $\varepsilon >0$.
    \[
    \Theta (\fq , \fl ; N )  = \mathcal O \left( N^{\varepsilon}
    \right)
    \]
  \item If $[\alpha] \in \Mod$ is of bounded type, then
    \[
    \Theta (\fq , \fl  ; N ) =\mathcal O \left(1\right)
    \]
  \end{itemize}
\end{corollary}

The Diophantine conditions in terms of the symmetric matrix  $\fq$ 
are written and discussed in remark \ref{rem:Diophantine}. 

As we mentioned above, dynamical methods have already been used to
study the sums $\Theta (\fq, \fl ; N)$. G\" otze and
Gordin~\cite{MR1979719},  generalizing \cite{MR1682276}, show that
some smoothings of $\Theta (\fq , \fl ; N )$ have a limit distribution. 
See also Marklof~\cite{MR1691543, zbMATH02052017}.

Geometrical methods, similar to ours, to estimate finite theta sums
are also used by Griffin and Marklof~\cite{zbMATH06376860} and
Cellarosi and Marklof~\cite{2015arXiv150107661C}. They focus on the
the distributions of these sums as $\fq$ and $\fl$ are uniformly
distributed in the $g=1$ case. As they are only interested in theta
sums, they may consider a single irreducible representation $\rho$ of
the Heisenberg group and a single intertwining operator between $\rho$
and $L^2(\M)$. The other more technical difference is that as
$\fq$ and $\fl$ vary, it is more convenient to generalize the ergodic
sums~\eqref{eq:sect1:1} to the case when $\omega$ is  transverse current.

Estimates of theta sums are also crucial in the paper of G\"otze and
Margulis~\cite{2010arXiv1004.5123G}, which focuses on the finer
aspects of the ``quantitative Oppenheim conjecture''. There is
question of estimating the error terms when counting the number of
integer lattice points of given size for which an indefinite
irrational quadratic form takes values in a given interval. This is
clearly a subtler problem  than the one considered here.


\paragraph{Article organization.} In Section 2, we
introduce the necessary background on the Heisenberg and symplectic
groups. In section 3 we prove the results about the cohomology of
\isotropic subgroups of the Heisenberg groups. Section 4 deals with
the relation between Diophantine properties and dynamics on the Siegel
modular variety. Finally in section 5 we prove the main
equidistribution result and the applications to finite theta sums.

\smallskip
Applications to the rigidity problem of higher-rank Abelian
actions on  Heisenberg nilmanifolds, as a consequence of the tame
estimates for these actions, will be the subject of further works.

\smallskip
\noindent \textit{Acknowledgements. } This work was partially done
while L.~Flaminio visited the Isaac Newton Institute in Cambridge,
UK. He wishes to thank the Institute and the organisers of the
programme \emph{Interactions between Dynamics of Group Actions and
  Number Theory} for their hospitality. L.~Flaminio was supported in
part by the Labex~CEMPI (ANR-11-LABX-07).  S. Cosentino was partially
supported by CMAT - Centro de Matem\' atica da Universidade do Minho,
financed by the Strategic Project PEst-OE/MAT/UI0013/2014.


\section{Heisenberg group and Siegel symplectic geometry}
\label{sec:heis-group-sieg}
\subsection{The Heisenberg group and the Schr{\"o}dinger representation}\label{ssec:heis-group-schr}

\paragraph{The Heisenberg group and Lie algebra.}
\label{par:heis_not}
Let $\omega$ denote the \emph{canonical symplectic form } on $\R^{2g} \approx
\R^{g}\times \R^g $, i.e.\ the non-degenerate alternate bilinear form $\omega
((x,\xi), (x',\xi')) = \xi \cdot x' - \xi' \cdot x$, where we use the notations
$(x,\xi ) \in \R^g \times \R^g $ and $\xi \cdot x := \xi_1x_1+\dots + \xi_g
x_g$.  The \emph{Heisenberg group} over $\R^g$ (or the {\em real
  $(2g+1)$-dimensional Heisenberg group}) is the set $\Heis = \R^{g} \times
\R^g \times \R $ equipped with the product law
\begin{equation}
  \label{eq:sect2:1}
  (x, \xi , t) \cdot (x', \xi ', t') = 
  (x+x', \xi+\xi ', t+t' + \tfrac{1}{2} \omega ((x,\xi), (x',\xi')) ) 
\end{equation}
It is a central extension of $\R^{2g}$ by $\R$, as we have an exact sequence
\begin{equation*}
  \label{eq:sect2:2}
  0 \rightarrow \Center (\Heis)\rightarrow \Heis
  \rightarrow \R^{2g} \rightarrow 0 \, ,
\end{equation*}
with $\Center (\Heis)=\{ (0,0, t) \} \approx \R $.

The Lie algebra of $\Heis$ is the vector space $\hei =\R^g \times \R^g \times
\R$ equipped with the commutator
\[ [ (q, p ,t) , (q', p ' , t') ] = (0, 0, p\cdot q' - p ' \cdot q) \, .
\]
Let $T=(0,0,1) \in Z(\hei)$.  If $(X_i)$ is a basis of
$\R^g$, and $(\Xi_i)$ the symplectic dual basis, we obtain a basis
$(X_i,\Xi_j,T)$ of $\hei$ satifying the \emph{canonical commutation relations}:
\begin{equation}
  \label{eq:sect2:3}
  [X_i,X_j]=0,\quad [\Xi_i,\Xi_j]= 0,\quad [\Xi_i,X_j]= \delta_{ij} T, 
  \qquad 1\le i,j\le g.
\end{equation}

A basis $(X_i,\Xi_j,T)$ of~$\hei$ satisfying the relations~(\ref{eq:sect2:3})
will be called a \emph{Heisenberg basis of $\hei$}. The Heisenberg basis
$(X_i^0,\Xi_j^0,T)$ where $X_i^0$ and $\Xi_j^0$ are the standard bases of
$\R^g$, will be called the \emph{standard Heisenberg basis}.

Given a Lagrangian subspace of $\mathfrak l\subset\R^g \times (\R^g)'$, there
exists a Heisenberg basis $(X_i,\Xi_j,T)$ such that $(X_i)$ spans $\mathfrak
l$; in this case the span $\mathfrak l'=\lin\Xi_j\rin$ is also Lagrangian and
we say that the basis $(X_i,\Xi_j,T)$ is \emph{adapted to the splitting
  $\mathfrak l\times \mathfrak l'\times Z(\hei) $} of $\hei$.

\paragraph{Standard lattices and quotients.}


The set $\Lattice := \Z^g \times \Z^g \times \frac{1}{2} \Z$ is a discrete and
co-compact subgroup of the Heisenberg group $\Heis$, which we shall call the
\emph{standard lattice} of $\Heis$.  The quotient
\[
\M := \Heis/ \Lattice
\]
is a smooth manifold that will be called the \emph{standard Heisenberg
  nilmanifold}.  The natural projection map
\begin{equation}
  \label{circle_fibration}
  p\colon \M \to  \Heis/ (\Lattice \Center (\Heis))\approx  ( \Heis/
  Z(\Heis))/(\Lattice/\Lattice \cap \Center (\Heis))
\end{equation}
maps $\M$ onto a $2g$-dimensional torus $\T ^{2g} := \R^{2g} / \Z^{2g}$.  All
lattices of $\Heis$ were described by Tolimieri in \cite{MR487050}.
Henceforth we will limit ourselves to consider only a standard Heisenberg
nilmanifold, our results extending trivially to the general case. Observe that
$\exp T$ is the element of $\Center (\Heis)$ generating $\Lattice\cap \Center(\Heis)$.

\paragraph{Unitary $\Heis$-modules and Schr{\"o}dinger representation.}

The \emph{Schr\"odinger representation} is a unitary representation of $\Sc :
\Heis \to U (L^2(\R^g))$ of the Heisenberg group into the group of unitary
operators on $L^2(\R^g)$; it is explicitly given by
\[ (\Sc (x, \xi , t) \varphi )(y) = e^{i t - i \xi \cdot y - \tfrac{1}{2} i \xi
  \cdot x} \varphi (y+x), \quad ( \varphi \in L^2 (\R^g),\enspace(x, \xi , t)
\in \Heis ).
\]
(see \cite{MR983366}). Composing the Schr\"odinger representation with the
automorphism $ (x,\xi,t) \mapsto (|h|^{1/2}x, \epsilon |h|^{1/2} \xi, h t)$ of
$\Heis$, where $h \neq 0 $ and $\epsilon = \operatorname{sign}(h)=\pm1$, we
obtain the \emph{Schr\"odinger representation with parameters $h$}: for all
$\varphi \in L^2 (\R^g)$
\begin{equation}
  \label{eq:sect2:5}
  (\Sch_h (x, \xi , t) \varphi )(y) = 
  e^{i h t - i \epsilon |\hbar|^{1/2} \xi \cdot y - \tfrac{1}{2}i h \xi \cdot x}
  \varphi (y+ |\hbar|^{1/2}x). 
\end{equation}

According to the Stone-von Neumann theorem \cite{MR0030532}, the unitary
irreducible representations $\pi : \Heis\rightarrow U ( \mathcal H)$ of the
Heisenberg group on a Hilbert space $\mathcal H$ are
\begin{itemize}
\item either trivial on the center; then they are equivalent to a
  one-dimensional representation of the quotient group $\Center (\Heis) \backslash
  \Heis$, i.e.\ equivalent to a character of $\R^{2g}$
\item or infinite dimensional and unitarily equivalent to a Schr\"odinger
  representation with some parameter $h \neq 0$.
\end{itemize}

\paragraph{Infinitesimal Schr{\"o}dinger representation.}
The space of smooth vectors of the Schr{\"o}dinger representation $\Sch _h :
\Heis \to U (L^2(\R^g))$ is the space $\SB(\R^g) \subset L^2(\R^g)$ of
Schwartz functions (\cite{MR0209834}). By differentiating the Schr{\"o}dinger
representation $\Sch _h$ we obtain a representation of the Lie algebra $\hei$
on $\SB(\R^g)$ by essentially skew-adjoint operators on $L^2(\R^g)$; this
representation is called the \emph{infinitesimal Schr{\"o}dinger representation
  with parameter $h$}. With an obvious abuse of notation, we denote it by same
symbol $\Sch _h$; the action of $X\in \hei$ on a function $f$ will be denoted
$\Sch _h(X)f$ or $X.f$ when no ambiguity can arise.  Differentiating the
formulas~(\ref{eq:sect2:5}) we see that, for all $k=1,2,\dots , g$, we have
\begin{equation}
  \label{eq:sect2:6} \sch_h
  (X_k) =|h|^{1/2}\frac{\partial}{\partial x_k}, \qquad \sch_h (\Xi
  _k) =- i  \epsilon|h|^{1/2}\ x_k, \qquad \sch_h (T) = i h,
\end{equation}
where $(x_i)$ are the coordinates in $\R^g$ relative to the basis $(X_i)$ and
$\epsilon= \operatorname{sign}(h)$.  More generally, by the  Stone-von Neumann
theorem quoted above, given any Heisenberg basis $(X_i, \Xi_j,T)$ of $\hei$ the
formula above defines via the exponential maps a Schr{\"o}dinger representation
$\rho_h$ with parameter~$h$ on $L^2 (\R^g)$ such that:
\begin{gather*}
  \label{eq:sect2:4}
  \rho_h(e^{x_1X_1 + \dots + x_gX_g  }) f(y)=f(y +|\hbar|^{1/2}x),\\
  \rho_h(e^{\xi_1\Xi_1 + \dots + \xi_g\Xi_g }) f(y)=e^{-i\epsilon|h|^{1/2}
    \xi\cdot y }f(y), \quad \rho_h(e^{t T }) f(y)=e^{ith }f(y).
\end{gather*}

\subsection{Siegel symplectic geometry}

\paragraph{Symplectic group and moduli space.} 
Let $\Symp_{2g}(\R)$ be the group of symplectic automorphisms of the standard
symplectic space $(\R^{2g}, \omega)$. The group of automorphisms of $\Heis$
that are trivial on the center is the semi-direct product $\bigAuto =
\Symp_{2g}(\R) \ltimes \R^{2g}$ of the symplectic group with the group of inner
automorphisms $\Heis/ \Center (\Heis )\approx \R^{2g}$.

The group of automorphisms of $\Heis$ acts simply transitively on the set of
Heisenberg bases, hence we may identify the set of Heisenberg bases of $\hei$
with the group of automorphisms of $\Heis$. However since we are interested in
the action of subgroups defined in terms of a choice of a Heisenberg basis and
since the dynamical properties of such action are invariant under inner
automorphisms, we may restrict our attention to bases which are obtained
applying an automorphisms $\alpha \in \Symp_{2g}(\R)$ to the standard
Heisenberg basis.

Explicitly, the symplectic matrix written in block form $\alpha =
\left( \begin{smallmatrix} A &B \\ C & D \end{smallmatrix} \right) \in
\Symp_{2g}(\R)$, with the $g \times g $ real matrices $A,B,C$ and $D$
satisfying $ C^t A = \, A ^t C$, $ A^t D - \, C^t B= 1$ and $ D^t B= \, B^t D$,
acts as the automorphism
\[
(x, \xi , t) \mapsto \alpha (x,\xi , t) :=(Ax+ B \xi, C x + D \xi , t ) \, .
\]

\paragraph{Siegel symplectic geometry.}
The stabilizer of the standard lattice $ \Lattice< \Heis$ inside $\Auto$ is
exactly the group $\GMod$.  We call {\em moduli space} of the standard
Heisenberg manifold the quotient $\Mod = \Auto/\GMod $. 
We may regard $\Auto$ as the
\emph{deformation (or Teichm\"uller) space} of the standard Heisenberg manifold
$\M = \Heis/\Lattice$ and $\Mod$ as the moduli space of the standard
nilmanifold, in analogy with the $2$-torus case.

The \emph{Siegel
  modular variety}, the moduli space of principally polarized abelian varieties
of dimension $g$, is the double coset space  $\Sigma _g := \Uni _g\backslash\Auto/\GMod $, 
where   $\Uni _g$ is the maximal compact
subgroup  $\Symp_{2g}(\R) \cap \SO _{2g}(\R)$ of $\Symp_{2g}(\R) $, isomorphic to the
unitary group $U _g(\CC)$.
Thus, $\Mod$ fibers over $\Sigma _g$ with compact fibers~$ \Uni _g$. 

The quotient space  $\Uni _g\backslash\Auto/{\pm \mathbf 1_{2g}}$ may be identified to
Siegel upper half-space in the following way. Recall that
the {\em Siegel upper half-space} of degree/genus $g$ \cite{MR0164063} is the
complex manifold
\[
\Siegel _g := \{ Z \in \Sym _g (\CC) \, | \, \Im(Z) >0 \} \, 
\]
of symmetric complex $g \times g$ matrices $Z=X+iY$ with positive definite
symmetric imaginary part $\Im(Z)=Y$ and arbitrary (symmetric) real part $X$.

The symplectic group $\Symp_{2g} (\R)$ acts on the Siegel upper
half-space~$\Siegel _g $ as generalized M\" obius transformations. The left
action of the block matrix $ \alpha = \left(\begin{smallmatrix} A & B \\ C & D
    \\ \end{smallmatrix} \right) \in \Symp_{2g}(\R) $ is defined as
\begin{equation}
  \label{eq:sect2:7}
  Z \mapsto   \alpha (Z)   := (A  Z+B)(CZ+D)^{-1} \, .
\end{equation}
This action leaves invariant the Riemannian metric  $ds^2 = \tr
(dZ\, Y^{-1} d\overline{Z} \, Y^{-1} ) $.

As the the kernel of this action is given by $\pm \mathbf 1_{2g}$ 
and the stabilizer
of the point $i:= i \mathbf 1_{g} \in \Siegel _g$ coincides with $\Uni
_g$,  the map 
\[
\alpha \in \Auto \mapsto \alpha^{-1}(i)\in \Siegel _g
\]
induces an identification $ \Uni _g\backslash \Symp_{2g}(\R)/{\pm
  \mathbf 1_{2g}} \approx \Siegel _g $ and consequently an
identification of the Siegel
modular variety $\Sigma _g\approx \GMod\backslash \Siegel _g$. 

\begin{notation}
For  $\alpha\in\Auto $ we denote by  $[\alpha] := \alpha \, \Symp_{2g}
(\Z) $ its   projection on the moduli space $\Mod$.  We denote by  $\class{\alpha} :=
 \Uni _g\, \alpha \,\Symp_{2g}(\Z) $  
the projection of $\alpha$ to the Siegel
modular variety $\Sigma _g$. We remark that under the previous
identification $[[\alpha]]$ coincides with the point $\GMod
\, \alpha^{-1}(i)\in \GMod\backslash \Siegel _g $.
\end{notation}


\section{Cohomology with values in $\Heis$-modules}
\label{sec:cohom-with-valu-1}

Here we discuss the cohomology of the action of a subgroup $\Subgroup \subset
\Heis$ on a Fr{\'e}chet $\Heis$-module $F$, that is to say the Lie algebra
cohomology of $\mathfrak p=\operatorname{Lie}(\Subgroup)$ with values in the
$\Heis$-module $F$.  We assume that $\Subgroup$ is a connected Abelian Lie
subgroup of $\Heis$ contained in a Legendrian subgroup~$\mathsf L$.

The modules interesting for us are, in particular, those arising from the
regular representation of $\Heis$ on the space $C^\infty(\M)$ of smooth
functions on a (standard) nilmanifold $\M := {\Heis}\!/\Lattice$.
As mentioned in the introduction, the fact that $\Heis$ acts on $M$ by left
translations, implies that the space $F= C^\infty(\M)$ is a $\mathfrak
p$-module: in fact for all $V \in \mathfrak p$ and $f \in F$ one defines (cf.\
formula \eqref{eq:1})
\[
(V .f )(m) = \left.\frac{\D}{\D{}t}f (\exp (-t V). m ) \right|_{t=0}, \quad (m
\in M) \, .
\]
As $\Subgroup$ is an Abelian group, the differential on the cochain complex
$A^\ast (\mathfrak p , F)=\Lambda^*\mathfrak p\otimes F$ of $F$-valued
alternating forms on $\mathfrak p$ is given, in degree~$k$, by the usual
formula
\[
\D{}\omega(V_0, \dots, V_k)= \sum_{j=0}^k (-1)^{j} \, V_j.\omega(V_0, \dots,
\widehat{V_j}, \dots, V_k) \, .
\]
\begin{notation}
  When $F$ is the space of $C^\infty$-vectors of a representation $\pi$ of
  $\Heis$ we may denote the complex $A^\ast (\mathfrak p , F)$ also by the
  symbol $A^\ast (\mathfrak p , \pi^\infty)$.
\end{notation}

In order to study the cohomology of the complex $A^\ast (\mathfrak p ,
C^\infty(\M))$, it is convenient to observe that the projection $p$ of $M$ onto
the quotient torus~$\T^{2g}$ (see~\eqref{circle_fibration}) yields a
$\Heis$-invariant decomposition of all the interesting function spaces on $M$
into functions with zero average along the fibers of $p$ --- we denote such
function spaces with a suffix $0$ --- and functions that are constant along
such fibers; these latter functions can be thought of as pull-back of functions
defined on the quotient torus~$\T$; hence we write, for example,
\begin{equation}
  \label{eq:sect3:1}
  C^\infty(\M) =  C_0^\infty(\M) \oplus p^* (C^\infty(\T))
  \approx  C_0^\infty(\M) \oplus C^\infty(\T),
\end{equation}
and we have similar decompositions for $L^2(\M)$ and --- when a suitable
Laplacian is used to define them --- for the $L^2$-Sobolev spa\-ces~$W^s(\M)$.

If we denote by $\bar {\Subgroup}$ the projection of $\Subgroup$ into $\T^{2g}$
and by $\bar {\mathfrak p}$ its Lie algebra, we obtain that we may split the
complex ${A}^\ast (\mathfrak p , C^\infty(\M))$ into the sum of ${A}^\ast
(\mathfrak p , C_0^\infty(\M))$ and ${A}^\ast (\mathfrak p , p^*
(C^\infty(\T^{2g})))\approx {A}^\ast (\bar {\mathfrak p} , C^\infty(\T^{2g}))
$. The action of $\bar {\Subgroup}$ on $\T^{2g}$ being linear, the computation
of the cohomology of this latter complex is elementary and folklore when $\dim
\bar {\Subgroup} =1$. For lack of references we review it in the next
section~\ref{sect:3.1} for any $\dim \bar {\Subgroup}$. In
section~\ref{sect:3.2} we shall consider the cohomology of ${C}^\ast (\mathfrak
p , C_0^\infty(\M))$.

\begin{remark}
  \label{rem:sect3:1}
  To define the norm of the Hilbert Sobolev spaces $\Sobolev^s(\M)$, we fix a
  basis $(V_i)$ of the Lie algebra $\hei$, set $ \Lap= - \sum V_i^{2} $ and
  define $\| f\|_s^2 = \lin f , (1 + \Delta)^s f\rin $ where $ \lin \cdot ,
  \cdot \rin $ is the ordinary $L^2$ Hermitean product. This has the advantage
  that for any Hilbert sum decomposition $L^2(\M)= \bigoplus_i H_i$ of $L^2(\M)$
  into closed $\Heis$-invariant subspaces we also have a Hilbert sum
  decomposition $W^s(\M)= \bigoplus_i W^s(H_i)$ of $\Sobolev^s(\M)$ into closed
  $\Heis$-invariant subspaces $\Sobolev^s(H_i):=\Sobolev^s(\M)\cap H_i$.
\end{remark}

\paragraph{Currents.}
Let $F$ be any tame Fr\'echet $\hei$-module, graded by increasing norms $(\|
\cdot \|_s)_{s\ge 0}$, defining Banach spaces $W^s \subset F$.

The space of continuous linear functionals on $A ^k (\mathfrak p , F)= \Lambda
^k \mathfrak p \otimes F$ will be called \emph{the space of currents of
  dimension $k$} and will be denoted $A _ k (\mathfrak p , F') $ where~$F'$ is
the strong dual of $F$; the notation is justified by the fact that the natual
pairing $(\Lambda _k \mathfrak p, \Lambda ^k \mathfrak p)$ between $k$
vectors and $k$-forms allows us to write
$A _ k (\mathfrak p , F') \approx \Lambda ^k \mathfrak p \otimes F'$.  Endowed
with the strong topology,  $A _ k (\mathfrak p , F') $ is the inductive limit of
the spaces $\Lambda ^k \mathfrak p \otimes (W^s)'$.
  
The \emph{boundary operators} $\partial: A _ k (\mathfrak p , F') \to A _ {k-1}
(\mathfrak p , F') $ are, as usual, the adjoint of the differentials $\D$, hence
they are defined by $\lin \partial T , \omega \rin = \lin T, \D \omega \rin$.  A
\emph{closed} current~$T$ is one such that $\partial T = 0$.  We denote by $Z _
k (\mathfrak p,F')$ the space of closed currents of dimension $k$ and by $Z _ k
(\mathfrak p,(W^s)')$ the space of closed currents with coefficients
in~$(W^s)'$.

\subsection{Cohomology of a linear $\R^d$ action on a torus}
\label{sect:3.1}
Let $\Latt$ be a lattice subgroup of $\R^\ell$ and let $\R^\ell$ acts on the
torus $\Torus^\ell =\R^\ell/\Latt$ by translations. We consider the
restriction of this action to a subgroup $\mathsf Q< \R^\ell$ isomorphic to
$\R^d$, with Lie algebra  $\mathfrak q $. Then the Fr\'echet space $C^\infty(\Torus^\ell )$ is a
$\mathfrak q$-module. In this section we consider the cohomology of
the associated complex $A^\ast (\mathfrak q,C^\infty(\Torus ^\ell
))$.

Let  $\Latt^\perp = \{ \lambda \in (\R^\ell)' \, | \, \lambda \cdot n =0 \, \, \forall n \in \Latt \} $ denotes the dual lattice of
$\Latt$.
We say that \emph{the subspace $\mathfrak q $ satisfies a Diophantine condition
  of exponent $\tau>0$} with respect to the lattice $\Latt$, and we write
$\mathfrak q \in \hbox{DC}_\tau(\Latt)$, if
\begin{equation}
  \label{sect3:eq:2}
  \exists\, C>0\quad \text{such that} \quad  \sup_{V \in \mathfrak q \setminus \{0\}}
  \frac{| \lambda \cdot V | }{\Vert  V \Vert} \ge C
  \Vert\lambda\Vert^{-\tau} ,\quad 
  \forall \lambda \in \Latt^\perp \setminus \{0\}.
\end{equation}
We set
\[
\mu(\mathfrak q, \Latt) = \inf \{\tau\colon \mathfrak q \in
\hbox{DC}_\tau(\Latt)\} \, . 
\]

\begin{remark}
  The Diophantine condition considered here is dual to the Diophantine
  condition on subspaces of $(\R^\ell)'\approx \R^\ell$ considered by Moser
  in~\cite{MR1069487}. In fact, if we set $\mathfrak q^\perp = \{\lambda \in
  (\R^\ell)' \colon \ker\lambda \supset\mathfrak q \} $,  the condition
  \eqref{sect3:eq:2} is equivalent to
  \[
  \exists\, C>0\quad \text{such that} \quad \dist (\lambda, \mathfrak q^\perp)
  \ge C \|\lambda\|^{-\tau} ,\qquad \forall \lambda \in \Latt^\perp \setminus
  \{0\}.
  \]
  Thus, by Theorem~2.1 of~\cite{MR1069487}, the inequalities \eqref{sect3:eq:2}
  are possible only if $\tau \ge \ell/d -1$,  and the set of subspaces $\mathfrak
  q^\perp$ with $\mu(\mathfrak q, \Latt)=\ell/d -1 $ has full Lebesgue
  measure in the Grassmannian $\operatorname{Gr}(\R^d;\R^\ell)$.
\end{remark}

We say that $\mathfrak q$ is \emph{resonant (w.r. to $\Latt$)} if, 
for some $\lambda \in \Latt ^\perp \setminus \{0\}$, we have $\mathfrak q
\subset \ker \lambda$; in this case the closure of the orbits of $\mathsf Q$ on
$\R^\ell/\Latt$ are contained in lower dimensional tori, the orbits of the
rational subspace $\ker \lambda$, and we may understand this case by
considering a lower dimensional ambient space $\R^{\ell'}$ with $\ell'<\ell$.

Thus we may limit ourselves to non-resonant $\mathfrak q$; in this case, if
$\mathfrak q$ is not Diophantine, we have $\mu(\mathfrak q, \Latt)=+\infty$
and we say that $\mathfrak q$ is \emph{Liouvillean (w.r.\ to $\Latt$)}.

\begin{theorem}[Folklore]
  \label{thm:sect3:1}
  Let $\mathfrak q\in \operatorname{Gr}(\R^d;\R^\ell)$ be a non-resonant subspace 
  with respect to the lattice $\Latt < \R^\ell$. Then the action of $\mathsf
  Q=\exp \mathfrak q$ on the torus $\Torus ^\ell:=\R^\ell/\Latt$ is
  cohomologically $C^\infty$-stable if and only if $\mathfrak q\in
  \hbox{DC}_\tau(\Latt)$ for some $\tau >0$.  In this case we have
  \[
  H^*(\mathfrak q, C^\infty(\Torus ^\ell)) \approx \Latt ^*\mathfrak q \, , 
  \]
  the cohomology classes being represented by forms with constant coefficients.
  Furthermore,  the  $\mathfrak q$-module $C^\infty(\Torus ^\ell)$ is tamely
    cohomologically $C^\infty$-stable  and  has tame splitting in all degrees.
\end{theorem}

\begin{proof}
  Without loss of generality we may assume $\Latt = \Z^\ell$. 
 The $s$-Sobolev norm of a function $f  \in C^{\infty}(\Torus ^\ell) $ with  
  Fourier series representation $f(x) = \sum _{n \in \Z^{\ell}} \hat f (n) \,e^{2\pi i n \cdot x }$ is given  by
  \[
  \| f \| _{s} ^2= \sum _{n \in \Z^{\ell} } \left( 1+ \| n \| ^2 \right)^{s} \, | \hat f (n) |
  ^2 \, .
  \]
  We have a direct sum decomposition $C^\infty(\Torus ^\ell) =
  \CC\lin 1 \rin \oplus C^\infty_0(\Torus ^\ell)$ ,  where $ \CC\lin 1
  \rin$ is the space of constant funtions and $C^\infty_0(\Torus ^\ell)$ 
  is the space of zero mean smooth functions on $\Torus ^\ell$. 
  An analogous \emph{orthogonal} decomposition 
  $W^s(\Torus ^\ell)= \CC\lin 1 \rin \oplus W_0^s(\Torus ^\ell) $ holds for
  Sobolev spaces. 
 Hence every
$\omega\in Z^k(\mathfrak q,C^\infty(\Torus ^\ell))$ splits (tamely) into a sum $
\omega =\omega_0+\omega_{c}$ of a form  $\omega_0\in Z^k(\mathfrak
q,C_0^\infty(\Torus ^\ell))$ and a constant coefficient form
$\omega_c\in \Lambda^k \mathfrak q$.  Consequently,  the cohomology
$H^*(\mathfrak q, C^\infty(\Torus ^\ell))$ splits into the sum of
 cohomology classes  represented by forms
with constant coefficients and $H^*(\mathfrak q, C_0^\infty(\Torus ^\ell))$.
We now show that, under the assumption  (\ref{sect3:eq:2}) on $\mathfrak
q$, we have  $H^*(\mathfrak q, C_0^\infty(\Torus ^\ell))=0$.

By Fourier analysis, $C^\infty_0(\Torus ^\ell)$ splits into a
  $L^2$-orthogonal sum of one-di\-men\-sio\-nal modules $\CC_n\approx \CC$, $n
  \in \Z^\ell\setminus\{0\}$; the space $\mathfrak q$ acts on $\CC_n$ by
  \[
  V . \, z = i \, ( n \cdot V ) \, z, \qquad \forall \, z \in \CC_n ,\ \forall  \, V \in
  \mathfrak q;
  \]
  hence,  for $\omega \in \Lambda^k \mathfrak q\otimes \CC_n$ and  $V_0, \dots,
  V_k\in \mathfrak q$ ,
  \[
  \D{}\omega(V_0, \dots, V_k)= \sum_{j=0}^k i \, ( n \cdot  V_j)  \, \omega(V_0,
  \dots,\widehat{V}_j, \dots, V_k) \, . 
  \]
  Let $X_1,X_2,\dots , X_d $ be a basis of $\mathfrak q$, and define the
  co-differential $\D{}^*$ by
  \[
  \D{}^* \eta (V_1, \dots, V_k) := -\sum_{m=1}^d i \, (  n \cdot X_m ) \,  \eta
  (X_m , V_1, \dots, V_k).
  \]
  We have $H=\D{}^*\circ\D{}+ \D{}\circ\D{}^* = \left(\sum_{m=1}^d | n \cdot 
    X_m |^2 \right)\operatorname{Id_{\Lambda^* \mathfrak q}}$. It follows
  that if $\omega \in \Lambda^k \mathfrak q\otimes \CC_n$ is closed then
  $\omega=\D{}\Omega$ with
  \[
  \Omega = H^{-1}\D{}^* \omega.
  \]
  We conclude that the map $\D_{-1}:=H^{-1}\D{}^*$ is a right 
  inverse
  of $\D$ on the space  $Z^k(\mathfrak q,\CC_n)$ of closed forms.
  From the definitions of the maps $\D^*$ and $H$ we obtain the estimate
  \[
  \Vert  \D_{-1}  \omega \Vert_0 \le \Big(\sum_{m=1}^d |n \cdot  X_m
   |^2\Big)^{-\frac 1 2}\Vert \omega \Vert_0, \quad \forall \omega \in
  Z^k(\mathfrak q,\CC_n).
  \]
  It is easily
  seen that the Diophantine condition (\ref{sect3:eq:2}) is
  equivalent to the existence of a constant $C>0$ such that $ \sum_{m=1}^d |
  n \cdot X_m |^2> C \,\Vert n\Vert^{-2\tau}$ for all $n \in \Z^\ell$. Hence,
  for some constant $C>0$ we have $ \Vert  \D_{-1}  \omega \Vert_{0} \le C^{-1} \Vert n
  \Vert ^{\tau}\Vert \omega \Vert_0 $, and therefore 
  \[
  \Vert  \D_{-1}  \omega  \Vert_{s} \le C^{-1} \Vert \omega \Vert_{s+\tau}
  \]
  for all $s\in \R$ and all $\omega \in Z^k(\mathfrak q,\CC_n)$.

  Since the Sobolev space $( W^s_0(\Torus ^\ell), \|\cdot\|_s) $ is
  equal to the Hilbert direct sum $ \bigoplus_{n\neq 0} ( \CC_n,
  \|\cdot\|_s) $, the map $\D{}_{-1}$ extends to a tame map
\[
  \D_{-1}\colon  Z^k(\mathfrak q,C_0^\infty(\Torus ^\ell)) \to A^{k-1}(\mathfrak q,C_0^\infty(\Torus ^\ell)) .
\]
satisfying a tame estimate of degree $\tau$ with base $0$ and
associating a primitive to each closed form. 

Combining these results with the previous remark on constant
coefficient forms, we conclude that under the Diophantine assumption
(\ref{sect3:eq:2}) the $\mathfrak q$-module $C^{\infty}(\Torus ^\ell)$ 
is  tamely cohomologically $C^\infty$-stable and has a tame
splitting in all degrees.

The ``only if'' part of the statement may be proved as in the case
$\dim \mathsf Q=1$ (see Katok~\cite[page 71]{MR2008435}).
\end{proof}

\subsection{Cohomology with values in $C^\infty_0(M)$}
\label{sect:3.2}

The previous section settles the study of the cohomology of the action of a
abelian subgroup $\Subgroup \subset \Heis$ with values in the
$\Heis$-sub-module $p^* (C^\infty(\Torus))$. We are left to consider the action
$\Subgroup$ with values in the $\Heis$-sub-module $C^\infty_0(\M)$.

Since the center $Z(\Heis)$ has spectrum $2\pi\Z\setminus \{0\}$ on $L_0^2(\M)$,
the space $L_0^2(\M)$ splits as a Hilbert sum of Schr{\"o}\-din\-ger
$\Heis$-modules $H_i$ equivalent to $\rho^{h}$, with $h \in 2\pi\Z\setminus
\{0\}$. The same remark applies to the the Sobolev space $\Sobolev_0^s(\M)$, 
which splits as a Hilbert sum of the (non-unitary) $\Heis$-modules
$\Sobolev_0^s(H_i) = H_i \cap \Sobolev_0^s(\M)$.

The space $C^\infty(\M)\cap H_i$ can be characterized as the space
$C^\infty(H_i)$ of $C^\infty$ vectors in the $\Heis$-module $H_i$; it is a tame
graded Fr\'echet space topologized and graded by the increasing family of
Sobolev norms.  This leads us to consider the action of $\Subgroup$ with values
in the space of smooth vectors of a Schr{\"o}dinger $\Heis$-module.


Thus let $\Subgroup$ be an \isotropic subgroup of $\Heis$ of
dimension~$d$. Fix a Legendrian subgroup $\mathsf L$ such that $\Subgroup <
\mathsf L < \Heis$.  Let $\rho_h$ be a Schr{\"o}dinger representation, with
$|h| > h_0 >0$,

Since the group of automorphisms of $\Heis$ acts transitively on Heisenberg
bases, we may assume that we have fixed a Heisenberg basis $(X_i,\Xi_j,T)$ of
$\hei$ such that $(X_1, \dots ,X_d)$ forms a basis of $\mathfrak p$ and $(X_1,
\dots ,X_g)$ is a basis of $ \operatorname{Lie}(\mathsf L)$. This yields
isomorphisms $\mathsf L\approx \R^g$ and $\Subgroup \approx \R^d$,  with the
latter group embedded in~$\R^g$ via the first $d$~coordinates.  With these
assumptions, the formulas yielding the representation $\rho_h$ on $L^2(\R^g)$
are given by the equations~(\ref{eq:sect2:5}) and the space~$\rho_h^\infty$ of
$C^\infty$ vectors for the representation $\rho_h$ is identified with $\mathcal
S(\R^g)$ on which $\hei$ acts by the formulas~(\ref{eq:sect2:6}).

\paragraph{Homogeneous Sobolev norms.} 
The infinitesimal representation extends to a representation of the enveloping
algebra $\mathfrak U(\hei)$ of $\hei$; this allows us to define the
``sub-Laplacian'' as the image via $\rho_h$ of the element
\[ \subLap_g = -( X_1^2 + \dots + X_g^2+\Xi_1^2 + \dots + \Xi_g^2)\in \mathfrak
U(\hei).
\]
Formulas~(\ref{eq:sect2:6}) yield
\begin{equation}
  \label{eq:sect3:3}
  \rho_h(\subLap_g)= | h | \, \left( |x|^2 - \sum_{k=1}^g \frac{\partial ^2 }{ \partial x_k^2}   \right)  = | h |  \, \rho_1(\subLap_g) \, . 
\end{equation}
Since $ \subLap_g$ is a positive operator with (discrete) spectrum bounded
below by $g | h | $, we define the space $W^s(\rho_h, \R^g)$ of functions of
Sobolev order $s$ as the Hilbert space of vectors $\varphi$ of finite
\emph{homogeneous} Sobolev norm
\begin{equation}
  \label{eq:sect3:3b}
  \norm \varphi \norm ^2_{s,h} :=\lin( \rho_h(\subLap_g))^s \varphi, \varphi\rin
\end{equation}

This makes explicit the fact that the space~$\rho_h^\infty$ of $C^\infty$ vectors
for the representation~$\rho_h$ coincides with $\mathcal S(\R^g)$.

The homogeneous Sobolev norms~\eqref{eq:sect3:3b} are not the standard ones
(later on we shall make a comparison with standard Sobolev norms). They have
however the advantage that the norm on $W^s(\rho_h, \R^g)$ is obtained by
rescaling by the factor $| h|^{s/2}$ the norm on $W^s(\rho_1, \R^g)$. For this
reason we can limit ourselves to study the case $h =1$; later we shall consider
the appropriate rescaling. Thus we denote $\rho=\rho_1$ and, to simplify, we
write $\subLap_g$ for $\rho(\subLap_g)$ and $W^s (\R^g$) for $W^s
(\rho_1,\R^g)$: also we set
\[
\norm \varphi \norm_s := \norm \varphi \norm_{s,1} = \| \subLap_g^{s/2} \varphi
\| _0 \, .
\]

\paragraph{The cochain complex $A^*(\mathfrak p, \rho^\infty)$.}

It will be convenient to use the identification $\R^g\approx\R^d \times \R^{g
  -d} $ and, accordingly, to write $\varphi(x,y)$, with $x\in \R^d$ and $y\in
\R^{g -d}$, for a function $\varphi$ defined on~$\R^g$.  We also write $\D{}x=
\D{}x_1\cdots\D{}x_d$.  Then, by the formula~(\ref{eq:sect2:5}), the group
element $q\in \Subgroup\approx \R^d$ acts on $\varphi\in \SB (\R^g)$ according
to
\[
\varphi ( x,y ) \mapsto \varphi (x+ q,y).
\]
Thus the complex $A^*(\mathfrak p, \rho^\infty)$ is identified with the complex
of differential forms on $\mathfrak p\approx\R^d$ with coefficients in
$\mathcal S(\R^g)$.  It will be also convenient to define the operators
$\subLap_d '= \left( |x|^2 - \sum_{k=1}^g \frac{\partial ^2 }{ \partial x_k^2}
\right)$ on $\SB (\R^d)$ and $\subLap_{g-d} ''= \left( |y|^2 - \sum_{k=1}^{g-d}
  \frac{\partial ^2 }{ \partial y_k^2} \right)$ on $\SB (\R^{g-d})$; they may
be also considered as operators on $\SB (\R^g)$, and then $\subLap_g =
\subLap_d '+ \subLap _{g-d} ''$  .

\begin{lemma}
  \label{lem:sect3:1}
  Consider $\mathcal S(\R^g)$ as a $\mathsf H^g$-module with parameter~$h=1$.
  Define the distribution $ \Int_g \in \SB '(\R^g)$ by
  \[
  \Int_g(f) := \int_{\R^g} f(x)\,\D{}x
  \]
  for $f\in \mathcal S(\R^g)$. Then, for any $s>g/2$, $\Int_g$ extends to a bounded
  linear functional on $W^s(\R^g)$, that is $\Int_g\in \Sobolev^{-s}(\R^g)$.
\end{lemma}
\begin{proof} Using Cauchy-Schwartz inequality we have
  \[
  | \Int_g (f) |^2 \le \int_{\R^g} |(g + |x|^2)|^{-s}\,\D{}x \cdot
  \int_{\R^g}(g + |x|^2)^{s} |f(x)|^2\,\D{}x
  \]
  As $g + |x|^2 \le 2 \subLap_g$, the second integral is bounded by a constant
  times $\norm f \norm _s^2$, and the result follows.
\end{proof}

For the next lemma we adopt the convention that $\R^0=\{0\}$ and $\mathcal
S(\R^{0})=W^s(\R^{0})=\mathbb C$ with the usual norm.

\begin{lemma}
  \label{lem:sect3:22}
  For $1\le d \le g$, consider the map $\Int_{d,g} :\mathcal S(\R^g) \mapsto
  \mathcal S(\R^{g-d})$ defined by
  \[
  ( \Int_{d,g} f) (x,y) := \int_{\R^d} f(x,y) \, \D{}x
  \]
  We consider $\mathcal S(\R^g)$ and $\mathcal S(\R^{g-d})$ as $\mathsf H^g$
  and $\mathsf H^{g-d}$-modules, respectively, with parameter~$h =1$. Then, for
  any $\varepsilon > 0$ and $s\ge 0$, the map $\Int_{d,g}$ extends to a bounded
  linear map from $W^{s+d/2+\varepsilon}( \R^{g})$ to $W^{s}( \R^{g-d})$, i.e.
  \[
  \norm \Int_{d,g} f \norm _s \leq C \, \norm f \norm _{s+d/2+\varepsilon}
  \]
  for some constant $C=C(s,\varepsilon, d,g)$.  In particular this proves that
  $\Int_{d,g} (\mathcal S(\R^g))\subset \mathcal S(\R^{g-d})$.
\end{lemma}
\begin{proof}
  
  For $d=g$ we have $\Int_{g,g} =\Int_{g}$ and the result is a restating of the
  previous lemma.

  Now suppose $d<g$. The operators $\subLap_d'$ and $\subLap_{g-d}''$,
  considered as operators on $L^2(\R^d)$ and $L^2(\R^{g-d})$, have discrete
  spectrum (they are independent $d$-dimensional and $(g-d)$-dimensional
  harmonic oscillators); thus identifying $L^2(\R^{g})\approx L^2(\R^d)\otimes
  L^2(\R^{g-d})$ their joint spectral measure on $L^2(\R^{g})$ is the product
  of the spectral measures on $L^2(\R^d)$ and $L^2(\R^{g-d})$
  respectively. Clearly $\subLap_g \ge \subLap_d'$ and $\subLap_g \ge
  \subLap_{g-d}''$.

  Let $(v_m)$ and $(w_n)$ be orthonormal bases of $L^2(\R^d)$ and
  $L^2(\R^{g-d})$ of eigenevectors of $\subLap_d'$ and $\subLap_{g-d}''$ with
  eigenvalues $(\lambda_m)$ and $(\mu_n)$, respectively. We may choose these
  bases so that $\{v_m\}\subset \mathcal S(\R^d) $ and $\{w_n\}\subset\mathcal
  S(\R^{g-d})$.

  Writing for $f \in \mathcal S(\R^{g})$ , $f= \sum f_{mn} v_m \otimes w_n$ and
  letting $d_m= \Int_d(v_m)$ we have $ \Int_{d,g} f = \sum_n (\sum_m( d_m)
  f_{mn}) w_n$. It follows that
  \[
  \norm \Int_{d,g} \, f \norm ^2_s = \sum_n \mu_n^s \Big|\sum_m d_m
  f_{mn}\Big|^2 \le \Big(\sum_m |d_m|^2 \lambda_m^{-d/2
    -\varepsilon}\Big)\Big(\sum_{m,n} \mu_n^s \lambda_m^{d/2
    +\varepsilon}|f_{mn}|^2\Big)
  \]
  The first term in this product equals $\| \Int_d \|_{-(d/2+ \varepsilon)}^2$,
  which is bounded by Lem\-ma~\ref{lem:sect3:1}; the second term is majorated
  by $\norm f \norm ^2_{s+d/2+ \varepsilon}$, since $\subLap_g \ge \subLap_d'$
  and $\subLap_g \ge \subLap_{g-d}''$,.
\end{proof}

The proof of the following corollary is immediate.
\begin{corollary}
  \label{coro:sect3:2}
  We use the notation of the previous Lemma. Suppose $d<g$. For all $t\ge 0$ 
  and all $s> t+d/2$ the map
  \[
  D\in \Sobolev ^{-t}(\R^{g-d})\mapsto D \circ \Int_{d,g} \in \Sobolev
  ^{-s}(\R^{g})
  \]
  is continuous. In particular, if $f\in \Sobolev ^{s}(\R^g)$ with $s>d/2$ then
  $\Int_{d,g}(f)=0$ if and only if $T \circ \Int_{d,g}(f) =0$ for all $T\in
  \Sobolev ^{0}(\R^{g-d})$.
\end{corollary}
Let $\varphi_{d}\in \mathcal S(\R^d)$ be the ground state of $\subLap_d$
normalized by the condition $\Int_d(\varphi_{d})=1$, namely
\[ \varphi_{d}(x) := (2\pi) ^{-d/2} e^{-|x|^2/2}, \qquad(x \in \R^d);
\]
we have $\norm \varphi_{d} \norm _{s} = \pi^{-d/4} d ^{s/2} $.

\begin{lemma}
  \label{lem:sect3:2}
  For $1\le d<g$, let $\Egauss_{d,g} :\mathcal S(\R^{g-d}) \mapsto \mathcal
  S(\R^g) $ be defined by
  \[
  (\Egauss_{d,g} \, f) (x,y) := \varphi_{d}(x)f(y)
  \]
  We consider $\mathcal S(\R^g)$ and $\mathcal S(\R^{g-d})$ as $\mathsf H^g$
  and $\mathsf H^{g-d}$-modules, respectively, with parameter $h =1$.  Then,
  for any $s \ge 0$, the map $\Egauss_{d,g}$ extends to a bounded linear map
  from $W^{s}( \R^{g-d})$ to $W^{s}(\R^{g})$, i.e.
  \[
  \norm\Egauss_{d,g} \, f \norm _s \le C \, \norm f \norm _{s}.
  \]
  for some constant $C=C(s,d)$.
\end{lemma}
\begin{proof}
  Consider $\subLap_d '=$ and $\subLap_{g-d} ''$ as operators on $\SB
  (\R^g)$. For all integers $n$, from the binomial identity for $(\subLap_d
  '+\subLap_{g-d}'')^n $, we obtain
  $ \norm \Egauss_{d,g} \, f \norm_n^2 = \sum_j \tbinom{n}{j} \,\norm
  \varphi_d\norm_j^2 \norm f\norm_{n-j}^2 \le 2^{n} \norm \varphi_d^2\norm_n
  \norm f\norm_{n}^2$, where for the last inequality we used $\subLap_d '\ge 1$
  and $\subLap_{g-d}''\ge 1$. This proves the lemma for integer~$s$; the
  general claim follows by interpolation.
\end{proof}

\begin{lemma}
  \label{estimate_P}
  Let $d=1 $. Let $f$ be an element of the $\mathsf H^g$-module $\mathcal
  S(\R^g)$  with parameter~$h=1$. Suppose  that $\Int_{1,g} \, f =0$. Set
  \[
  (\Prim f)(x,y) := \int_{-\infty}^{x} f(t , y) \, \D{}t.
  \]
  For all $t\ge 0$ and all $\varepsilon >0$ there exists a constant $C=C(t
  ,\varepsilon)$ such that
  \begin{equation} \label{formula_estimate_P} \norm \Prim f \norm_{t} \le C\,
    \norm f \norm_{t+1+\varepsilon} \, .
  \end{equation}
  In particular this proves that $\Prim (\mathcal S(\R^g))\subset \mathcal
  S(\R^{g-d})$.
\end{lemma}
\begin{proof}
  When $g=1$ this is a variation on the statement of Lemma~6.1
  in~\cite{MR2218767}, which can be easily proved by use of the Cauchy-Schwartz
  inequality as in Lemma~\ref{lem:sect3:1}.

  Suppose now that $g>1$ and consider $\subLap _1 '$ and $\subLap_{g-1}''
  $. The condition $\Int_{1,g} \, f =0$ implies that $\Int_{1,g}
  \,(\subLap_{g-1}'')^w f =0$ for any $w\ge 0$; furthermore $\Prim
  (\subLap_{g-1}'')^w f= (\subLap_{g-1}'')^w \Prim f$.  Using the result for
  the case $g=1$ and the definition of the norm $\|\cdot \|_0$ we have for all
  $t\ge 0$ and all $\varepsilon >0$
  \[
  \| (\subLap_1)^{t/2} \, (\subLap_{g-1}'')^{w/2} \, \Prim f\|_0 \le C(t
  ,\varepsilon) \| (\subLap_1')^{(t+1+\varepsilon)/2} \, (\subLap_{g-1}
  '')^{w/2} \, f \|_0.
  \]
  For integer values of the Sobolev order, using the above inequality and the
  binomial formula, we may write, for any $\varepsilon >0$ and $n\in \N$,
  \begin{eqnarray*}
    \norm \Prim f \norm_{n}^2 
    & = & {\langle   \Prim f ,  \subLap_g ^n  \Prim f \rangle}_0=\sum_{k=0}^n {n \choose k}\| (\subLap_1')^{k/2}  \,  (\subLap_{g-1}'')^{(n-k)/2} \, \Prim f \|_0^2\\
    & \le & C(\varepsilon,n) \sum_{k=0}^n {n \choose k} \| (\subLap_1')^{(k+1+\varepsilon)/2} \, (\subLap_{g-1}'') ^{(n-k)/2} \, f \| _0^2 \\
    & \leq  & C(\varepsilon,n) \,\|  (\subLap_g)^{n/2}  \, (\subLap_1 ')^{(1+\varepsilon)/2} \, f \| _0^2\\
    & \le & C(\varepsilon,n)\, \|  \subLap_g^{(n+1+\varepsilon)/2} \, f \| _0^2= C(\varepsilon,n)\,\norm f \norm_{n+1+\varepsilon}^2. 
  \end{eqnarray*}

  The general inequality follows by interpolation of the family of norms $
  \norm \cdot \norm_{n} $.
\end{proof}

\paragraph{Sobolev cocycles and coboundaries.} 
Having fixed an Euclidean product on $\hei $, we obtain, by restriction, an
Euclidean product on $\mathfrak p \subset \hei$ and, by duality and extension
to the exterior algebra, an Euclidean product on $\Lambda ^k \mathfrak p'$.
The spaces $A^k(\mathfrak p,\rho^\infty) \approx \Lambda ^k \mathfrak p'
\otimes \SB (\R^g)$ of cochains of degree $k$ are endowed with the Hermitian
products obtained as tensor product of the Euclidean product on $\Lambda
^k\mathfrak p'$ and the Hermitian products $\| \cdot \|_s$ or $\norm \cdot
\norm_s$ on $\SB (\R^g)$. Completing with respect to these norms, we define the
Sobolev spaces $\Lambda ^k \mathfrak p' \otimes \Sobolev^s (\R^g)$ of cochains
of degree $k$, and use the same notations for the norms.

It is clear that, for $k<d$, the cohomology groups are $H^{k} ( \mathfrak p, \SB
(\R^g) )=0$.  Here we estimate the Sobolev norm of a primitive $\Omega \in
A^{k-1} ( \mathfrak p, \SB (\R^g) )$ of a coboundary $\omega = \D\Omega\in B^k
( \mathfrak p, \SB (\R^g))= Z^k ( \mathfrak p, \SB (\R^g))$ in terms of the
Sobolev norm of $\omega$.

\begin{proposition}
  \label{prop:sect3:1}
  Let $s\ge 0$ and $1\le k<d\le g$.  Consider $\SB (\R^g)$ as a $\Heis$-module
  with parameter $h=1$. For every $\varepsilon >0$ there exists a constant
  $C=C(s,\varepsilon,g,d)>0$ and a linear map
  \[
  \D{}_{-1}\colon Z^k ( \mathfrak p, \SB (\R^g)) \to A^{k-1}(\mathfrak p, \SB
  (\R^g))
  \]
  associating to every $ \omega \in Z^k ( \mathfrak p, \SB (\R^g))$ a primitive
  $\Omega = \D{}_{-1}\omega \in A^{k-1}(\mathfrak p, \SB (\R^g))$ satisfying
  the estimate
  \begin{equation}
    \label{prop_cocycles}
    \norm \Omega \norm_s \le C  \,  \norm \omega \norm_{s +(k+1)/2 + \varepsilon}\,.
  \end{equation}
\end{proposition}

\begin{proof}
  We denote points of $\R^{g} \approx \mathfrak p\times \R^{g-d} \approx\R
  \times \R^{d-1} \times \R^{g-d}$ as triples $(t, x,y)$ with $t\in \R$, $x \in
  \R^{d-1}$ and $y \in \R^{g-d}$.  For $0\leq k \leq d \leq g$, one defines
  linear maps
  \[
  A^{k}(\R^{d},\SB (\R^{g})) \overset{\xrightarrow{\quad \Int
      \quad}}{\xleftarrow[\quad \Egauss \quad ]{}} A^{k-1}(\R^{d-1},\SB
  (\R^{g-1}))
  \]
  as follows.  For a monomial $\omega=f(t,x,y) \,dt \wedge \D{} x^a \in
  A^{k}(\R^{d},\SB (\R^{g}))$, where $a$ a multi-index in the set $\{1,2,
  \dots,d-1\}$, we define
  \begin{equation}
    \label{I}
    \Int   \omega := \left(\int_{-\infty}^{\infty} f(t,
      x,y)\, \D{}t\right)\D{} x^a  = (\Int_{1,g} f) \, \D{} x^a \, ;
  \end{equation}
  if $\D t $ does not divide $\omega$ we define instead $\Int \omega=0$.  For a
  monomial $\omega = f (x,y)\, \D{} x^a\in A^{k-1}(\R^{d-1},\SB (\R^{g-1}))$,
  we define
  \begin{equation}
    \label{E}
    \Egauss \, \omega : = \varphi (t ) f
    (x,y) \,\D{}t \wedge\D{} x^a = (\Egauss_{1,g}f) \, \D{}t \wedge\D{} x^a \, . 
  \end{equation}
  By Lemma \ref{lem:sect3:22} we obtain that for any $t\ge 0$ and
  $\varepsilon>0$ we have:
  \begin{equation}
    \label{estimate_I}
    \norm \Int \omega \norm_{t} \le C  \, \norm
    \omega\norm_{t+1/2+\varepsilon} ,\qquad C=C(t,\varepsilon,g).
  \end{equation}
  It follows from this inequality that the image of $\Int$ lies in $
  A^{k-1}(\R^{d-1}, \SB (\R^{g-1}))$.  For the map $\Egauss$ the inclusion $
  \Egauss (A^{k-1}(\R^{d-1},\SB (\R^{g-1})))\subset A^{k}(\R^{d},\SB (\R^{g}))$
  is obvious, and by Lemma~\ref{lem:sect3:2} we have, for any $s \geq 0$,
  \begin{equation}
    \label{estimate_E}
    \norm \Egauss \eta \norm_{s} \le C  \,\norm \eta\norm_{s} ,\qquad C=C(s,d).
  \end{equation}
  From \eqref{estimate_I} and \eqref{estimate_E} it follows that, for any $s
  \geq 0$,
  \begin{equation}
    \label{estimate_EI}
    \norm \Egauss \Int    \omega \norm_{s} \le  C 
    \,  \norm \omega\norm_{s+1/2+\varepsilon} \, . 
  \end{equation}
  The maps $\Int$ and $\Egauss$ commute with the differential~$\D$.  It is well
  known that $\Int$ and $\Egauss$ are homotopy inverse of each other. In fact,
  it is clear that $\Int \Egauss $ is the identity.
  
  We claim that the usual homotopy operator
  \[
  \Khom \colon A^{k}(\R^{d}, \SB (\R^{g}) ) \to A^{k-1}(\R^{d} , \SB (\R^{g}) )
  \]
  satisfying $1 - \Egauss \Int = \D \Khom - \Khom \D$ also satisfies tame
  estimates.  Indeed, for a monomial $\omega$ not divisible by $ \D{}t$,
  $\Khom$ is defined as $\Khom \omega =0$; for a monomial $\omega =f(t,x,y) \,
  \D{}t\wedge \D{} x ^a$ it is defined as $\Khom \omega =g(t, x,y) \, \D{} x
  ^a$ where
  \begin{equation} \label{K} g(t,x,y) = \int_{-\infty}^{t} \left[ f(r , x,y) -
      \varphi(r) \left(\smallint_{\R}f(u, x,y)\,\D{}u\right) \right]\,\D{}r =
    \Prim(f - \Egauss_{1,g} \, \Int_{1,g} f) \, .
  \end{equation}
  Then by Lemma~\ref{estimate_P} and~\eqref{estimate_EI} we have that for all
  $s\ge 0$:
  \begin{equation}
    \label{estimate_K}
    \norm \Khom \omega \norm _{s}
    \leq C(s, \varepsilon, g,d) \, \norm \omega \norm _{s+3/2 +\varepsilon}\,,
  \end{equation}
  unless $\Int \omega =0$, in which case we have
  \begin{equation}
    \label{estimatebis_K}
    \norm \Khom  \omega \norm _{s} \leq C(s, \varepsilon, g,d) \, \norm \omega \norm _{s+1 +\varepsilon} \, . 
  \end{equation}
  This prove the claim.
  
  Let $ \omega \in A^{1}(\R^{d},\SB (\R^{g}))$ be closed and $1 < d \le
  g$. Then $\Int \omega =0$ (by homotopying the integral in \eqref{I} with an
  integral with $ x\to \infty$) and therefore $\Omega = \Khom \omega \in
  A^{0}(\R^{d},\SB (\R^{g})) \approx \SB (\R^{g+1})$ is a primitive of
  $\omega$, i.e. $\D {} \Omega = \omega$, and by \eqref{estimatebis_K} it
  satisfies the estimate $\norm \Omega \norm _{s} \leq C(s) \cdot \norm \omega
  \norm _{s+1 +\varepsilon}$ for all $s> 1/2$. Thus the proposition is proved
  in this case.

  Assume, by recurrence, that the Proposition is true for all $g\ge 1$, all
  $d\le g$ and all $k \le \min \{n,d\}-1$.  Let $ \omega \in A^{n}(\R^{d}, \SB
  (\R^{g}))$, with $n < d$, be closed.  Then the $(n-1)$-form $\Int \omega \in
  A^{n-1}(\R^{d-1},\SB (\R^{g-1}))$ is also closed.  By recurrence, $\Int
  \omega = \D \eta$ for a primitive $\eta \in A^{n-2}(\R^{d-1}, \SB
  (\R^{g-1}))$ satisfying the estimate
  \begin{equation}\label{recurrence_eta}
    \norm \eta \norm_s \le C   \,  \norm \Int \omega  \norm_{s +n/2 + \varepsilon}\,.
  \end{equation}
  Since $\Egauss \Int \omega = \Egauss \D \eta$ and $\Egauss$ commutes with
  $\D$, we obtain that a primitive of $\omega$ is given by $\D{}_{-1}\omega
  :=\Omega := \Khom \omega + \Egauss \eta$. Therefore, from lemma
  \ref{lem:sect3:22} and the estimates \eqref{estimate_I}, \eqref{estimate_E},
  \eqref{estimate_K} and \eqref{recurrence_eta}, we have, for some constants
  $C$'s which only depend on $s \geq 0$ and $\varepsilon >0$,  
  \begin{equation}
    \label{eq:sect3:7}
    \begin{split}
      \norm \Omega \norm _{s}
      &\le \norm \Khom \omega \norm _{s} + \norm  \Egauss \eta \norm _{s}\\
      & \le C' \, \norm \omega \norm _{s+3/2 +\varepsilon}
      + C'' \, \norm \eta \norm _{s}\\
      & \le C' \, \norm \omega \norm _{s+3/2 +\varepsilon}
      + C'''  \norm \Int  \omega  \norm _{s+n/2+\varepsilon/2}\\
      & \le C' \,\norm \omega \norm _{s+3/2 +\varepsilon}
      + C''''  \norm \omega  \norm _{s+n/2+1/2+\varepsilon}\\
      & \le C \,\norm \omega \norm_{s+(n+1)/2+\varepsilon} .
    \end{split}
  \end{equation}
  Thus the estimate~\eqref{prop_cocycles} holds also for $k=n$.  This concludes
  the proof.
\end{proof}


We are left to consider the space $H^{k} ( \mathfrak p, \SB (\R^g) )$ when
$k=d:=\dim \mathfrak p$.

The map $\Int_{d,g}$ extends to a map
\begin{equation} \label{Idg} 
\Int_{d,g}\colon A^{d} ( \mathfrak p, \SB (\R^g) ) \to \SB (\R^{g-d})
\end{equation} 
by setting for a form $\omega = f(x,y) \, \D{}x_1\wedge \dots \wedge \D{}x_d$
\[
(\Int_{d,g} \, \omega ) (y) := \int_{\R^d} f(x,y) \, \D{}x \, . 
\]

\begin{proposition}
  \label{prop_cohom_eq_max_dim}
  Let $s \ge 0$ and $1 \le d \le g$.  Consider $\SB (\R^g)$ as a $\Heis$-module
  with parameter $h=1$ and let $\omega\in A^{d} ( \mathfrak p, \SB (\R^g)
  )$. The form $\omega$ is exact if and only if $\Int_{d,g}\omega =0$.
  Furthermore, for every $\varepsilon >0$ there exists a constant
  $C=C(s,\varepsilon,g,d)>0$ and a linear map
  \[
  \D{}_{-1}\colon \ker \Int_{d,g}\subset A^{d} ( \mathfrak p, \SB (\R^g) ) \to
  A^{d-1}(\mathfrak p, \SB (\R^g))
  \]
  associating to every $ \omega \in \ker \Int_{d,g} $ a primitive $\Omega$ of
  $\omega$ satisfying the estimate
  \begin{equation}
    \label{eq:sect3:4}
    \norm \Omega \norm_s \le C  \,  \norm \omega
    \norm_{s + (d+1)/2 + \varepsilon}\,.
  \end{equation}
\end{proposition}

\begin{proof} 
  The ``only if'' part of the statement is obvious. For $d=1$ and any $g\ge 1$,
  this is Lemma \ref{estimate_P}.  Indeed, a primitive of the $1$-form $\omega
  = f(x,y) \, \D {}x$ is the $0$-form $\Omega := (\Prim f)(x,y)$, and the
  estimate for the norms comes from \eqref{formula_estimate_P}.
  
  Assume, by recurrence, that the Proposition is true for all $g'< g$ and all
  $d \leq g'$.  Let $\omega \in A^{d}(\R^{d}, \SB (\R^{g}))$ be a $d$-form such
  that $\Int_{d,g} \omega =0$.  Consider $\Int \omega \in A^{d-1}(\R^{d-1}, \SB
  (\R^{g-1}))$, where $\Int$ is the operator defined in the previous proof (see
  \eqref{I}). It is clear from the definitions that $\Int_{d,g}( \omega )=0$
  implies $\Int_{d-1,g-1} \Int \omega=0$.  By recurrence, $\Int \omega = \D {}
  \eta $ for a primitive $\eta \in A^{k-1}(\R^{k}, \SB (\R^{g}))$ satisfying
  the estimate
  \begin{equation} \label{recurrence_eta_max} \norm \eta \norm _s \leq C \,
    \norm \Int \omega \norm _{s+d/2+\varepsilon }
  \end{equation}
  As in the previous proof, one verifies that the form $\D{}_{-1}\omega:=\Omega
  := \Khom \omega + \Egauss \eta \in A^{d-1}(\R^{d}, \SB (\R^{g}))$ is a
  primitive of $\omega$ (where the operators $\Egauss$ and $\Khom$ are defined
  in previous proof, see \eqref{E} and \eqref{K}).  Therefore, from Lemma
  \ref{lem:sect3:22} and the estimates \eqref{estimate_I}, \eqref{estimate_E},
  \eqref{estimate_K} and \eqref{recurrence_eta_max}, we have, for some
  constants $C$'s which only depend on $s \geq 0$ and $\varepsilon >0$, 
  \begin{equation}
    \label{eq:sect3:7b}
    \begin{split}
      \norm \Omega \norm _{s}
      &\le \norm \Khom \omega \norm _{s} + \norm  \Egauss \eta \norm _{s}\\
      & \le C' \, \norm \omega \norm _{s+3/2 +\varepsilon}
      + C'' \, \norm \eta \norm _{s}\\
      & \le C' \,\norm \omega \norm _{s+3/2 +\varepsilon}
      + C'''  \norm \Int  \omega  \norm _{s+d/2+\varepsilon/2}\\
      & \le C' \,\norm \omega \norm _{s+3/2 +\varepsilon}
      + C''''  \norm \omega  \norm _{s+d/2+1/2+\varepsilon}\\
      & \le C \,\norm \omega \norm_{s+(d+1)/2+\varepsilon} .
    \end{split}
  \end{equation}
  The proof is complete.
\end{proof}

\begin{proposition}
  \label{prop:tameness}
  Let $s \ge 0$ and $1 \le d \le g$.  Consider $\SB (\R^g)$ as a $\Heis$-module
  with parameter $h=1$. For any $k=0, \dots, d$, the space of coboundaries
  $B^d(\mathfrak p, \SB (\R^g))$ is a tame direct summand of $A^{k}(\mathfrak
  p, \SB (\R^g))$. In fact, there exist linear maps
  \[
  M^k: A^{k}(\mathfrak p, \SB (\R^g)) \to B^k ( \mathfrak p, \SB (\R^g))
  \]
  satisfying the following properties:
  \begin{itemize}
  \item the restriction of $M^k$ to $B^k ( \mathfrak p, \SB (\R^g))$ is the
    identity map;
  \item the map $M^k$ satisfies, for any $\varepsilon>0$, tame estimates of
    degree $(k+3)/2 + \varepsilon$  if $k<d$ and of degree $d/2 + \varepsilon$
    if $k=d$.
  \end{itemize}
\end{proposition}
\begin{proof}
  For $\omega= f \D{}x^1 \wedge \dots \wedge \D{}x^d\in A^d(\mathfrak p,
  \SB(\R^g))$ let
  \[
  M^d(\omega) = \omega - ( \Egauss_{d,g} \circ\Int_{d,g} f)\,\D{}x^1 \wedge
  \dots \wedge \D{}x^d .
  \] The Lemmata~\ref{lem:sect3:22} and \ref{lem:sect3:2} show that $M^d$ is a
  linear tame map of degree $d/2 + \varepsilon$, for every
  $\varepsilon>0$. Clearly for $\omega\in B^d(\mathfrak p, \SB(\R^g))$ we have
  $M^d(\omega) = \omega$. Since the map~$M^d$ maps $ A^d(\mathfrak p,
  \SB(\R^g))$ into $B^d(\mathfrak p, \SB(\R^g))$, we have proved that
  $B^d(\mathfrak p, \SB(\R^g))$ is a direct summand of $ A^d(\mathfrak p,
  \SB(\R^g))$.

  Now consider the case $k<d$. We have $B^k(\mathfrak p,
  \SB(\R^g))=Z^k(\mathfrak p, \SB(\R^g))$.  For $\omega \in A^k(\mathfrak p,
  \SB(\R^g))$ let
  \[
  M^k(\omega) = \omega - \D{}_{-1}\circ \D{}(\omega).
  \]
  The map $M^k$ is a linear tame map of degree $(k+3)/2 + \varepsilon$, for
  every $\varepsilon>0$.

  Clearly for $\omega\in Z^k(\mathfrak p, \SB(\R^g))$ we have $M(\omega) =
  \omega$. Furthermore $\D{}\circ M =0$.  Thus the map~$M^k$ sends $
  A^k(\mathfrak p, \SB(\R^g))$ into $Z^k(\mathfrak p, \SB(\R^g))$.  We have
  proved that $Z^d(\mathfrak p, \SB(\R^g))$ is a direct summand of $
  A^d(\mathfrak p, \SB(\R^g))$.
\end{proof}

\paragraph{$\Subgroup$-invariant currents of dimension $\dim \Subgroup$.}

  

Recall that the space of currents of dimension~$k$ is the space $A_k( \mathfrak p,
\SB (\R^g))$ of continuous linear functionals on $A^k( \mathfrak p, \SB
(\R^g))$ and that $A_k( \mathfrak p, \SB (\R^g))$ is identified with $ \Lambda
^k \mathfrak p \otimes \SB' (\R^g)$. For any $s\ge 0$, the space $\Lambda ^k
\mathfrak p \otimes\Sobolev^{-s}(\R^g)$ is identified with the  space of currents of
dimension~$k$ and Sobolev order $s$.

It is clear, from Lemma \ref{lem:sect3:1}, that $\Int_g=\Int _{g,g} \in
\Sobolev ^{-s}(\R^g)$ for any $s > g/2$, i.e. it is a closed current of
dimension $g$ and Sobolev order $g/2 + \varepsilon$, for any $\varepsilon >0$.

For $d < g$ and $t>0$, consider the currents $D \circ \Int_{d,g}$ with $D \in
\Sobolev^{-t} (\R^{g-d})$.  It follows from Lemma \ref{lem:sect3:22} that such
currents belong to $\Lambda ^d \frak p \otimes \Sobolev ^{-s}(\R^g)$ for any $s
> t+d/2$ and it is easily seen that they are closed.

In fact, we have the following proposition, whose proof follows immediately from
Lemma~\ref{lem:sect3:22} and Proposition~\ref{prop_cohom_eq_max_dim}.

\begin{proposition}\label{prop:sect:3:12}
  For any $s> \dim \Subgroup/2$,  the space of $\Subgroup$-invariant currents of
  dimension~$d:=\dim \Subgroup$ and order $s$ is a closed subspace of $
  \Lambda^k \mathfrak p \otimes \Sobolev^{-s}(\R^g)$ and it concides with the
  space of closed currents of dimension~$d$. It is
  \begin{itemize}
  \item a one dimensional space spanned by $\Int_g$, if $\dim \Subgroup=g$;
  \item an infinite-dimensional space generated by
    \[
    I_d( \mathfrak p, \SB(\R^g))=\{ D \circ \Int_{d,g}\mid D\in
    L^{2}(\R^{g-d})'\} .
    \]
    if $\dim \Subgroup <g$. We have $I_d( \mathfrak p, \SB(\R^g)) \subset
    W^{-d/2 -\varepsilon}(\R^g)$, for all $\varepsilon >0 $.
  \end{itemize}
  Let $\omega\in \Lambda ^d \mathfrak p' \otimes \Sobolev ^{s}(\R^g)$ with
  $s>(d+1)/2$. Then $\omega$ admits a primitive~$\Omega$ if and only if $T
  (\omega) =0$ for all $T\in I_d( \mathfrak p, \SB(\R^g))$; under this
  hypothesis we may have $\Omega\in \Lambda ^{d-1} \mathfrak p' \otimes
  \Sobolev ^{t}(\R^g)$ for any $t <s-(d+1)/2$.
\end{proposition}



\paragraph{Bounds uniform in the parameter $h$.} Here we observe that the
estimates in Propositions \ref{prop:sect3:1} and \ref{prop_cohom_eq_max_dim}
are uniform in the Planck constant $h$, provided that this constant is bounded
away from zero.

\begin{proposition}
  \label{prop_cohom_eq_uniform_h}
  Let $s \geq 0$ and $1 \le k \leq d \le g$, and consider the  
  $\Heis$-module $\SB (\R^g)$ with parameter $h$ such that $|h| \geq h_0 >0$.
  Let $B^k= Z^{k}(\R^{d}, \SB (\R^g))$ if $k<d$ and $B^d=\ker\Int_{d,g}$ if
  $k=d $.  For every $\varepsilon >0$ there exists a constant
  $C=C(s,\varepsilon,g,d, h_0)>0$ and a linear map
  \[
  \D{}_{-1}\colon B^k \to A^{k-1}(\mathfrak p, \SB (\R^g))
  \]
  associating to every $ \omega \in B$ a primitive $\Omega = \D{}_{-1}\omega
  \in A^{k-1}(\mathfrak p, \SB (\R^g))$ satisfying the estimate
  \begin{equation}
    \label{eq:sect3:4b}
    \norm \Omega \norm_s \le C  \,  \norm \omega
    \norm_{s + (k+1)/2 + \varepsilon}\,.
  \end{equation}
  Furthermore, for any $\varepsilon>0$ there exists a constant
  $C'=C'(s,\varepsilon,g,d, h_0)>0$ such that the splitting linear maps of
  Proposition~\ref{prop:tameness}
  \[
  M^k: A^{k}(\mathfrak p, \SB (\R^g)) \to B^k ( \mathfrak p, \SB (\R^g))
  \]
  satisfy tame estimates
  \[
  \norm M^k (\omega) \norm_s \le C'\,\norm \omega \norm_{s+ w}
  \]
  where $w=(k+3)/2+ \varepsilon$, if $k<d$, and $w=d/2 + \varepsilon$ if $k=d$.
\end{proposition}
\begin{proof}
  From \eqref{eq:sect2:6} we see that the boundary operators in the Schr\"
  odinger representation with Planck constant $h$ are $^\hbar \! \D := \rho ^h
  (\D) = |h|^{1/2} \, \D$. Therefore, if $\omega = \D \, \Omega$, then $ \omega
  = {}^\hbar \!  \D \, \Omega '$ with $\Omega '= |h|^{-1/2} \, \Omega
  $. Consequently, by \eqref{eq:sect3:3}, the estimates \eqref{prop_cocycles}
  and \eqref{eq:sect3:4} imply
  \begin{equation}
    \label{eq:sect3:7t}
    \begin{split}
      \norm \Omega ' \norm _{s,h} & = | h | ^{-1/2} \norm \Omega \norm _{s,h} =
      | h | ^{s/2-1/2}
      \norm \Omega \norm _s \\
      &\leq C \,| h | ^{s/2-1/2}  \norm \omega \norm _{s+ (k+1)/2 +\varepsilon} \\
      & =  C \, | h | ^{-(k+1+\varepsilon)/2} \, \norm \omega \norm _{s+(k+1)/2+\varepsilon,h} \\
      & \leq C' \, \norm \omega \norm _{s+t+\varepsilon,h} .
    \end{split}
  \end{equation}
  for some $C'$ depending also on $h_0$. The second statement is proved in an
  analogous manner.
\end{proof}

\paragraph{Comparison with the usual Sobolev norms.} 
Standard Sobolev norms associated with a Heisenberg basis $( X_i, \Xi_j,T)$ of
$\hei$ were defined in Remark~\ref{rem:sect3:1}. For a $\Heis$-module $\SB
(\R^g)$ with parameter $h$, the image of the Laplacian $ \ -( X_1^2 + \dots +
X_g^2+\Xi_1^2 + \dots + \Xi_g^2 + T^2)\in \mathfrak U(\hei)$ under $\rho_h$ is
$\Delta_g = \subLap_g + h^2$. Thus
\[
\| f\|_s ^2 = \langle f , (1 + \Delta_g)^s f \rangle= \langle f , (1 + h^2
+H_g)^s f \rangle
\]
Here we claim that the uniform bound as in Proposition
\ref{prop_cohom_eq_uniform_h} continues to hold with respect to the usual
Sobolev norms. This is a consequence of the following easy lemma which applies
to $\SB (\R^g)$ but also to any tensor product of $\SB (\R^g)$ with some finite
dimesional Euclidean space.
  
\begin{lemma}
  \label{lem:prop_cohom_eq_uniform_h_Sobolev}
  Let $L\colon \SB (\R^g) \to \SB (\R^g)$ be a linear map satisfying, for some
  $t\ge 0$ and every $s\geq 0$, the estimate
  \[
  \norm L(f) \norm _{s} \le C(s) \norm f \norm _{s+t}
  \]
  Then for every $s\geq 0$ we have
  \[
  \| L(f) \| _{s} \le C_1(s) \| f \| _{s+t},
  \]
  where $C_1(s) = \max_{u\in [0,s+1]} C(u)$.
\end{lemma}

\begin{proof} 
  For integer $s=n$, using the binomial formula, we get, with $C'(n):=
  \max_{j\in [0,n]} C(j)^2$,
  \[
  \label{eq:Laplacian_binomial}
  \begin{split}
    \| L(f)\| _{n}^2 & :=\lin   L(f), (H_g +1+h^2)^{n} \, L(f)  \rin _0 \\
    & = \sum_{j=0}^n {n \choose j}\|  (1+h^{2})^{(n-j)/2} H_g^{j/2} L(f)\|_0^2 \\
    & \leq C'(n)  \sum_{j=0}^n {n \choose k}\|  (1+h^{2})^{(n-j)/2} H_g^{(j+t)/2}  f \|_0^2 \\
    & = C'(n)\,  \| (1+\Delta_g) ^{n} H_g^{t/2}  f  \| _0^2 \\
    & \leq C'(n) \, \| f \| _{n+t}^2.
  \end{split}
  \]
  For non integer $s$ the lemma follows by interpolation.
\end{proof}


\subsection{Proofs of Theorems \ref{thm:sect1:2}  and \ref{thm:sect1:1}}
We are now in a position to integrate over Schr\" odinger
representations, and obtain our main result on the cohomology of
$\Subgroup< \Heis$ with values in Fr\'echet $\Heis$-modules.

\begin{theorem}
  \label{thm:sect3:2}
  Let $\Subgroup$ be a $d$-dimensional \isotropic subgroup of $\Heis$, and let
  $F^{\infty}$ be the Fr\'echet space of $C^{\infty}$-vectors of a
  unitary $\Heis$-module~$F$. Let $F=\int F_\alpha \,d\alpha$ be the
  direct integral decomposition of $F$ into irreducible sub-modules.
  Suppose that
  \begin{enumerate}
  \item $F$ does not contain any one-dimensional sub-modules;
  \item A generator of the center $\Center(\Heis)$ acting on $F$ has a spectral gap.
  \end{enumerate}
  Then the reduced and the ordinary cohomology of the complex $A^{*} (\mathfrak
  p, F^{\infty})$ coincide.  In fact, for all $k=1,\dots,d$, there are linear
  maps
  \[\D{}_{-1}\colon B^k(\mathfrak p, F^\infty)\to
  A^{k-1}(\mathfrak p, F^\infty)\] associating to each $\omega\in B^k(\mathfrak
  p, F^\infty)$ a primitive of $\omega$ and satisfying tame estimates of degree
  $(k+1)/2+\varepsilon$ for any $\varepsilon >0$.

  We have $H^{k} (\mathfrak p, F^{\infty})=0 $ for $k<d$; in degree $d$, we
  have that $H^{d} (\mathfrak p, F^{\infty}) $ is finite dimensional only if
  $d=g$ and the measure $d\alpha$ has finite support.

  For any $k=0, \dots, d$ and any $\varepsilon>0$, there exist a constant $C$
  and a linear map
  \[
  M^k: A^{k}(\mathfrak p, F^{\infty}) \to B^k ( \mathfrak p, F^{\infty})
  \]
  such that the restriction of $M^k$ to $B^k ( \mathfrak p,F^\infty)$ is the
  identity map and the following estimate holds:
  \[
  \| M^k \omega\|_s \le C \|\omega\|_{s+w}, \quad \forall \omega \in
  A^{k}(\mathfrak p, F^\infty)
  \]
  where $w=(k+3)/2 + \varepsilon$, if $k<d$ and $w=d/2 + \varepsilon$ if
  $k=d$. Hence the space of coboundaries $B^k(\mathfrak p, F^\infty)$ is a
  tame direct summand of $A^{k}(\mathfrak p, F^\infty)$.
\end{theorem}
 
(The hypotheses 1 and 2 of the above theorem could be stated more briefly by
saying that $F$ satisfies the following property: any non-trivial unitary
$\Heis$-module weakly contained in $F$ is infinite dimensional).

\begin{proof}
Let  $F^{\infty}$ the Fr\'echet space of $C^{\infty}$-vectors of a
unitary $\Heis$-module~$(\rho, F)$. Let $F=\int F_\alpha \D{}\alpha$ be the
direct integral decomposition of $F$ into irreducible
sub-modules~$(\rho_\alpha, F_\alpha)$ . The hypothesis of
Theorem~\ref{thm:sect3:2} imply that there exists $h_0>0$ such that for almost
every $\alpha$ the $\Heis$-module $F_\alpha$ is unitarily equivalent to a
Schr\"odinger module with parameter $h$ satisfying $|h|\geq h_0$.

For any $s\in\R$, we also have a decomposition of the Sobolev spaces $W^s(F,
\rho)$ as direct integrals $\int W^s(F_\alpha,\rho_\alpha) \D{}\alpha$; this
is due to the fact that we defined the Sobolev norms via the operator
$1+\Delta_g$, which is an element of the enveloping algebra $\mathfrak U(\hei)$, 
and the $\mathfrak U(\hei)$-invariance of the spaces $F_\alpha$.  It follows
that any form $\omega\in A^k(\mathfrak p,F^{\infty})$ has a decomposition
$\omega=\int \omega_\alpha \,\D{}\alpha$ with $\omega\in A^k(\mathfrak
p,F_\alpha^{\infty})$ and
\begin{equation}
  \label{eq:sob_norm_decomp}
  \|\omega\|^2_{W^s(F, \rho)}= \int \|
  \omega_\alpha \|^2_{W^s(F_\alpha, \rho_\alpha)} \,d\alpha.
\end{equation}
For the same reason mentioned above, we have
\begin{equation}
  \label{eq:passing-d}
  \D{}\omega=\int  (\D{}\omega_\alpha)\, \D{}\alpha
\end{equation}
Hence $\omega$ is closed if and only if $\omega_\alpha$ is closed for almost
all $\alpha$, i.e. $Z^k(\mathfrak p, W^s(F,\rho))= \int Z^k(\mathfrak p,
W^s(F_\alpha,\rho_\alpha))\, d\alpha$.

For $k<d$ we set $B_\alpha^k= Z^k(\mathfrak p,F^{\infty}_\alpha )$. 
For $k=d$ we set $B_\alpha^d=\ker
I_{d,g,\alpha}$,  where 
 $I_{d,g,\alpha}\colon A^d(\mathfrak
p,F^{\infty}_\alpha) \to \SB(\R^{g-d})$ are the tame maps defined, 
for each $\alpha$, as   in
\eqref{Idg}. 

By Proposition~\ref{prop_cohom_eq_uniform_h} and
Lemma~\ref{lem:prop_cohom_eq_uniform_h_Sobolev}, we have a constant
$C=C(s,\varepsilon,g,d,h_0)$ and, for each $\alpha$, a linear map
\[
\D{}_{-1,\alpha} : B_\alpha^k  \to A^{k-1}(\mathfrak p, F_\alpha^\infty)
\]
associating to each $\omega\in B_\alpha^k (\mathfrak p, F_\alpha^\infty)$ a primitive 
$\Omega = \D{}_{-1} \omega$ of $\omega$ satisfying
the estimates
\begin{equation}
  \label{eq:tame_estimate_alpha}
  \|\D{}_{-1,\alpha}\omega\|_{W^s(F_\alpha,\rho_\alpha)} \le C\,
  \|\omega\|_{W^{s+(k+1)/2+\varepsilon}(F_\alpha,\rho_\alpha)}.
\end{equation}

Let $B^k$ be the graded Fr\'echet subspace of $A^k(\mathfrak p,F^{\infty})$
defined as $\int B_\alpha^k \,d\alpha$. Clearly for $k<d$ we have
$B^k=Z^k(\mathfrak p, F^\infty)$ and, in degree $d$, we have $B^d\supset
B^d(\mathfrak p, F^\infty)$. 

The above estimate shows that it is possible to define a linear map $\D{}_{-1}
:B^k \to A^{k-1}(\mathfrak p, F^\infty)$, by setting, for $\omega=\int \omega_\alpha \,\D{}\alpha\in B^k$,
\[\D{}_{-1} \omega:= \int
\D{}_{-1,\alpha}\omega_\alpha \, d\alpha.
\]
By~\eqref{eq:sob_norm_decomp} and \eqref{eq:passing-d}, 
the estimates~\ref{eq:tame_estimate_alpha} are
still true if we replace $\D{}_{-1,\alpha}$ by $\D{}_{-1}$. 

This shows that  $\D{}_{-1}$  is a tame map
of degree $(k+1)/2+\varepsilon$, for all $\varepsilon>0$ associating to each
$\omega\in B^k$ a primitive of $\omega$.

Thus $H^k(\mathfrak p,F^{\infty})=0$ if $k<d$. For $k=d$, we have
$H^d(\mathfrak p,F^{\infty})= \int H^d(\mathfrak p,F_\alpha^{\infty})\,
d\alpha$. By Proposition~\ref{prop_cohom_eq_max_dim}, we have $H^d(\mathfrak
p,F_\alpha^{\infty})\approx \SB(\R^{g-d})$, hence the top degree cohomology is
infinite dimensional if $d<g$, and one-dimensional if $d=g$. This shows that
$H^d(\mathfrak p,F^{\infty})$ is finite dimensional if and only if $d=g$ and
the measure $d\alpha$ has finite support.

Finally for each $\alpha$, we have tame maps $M_\alpha^k$ given by
Proposition~\ref{prop:tameness}. Setting $M^k = \int M_\alpha^k \,d\alpha$ we
obtain the maps $M^k$ satisfying the conclusion of the Theorem.  
\end{proof}

\paragraph{Proof of theorem \ref{thm:sect1:2}} The proof is immediate
as the space $F=L_0^2(\M)$ formed by the $L^2$  functions on $\M$ of
average zero along the fibers of the central fibration of $\M$ satisfy
the hypothesis of the theorem above. In fact $L_0^2(\M)$ is a direct
sum of irreducible representations of $\Heis$ on which the generator
$Z$ of the center $\Center (\Heis)$ acts as scalar multiplication by
$2\pi n$, with $n\in \Z\setminus\{0\}$. 

\paragraph{Proof of theorem \ref{thm:sect1:1}} The theorem follows
from the theorem above and the ``folklore'' theorem \ref{thm:sect3:1}, 
as explained at the beginning of Section~\ref{sec:cohom-with-valu-1}.

\section{Sobolev structures and best Sobolev constant}
\label{section:harmonic}

\subsection{Sobolev bundles}

\paragraph{Sobolev spaces.}
The group $\Auto< \operatorname{Aut}(\Heis)\approx
\operatorname{Aut}(\hei)$ acts (on the right) on the enveloping
algebra $\mathfrak U \left( \hei \right)$ in the following way: we
identify $\mathfrak U \left( \hei \right)$ with the algebra of right
invariant differential operator on $ \Heis$; if $V \in \mathfrak U
\left( \hei \right)$ and $\alpha\in \Auto$, the action of $\alpha$ on
$V$ yields the differential operator $V_\alpha$ defined by
\begin{equation}
  \label{eq:sect4:1}
  V_\alpha(f):=\alpha^*V\big((\alpha^{-1})^*f\big), \qquad f\in C^\infty (\Heis).
\end{equation}
Let $\Lap = -( X_1^2 + \dots + X_g^2+\Xi_1^2 + \dots + \Xi_g^2 +
T^2)\in \mathfrak U(\hei)$ denote the Laplacian on $\Heis$ defined via
the ``standard'' basis $(X_i,\Xi_j,T)$
(cf.~sect.~\ref{par:heis_not}). Then $\Lap_\alpha=-((\alpha^{-1}
X_1)^2 + \dots + (\alpha^{-1}\Xi_g)^2 + T^2)$, i.e.\ $\Lap_\alpha$ is
the Laplacian on $\Heis$ defined by the basis $(\alpha^{-1}
(X_i),\alpha^{-1} (\Xi_j),T)$.

Let $\Lattice ' $ be any lattice of $\Heis$ and $\M ' :=
\Heis/\Lattice '$ the corresponding nilmanifold. For each $\alpha\in
\Auto$, the operator $\Lap_\alpha$ is an elliptic, positive and
essentially self-adjoint ope\-ra\-tor on $L^2 (\M')$. Recall that
$L_0^2(\M')$ denotes the space of ell-two functions on $\M'$ with
zero average along the fibers of the toral projection. Its norm is
defined via the ell-two Hermitian product $\langle \cdot ,
\cdot\rangle$ with integration done with respect to the normalised
Haar measure. Setting $L_\alpha = 1 + \Lap_\alpha$ we define the
Sobolev spaces
\begin{equation}
  \label{eq:sect4:4}
\Sobolev ^{s}_\alpha (\M') := L_\alpha^{-s/2} L_0^2(\M'),
\end{equation}
which are Hilbert spaces equipped with the inner product
\[
\langle f_1 , f_2 \rangle _ {s,\alpha} := \langle L_\alpha ^{s/2} f_1
, L_\alpha ^{s/2} f_2 \rangle = \langle f_1 , L_\alpha ^{s} f_2
\rangle.
\]
For simplicity, we denote by $\Sobolev ^{s}(\M')$ the Sobolev spaces
defined via the operator $1+\Lap$.  The space $\Sobolev ^{-s}_\alpha
(\M')$ is canonically isomorphic to the dual Hilbert space of
$\Sobolev ^{s}_\alpha (\M')$.

\begin{remark}
  \label{rem:sect4:1}
  It is useful to notice that, since the Laplacian $\Delta$ is
  invariant under the above action of the maximal compact subgroup
  $\Uni_g$ of $\Auto$, the Sobolev space $\Sobolev ^{-s}_\alpha (\M')$
  depends only on the class $\mathsf \Uni_g \alpha \in \Siegel_g$ in
  the Siegel  upper half-space.
\end{remark}

Let $\Lattice $ be the standard lattice of $\Heis$ and $\M := \Heis/
\Lattice$.  For $\alpha \in \Auto$, let $\Lattice _\alpha :=
\alpha(\Lattice)$ and $\M_\alpha := \Heis/ \Lattice _\alpha$ the
corresponding nilmanifold. The automorphim~$\alpha$ induces a
diffeomorphism (denoted with the same symbol) according to the formula
\[
\alpha : \M \to \M_\alpha, \qquad h\Lattice \mapsto \alpha ( h )
\Lattice_\alpha , \quad\forall h \in \Heis \,.
\]
It is immediate that the pull-back map $\alpha^*: C^\infty
(\M_\alpha)\to C^\infty (\M)$ satisfies
\[
\alpha^*( \Lap f ) = \Lap_\alpha (\alpha^* f ), \quad f \in C^\infty
(\M_\alpha );
\]
since $\alpha^*$ preserves the volume, we obtain an isometry
\[
\alpha^*: \Sobolev ^{s} (\M_\alpha) \to \Sobolev ^{s}_\alpha (\M).
\]

Observe that, as topological vector spaces, the spaces $\Sobolev
^s_\alpha (\M)$, ($\alpha \in \Auto$), are all isomorphic to $\Sobolev
^s (\M)$.  Only their Hilbert structure varies as $\alpha$ ranges in
$\Auto$. In fact we have the following lemma, whose proof is omitted.
\begin{lemma}
  \label{lem:sect4:1}
  For every $R>0$ there exists a constant $C(s)>0$ such that for all
  $\alpha , \beta \in \Auto$ with $\dist (\alpha , \beta)<R$ we have
  \[
  \| \varphi \| _{s, \alpha} \leq C(s) \,(1+ \dist (\alpha , \beta) ^2
  )^{|s|/2} \cdot \| \varphi \| _{s, \beta}.\] Here, $\dist (\cdot ,
  \cdot)$ is some left-invariant distance on $\Auto $.
\end{lemma}

\begin{lemma}
  \label{lem:sect4:2}
  Let $s\ge 0$. For $\gamma\in \GMod$ and $\alpha \in \Auto$, the
  pull-back map $\gamma^*$ is an isometry of $\Sobolev
  _{\alpha}^s(\M)$ onto $\Sobolev _{\alpha\gamma}^s (\M) $. Hence
  $\gamma_* : \Sobolev _{\alpha\gamma}^{-s} (\M)\to \Sobolev
  _{\alpha}^{-s}(\M)$ is an isometry.
\end{lemma}

\begin{proof}
  By the above, we have isometries $(\alpha\gamma)^*
  :\Sobolev^s(\M_{\alpha\gamma}) \to \Sobolev _{\alpha\gamma}^s (\M) $
  and $\alpha^* :\Sobolev ^s (\M _{\alpha}) \to \Sobolev _{\alpha}^s
  (\M) $. However,  $\M_{\alpha\gamma}= \M_{\alpha}$, since
  $\Lattice_{\alpha\gamma}= \Lattice_{\alpha}$. It follows that
  $\gamma^*=(\alpha\gamma)^* (\alpha^*)^{-1}$ is an isometry of
  $\Sobolev _{\alpha}^s(\M)$ onto $\Sobolev _{\alpha\gamma}^s (\M) $.
\end{proof}

\paragraph{The Sobolev bundle over the moduli space and its dual.}

For $s\ge 0$, let us consider $ \Sobolev ^{s} (\M)$ as a topological
vector space. The group $\GMod$ acts on the right on the trivial
bundles $\Auto \times \Sobolev ^{s} (\M)\to \Auto$ according to
\[
(\alpha , \varphi ) \mapsto (\alpha , \varphi )\gamma:=(\alpha \gamma
, \gamma^*\varphi) \qquad\gamma\in \GMod, \quad (\alpha , \varphi )\in
\Auto \times \Sobolev ^{s} (\M)\] By Lemma~\ref{lem:sect4:2}, the
norms
\[
\| (\alpha , \varphi ) \| _s := \| \varphi \| _{s,\alpha }
\]
are $\GMod$-invariant. In fact, by that lemma we have $ \|
\gamma^*\varphi \|_{s, \alpha \gamma} = \| \varphi \|_{s,\alpha} $.
Consequently, we obtain a quotient flat bundle of Sobolev spaces over
the moduli space:
\[
(\Auto \times \Sobolev ^{s} (\M))/\GMod \to \Mod=\Auto/\GMod\,;
\]
the fiber over $[\alpha] \in \Mod$ may be locally identified with the
space $\Sobolev ^s _\alpha (\M) $ normed by $\| \cdot
\|_{s,\alpha}$. We denote this bundle by $\Bundle^{s}$ and the class
of $(\alpha , \varphi )$ by $[\alpha , \varphi]$.
 
By the duality paring, we also have a flat bundle of distributions
$\Bundle^{-s}$ whose fiber over $[\alpha] \in \Mod$ may be locally
identified with the space $\Sobolev ^{-s} _\alpha (\M) $ normed by $\|
\cdot \|_{-s,\alpha}$. Observe that for this bundle $(\alpha, \mathcal
D) \equiv (\alpha\gamma^{-1}, \gamma_*\mathcal D)$ for all $\gamma\in
\GMod$ and $(\alpha ,\mathcal D)\in \Auto \times \Sobolev ^{-s} (\M)$.
We denote the class of $(\alpha , \mathcal D )$ by $[\alpha , \mathcal
D]$.

\subsection{Best Sobolev constant}

\paragraph{The best Sobolev constant.}  The Sobolev embedding theorem
implies that for any $\alpha \in \Auto$ and any $s > g+1/2$ there
exists a constant $B_s(\alpha ) >0$ such that any $f\in \Sobolev
_\alpha ^s (\M)$ has a continuous representative such that
\begin{equation} \label{Sobolev_embedding} \| f \| _\infty \leq B_{s}
  (\alpha ) \cdot \| f\| _{s,\alpha} \, .
\end{equation}
For any Sobolev order $s>g+1/2$, the \emph{best Sobolev constant} is
defined as the function on the group of automorphisms $\Auto$  given
by
\begin{equation} \label{best_Sobolev_constant} B_{s} (\alpha ) :=
  \sup_{f \in \Sobolev _\alpha ^s (\M) \backslash \{0\} } \frac {\|f
    \| _\infty}{\|f \| _{s,\alpha} }
\end{equation}

\begin{lemma}
  \label{lem:best-sobol-const}
  The best Sobolev constant $B_s$ is a $\GMod$-modular function on
  $\Siegel_g$, i.e. $B_s(\alpha) = B_s(\kappa \alpha\gamma)$ for all
  $\alpha \in \Auto$, all $\gamma \in \GMod$ and all $\kappa \in
  \Uni_g$.
\end{lemma}
\begin{proof}
  The $\Uni_g$ invariance is an immediate consequence of
  Remark~\ref{rem:sect4:1}.  By Lemma~\ref{lem:sect4:2}, the the
  pull-back map $\gamma^*$ is an isometry of $\Sobolev
  _{\alpha}^s(\M)$ onto $\Sobolev _{\alpha\gamma}^s (\M) $. As the map
  $\gamma^*$ is also an isometry for the sup-norm, the lemma follows.
\end{proof}

\begin{sloppypar}
Thus, we may regard $B_s$ as a function on the Siegel modular variety
$\Sigma _g= \Uni_g\backslash\Auto/\GMod$ or as a $\GMod$-invariant
function on the Siegel upper half-space $\Siegel_g$.  Recalling that
$\class{\alpha}$ denotes the class of $\alpha\in\Auto $ in $\Sigma
_g$, we shall write $ B_s(\class{\alpha})$ or  $ B_s([\alpha])$ for $B_s(\alpha)$. 
\end{sloppypar}

Let $\Cartan \subset \Symp_{2g}(\R) $ denote the Cartan subgroup of
diagonal symplectic matrices, $\Cartan^+ \subset \Cartan $ the
subgroup of positive matrices and let $\cartan \subset \symp_{2g}$ be
the Lie algebra of $\Cartan$.

For $ \alpha = \left( \begin{smallmatrix} \delta & 0 \\ 0 &
    \delta^{-1} \end{smallmatrix} \right) \in \Cartan^+$, where
$\delta = \mathrm{diag}(\delta _1,\dots , \delta _g)$ we define
\[
\height (\alpha) := \prod_{i=1}^g (\delta_i + \delta_i^{-1})
\]
\begin{proposition}
  \label{estimate_Sobolev_constant}
  For any order $s > g+1/2$ and any $\alpha \in \Cartan^+$ there
  exists a constant $C=C(s)>0$ such that
  \[
  B_{s} ( \class{\alpha }) \leq C \, \height (\alpha)^{1/2} \, .
  \]
\end{proposition}

\begin{proof} 
  Let $ \alpha = \left( \begin{smallmatrix} \delta & 0 \\ 0 &
      \delta^{-1} \end{smallmatrix} \right) \in \Cartan^+$, where
  $\delta = \mathrm{diag}(\delta _1,\dots , \delta _g)$. Since the map
  $\alpha^*:\Sobolev ^s (\M _{\alpha}) \to \Sobolev _{\alpha}^s (\M) $
  is an isometry, the best $s$-Sobolev constant $B_{s} ( [\alpha])$
  for the operator $1+\Delta_\alpha$ on the Heisenberg manifold $\M$
  is equal to the best $s$-Sobolev constant for the operator $1+\Delta$ on
  the Heisenberg manifold $\M_{\alpha}$, namely
  \begin{equation}
    \label{best_Sob_constant_alternative} 
    B_{s} ([\alpha] ) = \sup_{f \in \Sobolev ^s (\M_{\alpha}) \backslash
      \{0\} } \frac {\| f \| _\infty}{ \| (1+\Delta) ^{s/2}f \|
      _{L^2 (\M_{\alpha})}} \,.
  \end{equation}

  We fix the fundamental domain $F = [0,1]^g \times [0,1]^g \times
  [0,1/2]$ for the action of the lattice $\Lattice$ on $\Heis$. By the
  standard Sobolev embedding theorem, for any $s>g +1/2$ there exists
  a constant $C(s)$ such that for any $f \in W^s_{\text{loc}}(\Heis)$
  we have
  \[
  |f(I) |^2 \le C(s)\int_F |(1+\Delta)^{s/2}f(x)|^2\,\D{}x
  \]
  where $I=(0,0,0)$ is the identity of $\Heis$ and $\D{}x $ is the
  Haar measure assigning volume~$1$ to~$F$. Since left and right
  translation commute and since $(1+\Delta)$ operates on the left, 
  for every $f \in W^s_{\text{loc}}(\Heis)$ and every
  $h\in\Heis$ we have
  \begin{equation}
    \label{eq:sect4:234}
    |f(h) |^2 \le C(s)\int_{Fh} |(1+\Delta)^{s/2}f(x)|^2\,\D{}x\,.
  \end{equation}
  It easy to see that, for any $h\in \Heis$, the set $Fh$ is also a
  fundamental domain for~$\Lattice$. Furthermore, if we let $p_\alpha
  : h\in \Heis \mapsto h \Lattice _{\alpha}\in \M_{\alpha} $ denote
  the natural projection, the projection $p_{\alpha} ((Fh)^o)$ of the
  interior of $Fh$ covers each point of~$ \M_{\alpha^{-1}}$ at most
  \begin{equation}
    \label{covering_estimate} 
    2^g\prod_{i=1}^g \max \{\delta_i, {\delta_i}^{-1} \} \le 2^g \height (\alpha)
  \end{equation}
  times.

  Given  any $f \in \Sobolev ^s (\M_\alpha)$, let $\tilde f = f \circ
  p_\alpha$. Then,  for any $h\in \Heis$ and any integer $n\ge 0$
  \begin{align*}
    \int_{Fh} \left| (1+\Delta) ^{n/2} \tilde f(x) \right| ^2 \,\D{}x
    &\leq 2^g\, \height (\alpha) \int _{\M_\alpha}
    \left| (1+\Delta)^{n/2} f(x) \right| ^2  \,\D{}x  &  \qquad \text{(by  \eqref{covering_estimate})}\\
    & = 2^g \height (\alpha)\, \| (1+\Delta) ^{n/2} f\| ^2 _{L^2
      (\M_\alpha)}
  \end{align*}
  We deduce, by interpolation and by \eqref{eq:sect4:234}, that for
  any $s \geq g+1/2 $ there exists a constant $C$ such that
  \begin{equation} \label{Sob_interpolation} \sup_{h\in \M_{\alpha}} |
    f(h) | \le C \left( \height (\alpha) \right)^{1/2}\| f\|_{W^s
      (\M_{\alpha})} \, .
  \end{equation}
  This concludes the proof.
\end{proof}
 
\subsection{Best Sobolev constant and height function}
\label{sec:best-sobol-const-height function.}
The height of a point $Z \in \Siegel_g$ is the positive number
\begin{equation}
  \label{detY}
  \detY (Z) : = \det \Im(Z) \, . 
\end{equation}

Let $F_g \subset \Siegel _g$ denotes the Siegel fundamental domain for
the action of $\Symp_{2g}(\Z)$ on $\Siegel _g$ (see \cite{MR1046630}).
We define the \emph{height function} $ \Height \colon \Sigma_g \to
\R^+$ to be the maximal height of a $\Symp_{2g}(\Z)$-orbit (which is
attained by Proposition 1 of \cite{cartan10ouverts}), or,
equivalently, the height of the unique representative of an orbit
inside $F_g$. Thus, if $[Z] \in \Sigma _g$ denotes the class of $Z \in
\Siegel_g$ in the Siegel modular variety,
\begin{equation}
  \label{height_function}
  \Height ( [Z] ) := \max _{\gamma \in \Symp_{2g} (\Z) } \detY (\gamma (Z)) 
  =  \max _{\gamma \in \Symp_{2g} (\Z) } \det \Im  (\gamma (Z))
\end{equation}

Any point in $\Siegel_g$ may be uniquely written as $Z = X+i W^\top
DW$, where $X=(x_{ij})$ is a symmetric real matrix, $W= (w_{ij})$ is a
upper triangular real matrix with ones on the diagonal, and $D=
\mathrm{diag} (\delta _1 , \dots , \delta _g)$ is a diagonal positive
matrix. The coordinates $(x_{ij})_{1\leq i \leq j \leq g}$ ,
$(w_{ij})_{1\leq i<j \leq g}$ and $(\delta_i)_{1\leq i \leq g}$ thus
defined are called Iwasawa coordinates on the Siegel upper half-space.
For $t>0$, define $S _g(t) \subset \Siegel _g$ as the set of those
$Z=X+iW^\top DW \in \Siegel _g$ such that
\begin{equation} \label{S1} |x_{ij} |< t \qquad (1 \leq i,j \leq g )
\end{equation}
\begin{equation} \label{S2} |w_{ij} | < t \qquad ( i < j )
\end{equation}
\begin{equation} \label{S3} 1 < t\delta_1 \qquad \text{ and }\qquad 0
  < \delta _{k} < t \delta _{g+1} \qquad (1 \leq k \leq g-1 )
\end{equation}
For all $t$ sufficiently large, $S_g(t)$ is a ``fundamental open set''
for the action of $\Symp_{2g}(\Z)$ on $\Siegel _g$ (see
\cite{cartan10ouverts} or \cite{MR1046630}).  We will need the
following Lemma, which is an easy consequence of the expression
\begin{equation} \label{metric_Iwasawa_coordinates} ds^2 = \tr \left(
    dX Y^{-1}dX Y^{-1} + dD D^{-1} dD D^{-1} + 2 (W^\top)^{-1} dW^\top
    D dW W^{-1} D^{-1} \right)
\end{equation}
for the Siegel metric in Iwasawa coordinates.

\begin{lemma} \label{bounded_distance_diagonal} Any point $Z=X+iW^tDW
  $ inside a Siegel fundamental domain $F_g$ (actually inside the
  Siegel fundamental open set $S_g(t)$ for any fixed $t$ sufficiently
  large) is at a bounded distance from the point $iD$.
\end{lemma}


{\begin{proof}This is clear from the expression
    \eqref{metric_Iwasawa_coordinates} for the Siegel metric in
    Iwasawa coordinates. Indeed, let $Z = X+iW^\top DY$, with $W$ and
    $D$ as explained above, be a point in $S_g(t)$. We first observe
    that \eqref{S2} says that (the entries of) $W$ and $W^t$ are
    bounded, and, since the inverse of a bounded unipotent matrix is
    bounded as well, the same is true for $W^{-1}$ and
    $(W^\top)^{-1}$. Then, we observe that the non-zero entries of
    $(W^\top)^{-1} dW^\top DdW W^{-1}D^{-1}$ are all proportional to
    terms like $\delta _i / \delta _j$ with $j > i$ times something
    bounded, and $\delta _i / \delta _j < t^{j-i}$ by
    \eqref{S3}. Thus, all terms are bounded by $C \cdot t^{n-1}$ for
    some constant $C$ and all $t >1$ sufficiently large. Consequently,
    the integral
    \[
    \int_0^1 \sqrt{ 2(W^\top )^{-1}dW^\top DdW W^{-1}D^{-1}}
    \]
    along the path $[0,1] \ni t \mapsto t \cdot W $ is bounded,
    i.e. there exists a constant $C>0$ such that $d(X+iW^\top DW,
    X+iD) < C$ for any $Z=X+iW^ \top DW \in F_n$. Finally, it is clear that 
    we may set to zero each of the coordinates $x_{ij}$ of $X$ still
    staying a bounded distance away. Indeed, a path sending the
    $x_{ij}$ coordinate linearly to zero while keeping constant the
    other coordinates has length bounded by
    \[
    \int_{0}^1 \frac{dt}{\sqrt{\delta_j\delta_i}}
    \]
    which is bounded by $t$ because of  \eqref{S1}. Thus, a point $X+iD$ lies
    within a bounded distance from $iD$.
  \end{proof}

  The Siegel volume form $ dX dY / (\det Y)^{g+1}$ in Iwasawa
  coordinates is
  \begin{equation} \label{volume_Iwasawa_coordinates} d \Vol _g =
    \prod _{i \leq j} dx_{ij} \cdot \prod _{i < j} dw_{ij} \cdot \prod
    _{k} \delta_k^{-(k+1)} d \delta_k \, .
  \end{equation}

A computation, using again the fundamental open set $S_g(t)$, gives
the following.

\begin{lemma} \label{DL_function} The logarithm of the height function
  on the Siegel modular variety is distance-like with exponent $k_g=
  \tfrac{g+1}{2}$. More precisely, for any $\tau \gg 0$
  \[
  \mathrm{Vol}_g \left\{ [Z] \in \Sigma _g \, \, \mathrm{ s.t. }\,
    \Height ([Z]) \geq \tau \right\} \asymp e^{- \tfrac{g+1}{2} \tau}
  \, .
  \]
\end{lemma}

\begin{proof}
  A change of variable as in page 67 of \cite{MR1046630} shows that
  this volume is within a bounded ratio of
  \[\int_{e^\tau}^\infty t^{-(g+3)/2} dt \, . \]
\end{proof}

\begin{proposition} \label{best_Sobolev_constant2} For any $s > g+1/2$
  there exists a constant $C(s)>0$ such that the best Sobolev constant
  satisfies the estimate
  \[
  B_{s}(\class{\alpha}) \leq C(s) \cdot \left( \Height
    (\class{\alpha}) \right) ^{1/4} \, .
  \]
\end{proposition}

\begin{proof} Let $Z =X+iW^\top DW \in F_g$ be the representative of
  $\class{\alpha} \in \Sigma_g$ inside the Siegel fundamental domain,
  so that $B_s(Z)= B_{s}(\class{\alpha}) $. According to Lemma
  \ref{bounded_distance_diagonal}, $Z $ is within a uniformly bounded
  distance from the point $ i D$.  Thus, by Lemma \ref{lem:sect4:1},
  there exists a constant $C=C(s) >0 $ such that
  \[
  B_s(Z) \le C \, B_s(i D).
  \]
  Since $i D = \beta ^{-1} ( i) $, with $\beta =\left(
    \begin{smallmatrix}
      D^{-1/2}&0\\0&D^{1/2}
    \end{smallmatrix}\right)
  $ , we have $B_s(i D)= B_s(\beta)$ and, by
  Proposition~\ref{estimate_Sobolev_constant}, $B_s(\beta) \le C
  \height (\beta)^{1/2} \le C' \det(D)^{1/4} = C'\detY(\class{\alpha}
  )^{1/4}$.  
\end{proof}

\subsection{Diophantine conditions and logarithm law}
\label{subsection:Diophantine}

We will need, in the final renormalization argument, some control on
the best Sobolev constant $B_{s}( \class{\rho \alpha })$, hence, by
Proposition \ref{best_Sobolev_constant2}, on $\Height ( \class{ \rho
  \alpha } )$, when $\rho$ are certain automorphisms in the Cartan
subgroup $\Cartan \subset \Symp_{2g}(\R)$ of diagonal symplectic
matrices.  This control is the higher-dimensional analogue of the
escape rate of geodesics into the cusp of the modular surface.

\paragraph{Diophantine conditions.} 
Let $\cartan ^+ \subset \symp_{2g}$ be the cone of those 
$ \dcart = \left( \begin{smallmatrix} \delta & 0 \\ 0 &
    -\delta \end{smallmatrix} \right) \in \symp_{2g}$  where $\delta = \mathrm{diag}(\delta _1,\dots ,
\delta _g)$ is a non-negative diagonal matrix. 
We consider the
corresponding one-parameter subgroup of diagonal symplectic matrices
$e^{t \dcart} \in \Cartan \subset \Symp_{2g}(\R) $, and also denote by
$e^{-t\dcart}$ the corresponding automorphisms $(x,\xi,z) \mapsto (e^{-t
  \delta }x, e^{t \delta}\xi, t)$ of the Heisenberg group.

We recall the under the
left action of  the symplectic matrix $\beta = \left(\begin{smallmatrix} A & B \\
    C & D \\ \end{smallmatrix}\right) \in \Symp_{2g}(\R)$, the height
on $\Siegel _g$ transforms according to
\begin{equation}
  \label{eq:sect4:3}
  \detY (  \beta (Z) ) = \left| \det (C Z+D) \right| ^{-2} \detY( Z )  
\end{equation}

\begin{lemma} \label{trivial_Diophantine_bound} Let
  $\delta =\mathrm{diag}(\delta_1,\delta_2,\dots, \delta_g)$  a
  non-negative diagonal matrix and let
  $\dcart = \left( \begin{smallmatrix} \delta & 0 \\ 0 &
      -\delta \end{smallmatrix} \right) \in \frak a$ generating the
  automorphism $e^{t\dcart} \in \Symp_{2g}(\R)$.  For any $[\alpha] \in
  \Mod$ and any $t \geq 0$ we have the trivial bound
  \[ \Height ( \class {e^{-t\dcart} \alpha } ) \leq ( \det e^{t \delta} )^2
  \, \Height (\class{\alpha}) \, .
  \]
\end{lemma}

\begin{proof}
  We recall that $\Height$ is the maximal $\detY$ of a
  $\Symp_{2g}(\Z)$ orbit. Therefore, we may take the representative
  $\beta = \alpha \gamma$, with $\gamma \in \Symp_{2g}(\Z)$, such that
  $(e^{-t\dcart}\beta) ^{-1} (i) \in \Siegel_g$ realizes the maximal
  height, i.e.  $\Height ( \class {e^{-t\dcart} \alpha } ) = \detY
  ((e^{-t\dcart} \beta)^{-1} (i) )$, and prove the inequality for the
  function $\detY$, namely
  \[\detY( (e^{-t\dcart} \beta)^{-1} (i) ) \leq ( \det e^{t \delta} )^2 \,
  \detY(\beta^{-1}(i)) \, , \] since then $ \detY(\beta^{-1}(i)) \leq
  \Height (\class{\alpha})$.  By the Iwasawa decomposition, any
  symplectic matrix $\beta \in \Symp_{2g}(\R)$ sending the base point
  $i := i \mathbf 1_g$ into the point $\beta ^{-1} (i) = X+iW^\top DW$
  may be written as $\beta^{-1} =
  \nu \eta \kappa $ with $ \nu = \left( \begin{smallmatrix} W^\top & XW^{-1} \\
      0 & -W^{-1} \end{smallmatrix} \right) $, $ \eta =
  \left( \begin{smallmatrix} \sqrt{D} & 0 \\ 0 &
      \sqrt{D}^{-1} \end{smallmatrix} \right)$ and $\kappa \in
  \Uni_g$. By the formula~\eqref{eq:sect4:3},
  \[ \detY (\nu \eta \kappa (Z)) = \detY (\eta \kappa (Z)) = (\det D)
  \, \detY (\kappa (Z)) \] (because $\det W =1$) for all $Z \in
  \Siegel_g$.  Therefore, since $\detY (\kappa (i) ) =1$, we only need
  to prove
  \[ \detY (\kappa e^{t\dcart} (i)) \leq \det e^{2t\delta} \, . \] Let
  $\kappa = \left( \begin{smallmatrix} A & B \\ -B &
      A \end{smallmatrix} \right) \in \Uni_g$, i.e.  with $A^\top
  A+B^\top B = \mathbf 1_g$ and $A^\top B$ symmetric. Since $e^{t\dcart}
  (i) = ie^{2t\delta }$, using formula~\eqref{eq:sect4:3}, the above
  inequality is equivalent to
  \[ | \det (-iBe^{2t\delta}+A) |^{-2} \cdot \det e^{2t\delta} \leq
  \det e^{2t\delta}\] i.e. to
  \[ | \det (A-iBe^{2t\delta}) |^{2} \geq 1 \, , \] and therefore to
  \[ | \det (AA^\top + Be^{4t\delta}B^\top) | \geq 1 \, .  \] 
  But, by
  our hypothesis on $\delta$ and $t$, the norm of $e^{2t\delta}$ is $\|
  e^{2t \delta } \| \geq 1$, and therefore
  \[ \lin x , (A^\top A + B^\top e^{4t\delta}B) x \rin \geq \lin x ,
  (A^\top A+ B^\top B) x \rin = \| x \| ^2 \] for any vector $x \in
  \R^g$. Hence, all the eigenvalues of the symmetric matrix $A^\top A
  + B^\top e^{4t\delta}B$ are $\geq 1$, and the same occurs for the
  determinant.
\end{proof}

\begin{definition}
  \label{def:Diophantine_properties}
  Let $\delta = \mathrm{diag} (\delta_1,\dots, \delta_g)$ be a
  non-negative diagonal matrix, and $\dcart = \left(\begin{smallmatrix}
      \delta & 0 \\ 0 & -\delta \\ \end{smallmatrix}\right) \in
  \cartan ^+\subset \symp_{2g}$.  We say that an automorphism $\alpha
  \in \Auto$, or, equivalently, a point $[\alpha] \in \Mod$ in the
  moduli space,
  \begin{itemize}
  \item is $\dcart$-\emph{Diophantine} of \emph{type} $\sigma$ if there
    exists a $\sigma >0$ and a constant $C>0 $ such that
    \begin{equation} \label{Diophantine} \Height ( \class{ e^{-t\dcart}
        \alpha }) \leq C \, \Height ( \class{e^{-t\dcart} })^{(1-\sigma)}
      \, \Height (\class{\alpha}) \qquad \qquad \forall \, t \gg 0 \,
      ,
    \end{equation}

  \item satisfies a $\dcart$-{\em Roth condition} if for any $\varepsilon
    >0$ there exists a constant $C >0$ such that
    \begin{equation} \label{Roth}  \Height ( \class{e^{-t\dcart} \alpha
      }) \leq C \Height ( \class{ e^{-t\dcart} })^\varepsilon \, \Height
      (\class{\alpha}) \qquad \qquad \forall \, t \gg 0 \, ,
    \end{equation}
    i.e. if it is Diophantine of every type $0 < \sigma < 1$.
  \item is of {\em bounded type}  if
    there exists a constant $C>0$ such that
    \begin{equation}
      \label{bounded}
      \Height ( \class{e^{-t\dcart} \alpha })
      \leq C  \end{equation} 
    for all $\dcart \in \cartan ^+$ and all $t \geq 0$. 
  \end{itemize}
\end{definition}

\begin{remark}
  \label{rem:Diophantine}
  In the final section, dealing with theta sums, we will be interested
  in Diophantine properties in the direction of the particular $
  \dcart = \left(\begin{smallmatrix} I & 0 \\ 0 & -I
      \\ \end{smallmatrix}\right) \in \cartan $.  For such $\dcart$,
  the Diophantine properties of an automorphism $\alpha \in \Symp_{2g}
  (\R)$ only depend on the right $\mathrm{T}$ class of $\alpha^{-1}$,
  where $\mathrm{T} \subset \Symp_{2g}(\R)$ is the subgroup of
  block-triangular symplectic matrices of the form $
  \left( \begin{smallmatrix} A & B \\ 0 &
      (A^\top)^{-1} \end{smallmatrix} \right)$.  In particular, those
  $\alpha$ in the full measure set of those automorphisms such that
  $\alpha ^{-1} =\left( \begin{smallmatrix} A & B \\ C &
      D \end{smallmatrix} \right) $ with $A \in \GL_g(\R)$ are in the
  same Diophantine class of $\beta=\left( \begin{smallmatrix} I & 0 \\
      -X & I \end{smallmatrix} \right) $, where $X$ is the
  symmetric matrix $X = CA^{-1}$.  For such lower-triangular block
  matrices $\beta$, the Height in the Diophantine conditions above is
  (see \eqref{eq:sect4:3})
  \begin{equation} \label{Diophantine_condition} \begin{split} 
      \Height ( \class{e^{-t\dcart} \beta})
      & = \max \left| \det (QQ^\top e^{-2t}+ (QX+P) (QX+P)^\top
        e^{2t} )\right| ^{-1}
    \end{split} \end{equation} the maximum being over all $
  \left(\begin{smallmatrix} N & M \\ P & Q \\ \end{smallmatrix}\right)
  \in \Symp_{2g}(\Z)$.  When $g=1$, we recover the classical relation
  between Diophantine properties of a real number $X$ and geodesic
  excursion into the cusp of the modular orbifold $\Sigma _1$, 
  or the behaviour of a certain flow
  in the space $\mathfrak{M} _1= \SL_2(\R) / \SL_2(\Z)$ of unimodular
  lattices in the plane. Indeed, our \eqref{Diophantine_condition}
  coincides with the function $\delta (\Lambda_t) = \max _{v \in
    \Lambda_t \backslash \{ 0 \} } \, \| v \|_2^{-2} $, where $\Lambda
  _t$ is the unimodular lattice made of $\left(\begin{smallmatrix} e^t
      & 0 \\ 0 & e^{-t} \\ \end{smallmatrix}\right)
  \left(\begin{smallmatrix} 1 & X \\ 0 & 1
      \\ \end{smallmatrix}\right) \left(\begin{smallmatrix} P \\
      Q \end{smallmatrix}\right)$, with $P,Q \in \Z$.  
  The maximizers, for increasing time $t$,
  define a sequence of relatively prime integers $P_n$ and $Q_n$ which
  give best approximants $P_n/Q_n$ to $X$ in the sense of
  continued fractions.  In particular, our definitions of Diophantine,
  Roth and bounded type coincide with the classical notions. 
  
  This same function $\delta (\Lambda_t)$, extended  to the space $\SL_n(\R) / \SL_n(\Z)$ 
  of unimodular lattices in $\R^n$,  has been used by Lagarias
  \cite{MR662052}, or, more recently,  by Chevallier \cite{MR2172275}, to understand 
  simultaneous Diophantine approximations. 
   A similar function,   $\Delta (\Lambda_t) = \max _{v \in
    \Lambda_t \backslash \{ 0 \} } \, \log (1/ \| v \|_\infty )$,  has been considered by Dani \cite{MR794799} 
    in his correspondance between Diophantine properties of systems of linear forms 
    and certain flows in $\SL_n(\R) / \SL_n(\Z)$,  or more recently by Kleinbock and Margulis \cite{MR1719827} to prove a ``higher-dimensional multiplicative Khinchin theorem''. 
 
\end{remark}


\paragraph{Khinchin-Sullivan-Kleinbock-Margulis logarithm law.}
A stronger control on the best Sobolev constant comes from the
following generalization of the Kinchin-Sullivan logarithm law for
geodesic excursion ~\cite{MR688349}, due to Kleinbock and Margulis
~\cite{MR1719827}.
  
Let $\mathsf X=\G/ \mathsf \Lambda $ be a homogeneous space, equipped
with the probability Haar measure $\mu$. A function $\phi : \mathsf X
\to \R$ is said $k$-DL (for ``distance-like'') for some exponent $k>0$
if it is uniformly continuous and if there exist constants $c_\pm>0$
such that
\[
c_-e^{-kt} \leq \mu \left( \{ x \in \mathsf X \, \text{ s.t. } \,
  \phi(x) \geq t \}\right) \leq c_+ e^{-kt}
\]
Theorem 1.7 of \cite{MR1719827} says the following. 
  
\begin{proposition}[Kleinbock-Margulis] \label{Khinchin-Kleinbock-Margulis_law}
   Let $\G$ be a
  connected semisimple Lie group without compact factors, $\mu$ its
  normalized Haar measure, $\mathsf \Lambda \subset \G$ an irreducible
  lattice, $\mathfrak a$ a Cartan subalgebra of the Lie algebra of
  $\G$, $\mathbf z$ a non-zero element of $\mathfrak a$. If $\phi :
  \G/ \mathsf \Lambda \to \R$ is a $k$-DL function for some $k>0$,
  then for $\mu$-almost all $x \in \G/ \mathsf \Lambda$ one has
  \[
  \limsup _{t\rightarrow \infty} \frac{\phi(e^{t\mathbf z} x )}{\log
    t} = 1/k \, .
  \]
\end{proposition}
  
We have seen in Proposition \ref{DL_function} that the logarithm of
the height function $\Height$ is a DL-function with exponent
$\frac{g+1}{2}$ on the Siegel variety $\Sigma_g$, hence (induces a
DL-function) on the homogeneous space $ \Mod = \Symp_{2g}(\Z)
\backslash \Symp_{2g}(\R)$.  Thus, the following proposition is a consequence  of the easy part of
Proposition \ref{Khinchin-Kleinbock-Margulis_law} and of Proposition
\ref{best_Sobolev_constant2}. 
  
\begin{proposition} \label{Sobolev_constant_estimate} Let $s > g +
  1/2$. For any non-zero vector $\dcart \in \frak a $ in the Cartan
  subalgebra of diagonal symplectic matrices there exists a full
  measure set $\Omega_g(\dcart) \subset\Mod$ such that for all $[\alpha]\in
  \Omega_g( \dcart )$ we have
  \[
  \limsup _{t \to \infty} \frac{\log \Height ( \class{e^{-t\dcart} \alpha})
  }{\log t} \leq \frac{2}{g+1} \, .
  \]
  In particular, any such $[\alpha] $ satisfies a $\dcart$-Roth condition.
\end{proposition}

\subsection{Flat TV-bundles over moduli spaces}

We fix the standard lattice $\Lattice$ of $\Heis$ and restrict the
space of lattices to the orbit of $\Gamma$ under the automorphism
group $\Auto$ of $\Hei$. Then by a \emph{marked lattice of $\Heis$},
we simply mean an automorphism $\alpha \in \Auto$, and the lattice
$\Lattice _\alpha = \alpha^{-1} \mathsf\Gamma$ will be identified with
the point $\GMod\alpha\in \Mod:=\GMod\backslash \Auto$, where $ \GMod
= \operatorname{Stab}_\Auto(\Gamma)$.


\paragraph{Flat $TV$-bundles.}  
Let $V$ be a topological vector space endowed with a representation
$\rho\colon (\gamma, v)\mapsto \gamma.v$ of $\GMod$~on~$V$. We let
$V(\Mod) \to \Mod$ to be flat bundle of topological vector spaces with
\emph{(typical) fibre} $V$ and \emph{monodromy} $\rho$.

We recall, for the reader's convenience, the definition of $V(\Mod)$:
we let $\GMod$ act on the left on $\Auto\times V$ by $\gamma.(\alpha,
v) =(\gamma\alpha , \gamma.v)$; then the map $ (\alpha, v) \mapsto
\alpha$ is $\GMod$-equivariant and it defines a quotient bundle map
\[
V(\Mod):=\GMod\backslash(\Auto\times V) \to \Mod:=\GMod\backslash
\Auto \, .
\]
The class of $(\alpha,v)$ in $V(\Mod)$ will be indifferently denoted
$[\alpha,v]$ or $\GMod (\alpha,v)$.

If $V(\Mod)$ is a flat TV-bundle over $\Mod$ with fibre $V$ and
monodromy $\rho$, we define the \emph{dual bundle} $V'(\Mod)$ of
$V(\Mod)$ as the bundle obtained by letting~$\GMod$ act on the
topological dual $V'$ of $V$ by the contragredient
representation~$\rho^t$.  Thus, if $\lin T , v \rin$ denotes the value
of $T \in V'$ on the vector $v \in V$, the group $\GMod $ operates on
$V'$ according to $\colon (\gamma, T)\mapsto \gamma.T$, where $\lin
\gamma . T , v \rin = \lin T , \gamma . v \rin$.
   
If, moreover, $V$ is a Banach space and $\GMod$ acts by isometries on
$V$, then we obtain a Banach space (flat) bundle over $ \Mod$.
Indeed, let $\| \cdot \|_V$ denotes the norm of $V$, and assume that
$\Vert v\Vert_V= \Vert \gamma .v\Vert_V$ for any $\gamma \in \GMod$
and any $v \in V$.  If we define $\Vert (\alpha , v) \Vert
_{V(\Mod)}:=\Vert v\Vert_V$ for $(\alpha , v)\in \Auto\times V$, we
obtain a $\GMod$-invariant norm on the fibers of the map $(\alpha ,
v)\mapsto \alpha$; hence these norms pass to the quotient bundle
$V(\Mod)$.

More generally, let $\{ \Vert \cdot\Vert_{V_\alpha} \}_{\alpha \in
  \Auto}$ be a $\GMod$-equivariant family of norms on $V$, that is a
family of equivalent norms satisfying $\Vert v\Vert_{V_\alpha}= \Vert
\gamma .v\Vert_{ V_{\gamma\alpha}}$ for any $\gamma \in \GMod$ and any
$v \in V$.  Then the function $(\alpha , v ) \to \Vert
v\Vert_{V_\alpha}$ is $\GMod$-invariant function, and we obtain a
Banach bundle defining the norm of $[\alpha , v ] \in V (\Mod)$ as
\[
\|\, [\alpha , v] \,\| _{V (\Mod)} :=\Vert v\Vert_{V_\alpha} \, .
\]

Analogous considerations apply when $V$ is Fr\'echet spaces
topologized by a family of $\GMod$-equivariant seminorms $\Vert
\cdot\Vert_{V_{\alpha,i}}$, for $i \in \mathcal I$, or when $V$ is a
Hilbert space.

\paragraph{Examples.}
\begin{itemize}
\item Take the standard nilmanifold $M= \Gamma \backslash \Heis$. Let
  $V= C^0(M)$ or $V=L^p(M,\D x)$. Then $\GMod$ acts isometrically on
  $V$. Recall that every automorphism $\gamma$ of $\Heis$ that
  stabilizes $\Gamma$ defines a diffeomorphism $\gamma :M \to M$
  according to $\gamma (\Gamma h) = \Gamma (\gamma h)$.  Therefore
  $\GMod$ acts (on the left) on $V$ by pullback: $(\gamma , \varphi)
  \mapsto {}^{\gamma} \! \varphi$, where ${}^{\gamma} \!  \varphi (m):
  = \varphi (\gamma^{-1}m)$, for all $(\gamma, \varphi) \in
  \GMod\times V$ and $m=\Gamma h \in M$. This action is clearly
  isometric w.r.t. to the norms $\| \cdot \|_\infty$ or $\| \cdot \|
  _{L^p(M)}$.
 
\item Let $L$ be a left invariant operator on
  $C^\infty(\Heis)$\footnote{Left invariance means that
    $L({}^g\!\varphi)= {}^g\!(L\varphi)$, where ${}^g\!\varphi(h) =
    \varphi(gh)$}.  The deformation space $\Auto$ acts on
  $C^\infty(\Heis)$ by the standard right and left actions: $ \varphi
  ^\alpha := \varphi \circ\alpha$ and ${}^\alpha \!\varphi := \varphi
  \circ\alpha^{-1}$, respectively.  Therefore it operates on
  left-invariant operators on $C^\infty(\Heis)$ by push-forward,
  namely $(\alpha , L) \mapsto {} _\alpha \! L$, defined by $ _\alpha
  L \, \varphi = {}^{\alpha} \! ( L \, \varphi^\alpha) $ for any
  $\varphi \in C^\infty (\Heis)$, i.e.
  \[
  (_\alpha L \, \varphi ) (h) = (L \, \varphi ^\alpha ) ( \alpha ^{-1}
  h) \, .
  \]
  Fix $L$, and consider the family $\{ _\alpha L \} _{\alpha \in
    \Auto}$ of left invariant operators on $C^\infty(\Heis)$, hence a
  family of operators on $C^\infty(M)$, which we still denote by
  $(\alpha.L)_{\alpha\in \Auto}$. Defining
  \[
  \Vert \varphi \Vert_{V_\alpha} = \Vert {}_\alpha \! L \, \varphi
  \Vert _{V}
  \]
  where on the right hand side $\Vert \cdot \Vert_V$ denotes the $
  C(M)$ or the $L^p(M,\D x)$ norm, we obtain a family of
  $\GMod$-equivariant seminorms on $C^\infty(M)$. In fact, for $\gamma
  \in \GMod$,
  \[
  \Vert {}^\gamma \! \varphi \Vert_{V_{\gamma\alpha}} = \Vert
  {}_{\gamma\alpha} \! L \, {}^{\gamma} \! \varphi \Vert _V = \Vert
  {}^{\gamma \alpha} \! \left( L \, \varphi ^{\alpha} \right) \Vert _V
  = \Vert {}^{\gamma} \! \left( {}_\alpha \! L \, \varphi \right)
  \Vert _V = \Vert {}_\alpha \! L \, \varphi \Vert _V = \Vert \varphi
  \Vert_{V_{\alpha}}
  \]
  where in the last step we used the fact that the action of $\GMod$
  on $ C(M)$ or the $L^p(M,\D x)$ is isometric.
\end{itemize}

\paragraph{Sobolev bundles.}  The most notable instance of the
previous example is obtained considering an elliptic operator which is
self-adjoint on $L^2(M,\D x)$, e.g.\ $(1+ \Delta)$. In this case, for
$L=(1+ \Delta)^{s/2}$, $s\ge 0$, we obtain, completing $C^\infty(M)$
for the norm $\| {}_ \alpha \!L \, \varphi \|_{L^2(M)}$, $\alpha\in
\Auto$, a bundle of Hilbert spaces over~$ \Mod=\GMod\backslash \Auto$,
denoted
\[
\Sobolev^s(M) \hookrightarrow \Sobolev^s (\Mod) \to \Mod \, ,
\]
endowed with the product
\[
\lin \,[\alpha, \varphi ], [\alpha, \psi ]\,\rin _{\Sobolev^s(\Mod)}
:= \lin \varphi , \psi \rin _{\Sobolev^s_\alpha(M)} = \lin (1+
{}_\alpha \! \Delta)^{s} \varphi , \psi \rin _{L^2(M)} \,
\]
and with the corresponding norm $\Vert \cdot
\Vert_{\Sobolev^s(\Mod)}$. The fibre above~$\GMod\alpha$ is the
standard Hilbert space $\Sobolev^s (M)$ of $s$-Sobolev
functions~$\varphi$ on~$M$

The dual bundles of the bundle $\Sobolev^s (\Mod)$ is the bundle
$\Sobolev^{-s} (\Mod)$ . TIt is a space of distributions, equipped
with the norm
\[
\| [ \alpha , D ] \| _{\Sobolev^{-s}(\Mod)} = \sup_{\| \varphi \|
  _{\Sobolev _\alpha ^s(M)} =1} | \lin D , \varphi \rin _{\Sobolev
  _\alpha ^s (M)} | \, .
\]

Finally we remark that the Frechet space $V=C^\infty(M)$ give rise to
the bundle $\mathcal C^{\infty}(\Mod)$, topologized by the norms $\|
\cdot \| _{\alpha, {n}}$, ($n\in \N$). (Attention to the difference
between the bundle $\mathcal C^{\infty}(\Mod)$ an the Frechet space
$C^{\infty}(\Mod)$).

\paragraph{Observation.}  If $W\hookrightarrow V$ is a continuons
embedding of Banach spaces then for any action of $\GMod$ on $V$ we
obtain a continuous embedding of bundles $W(\Mod)\hookrightarrow
V(\Mod)$. In particular if the monodromy of $\GMod$ on $V$ is
isometric, we see that
\[
\sup_{v\in W\setminus \{0\}}\frac {\| [ \alpha, v ]\|_{V_\alpha}}{\|
  [\alpha, v] \|_{W_\alpha}} = \sup_{v\in W\setminus \{0\}}\frac {\|
  v\|_{V}}{\| [ \alpha, v ] \|_{W_\alpha}}
\]
is a $\GMod$-equivariant function.  The most significative instances
of this is obtained considering the compact embedding $
\Sobolev^s(M)\hookrightarrow C(M)$, for $s> g+1/2$, provided by the
Sobolev embedding theorem.  We are therefore led to the following
definition.


\subsection{Harmonic analysis on nilmanifolds}
\label{subsection:harmonic}

\paragraph{Derived representation and Laplacians.} Consider a (the
standard?)  nilmanifold $M = \Gamma \backslash \Hei$. The derivative
of the regular representation is the representation of the Lie algebra
$\hei $ on $C^\infty (M)$, defined as
\[
(V \, \varphi )(m) = \left.  \frac{d}{d\varepsilon } \varphi \left( m
    e^{\varepsilon V} \right) \right| _{\varepsilon =0}
\]
for $V \in \hei$ and $\varphi \in C^\infty (M)$.  It extends to a
representation of the universal enveloping algebra $\mathcal U (\hei
)$, acting as (left-invariant) differential operators on $C^{\infty}
(M)$.

A \emph{Laplacian} (or {\it Laplace-Beltrami operator}, or
\emph{Casimir operator}) on the Lie group $\Heis$ is an operator
$\Delta = -\sum_{k=1}^g V_k^2$, where $V_1,\ldots ,V_{2g+1}$ is any
basis of $\hei$. We fix the canonical ``symplectic basis'' $X_1,\dots
, X_g, \Xi _1, \dots , \Xi _g, T$ as in \eqref{symplectic_base_hei}.
For $k=1,2,\dots, g$, we define the \emph{partial Hermite operators}
$H_k := - X_k^2 - \Xi _k ^2$ and the \emph{partial Laplacians} $\Delta
_k := H_k - T^2 = - X_k ^2 - \Xi _k ^2 - T^2$.  An important
observation is that the partial Laplacians commute, i.e.\ $[\Delta _k
, \Delta _j ] = 0$ for all $k,j $. Define the \emph{Hermite operator}
$H := H_1+H_2+\dots+H_g$ and the \emph{Laplacian }
\[
\Delta := H - g T^2 = \Delta _1 + \Delta_2+ \cdots + \Delta _g
\]
We also define the operator
\[
K := (1+ \Delta _{1}) (1+ \Delta _{2})\ldots (1+ \Delta_{g}) \, .
\]
(so that $K^{1/g}$ is the geometric mean of the partials
$(1+\Delta_k)$).

\paragraph{The action of automorphisms on the universal enveloping algebra.}
The automorphism group $\Auto$ operates (on the right) on
$C^\infty(\Hei)$ by pullback, namely $(\alpha , \varphi) \mapsto
\varphi ^\alpha $, defined as $\varphi ^\alpha (h) = \varphi ( \alpha
h)$ \cyan{(the ``differential geometry'' notation for $\varphi
  ^\alpha$ would be $\alpha ^\ast \varphi$)}.  We get a right action
if we set ${}^\alpha \!  \varphi = \varphi^{\alpha^{-1}}$.  It
operates on the universal enveloping algebra $\mathcal U (\hei )$, or
on left-invariant (differential) operators, by push-forward, namely
$(\alpha , L) \mapsto {} _\alpha \! L$, defined by $( _\alpha L \,
\varphi ) = {}^{\alpha} \! ( L \, \varphi^\alpha) $ for any $\varphi
\in C^\infty (\Hei)$, i.e.
\[
(_\alpha L \, \varphi ) (h) = (L \, \varphi ^\alpha ) ( \alpha ^{-1}
h) \, .
\]
In particular, an automorphism $\alpha = (S, (y, \eta)) \in \Auto$
operates on the Lie algebra $\hei$ according to
\[
(q,p,t) \mapsto \, _\alpha (q,p,t) = (S (q,p) , t+ p \cdot y +\eta
\cdot q)\, .
\]

Let $\Gamma $ be a lattice in $\Hei$, and $M = \Gamma \backslash
\Hei$.  For $\alpha \in \Auto$, define the lattice $\Gamma _\alpha =
\alpha ^{-1} \Gamma$, and the nilmanifold $M_\alpha = \Gamma _\alpha
\backslash \Hei$. Then $\alpha $ induces a diffeomorphism (which we
will denote with the same symbol!)  $\alpha : M_\alpha \to M$
according to
\[
\Gamma _\alpha \cdot h \mapsto \alpha ( \Gamma _\alpha \cdot h ) =
\Gamma \cdot ( \alpha h ) \, ,
\]
where $h \in \Hei$.  If $\varphi \in C^\infty(M)$, then $\varphi
^{\alpha}$, defined by $\varphi ^{\alpha } (\Gamma _\alpha h) =
\varphi (\Gamma (\alpha h))$, belongs to $C ^\infty (M_\alpha)$. Let
$L \in \mathcal U (\hei)$ be a left-invariant (differential) operator
on $\Hei$, seen as a (differential) operator on $M_\alpha$. Then
$_\alpha L$, defined as
\[
( _\alpha L \, \varphi ) (\Gamma \cdot h) = (L \cdot \varphi^{\alpha})
(\Gamma _\alpha \cdot (\alpha^{-1} h))
\]
for $\varphi \in C^\infty (M)$, is a differential operator on $M$.

\paragraph{The action of automorphisms on the Laplacian.}  An
automorphism $\alpha \in \Auto$ sends the fixed basis of $\hei (\R^g)$
into the new basis $ _\alpha X_{1}, \, _\alpha X_{2}, \, \dots , \,
_\alpha X_{g}, \, _\alpha \Xi _{1}, \, _\alpha \Xi _{2}, \, \dots , \,
_\alpha \Xi _{n}, T$ which is still symplectic (i.e. the commutators
are still $[ \, _\alpha X_{k}, \, _\alpha \Xi _{j}]= [X_k, \Xi _j] =
\delta _{k,j}$). Correspondingly, we define the new partial Laplacians
\[
_\alpha \! \Delta _{k} = - \, _\alpha X_{k}^2 - \, _\alpha \Xi _{k} ^2
- T^2 \qquad \qquad k=1,2,\ldots ,g
\]
the Laplacian
\[
_\alpha \! \Delta = \sum _{k=1}^g \, _\alpha \Delta _{k}
\]
and
\[
_\alpha \! K = (1+ \, _\alpha \Delta _{1}) (1+ \, _\alpha \Delta
_{2})\ldots (1+ \, _\alpha \Delta_{g}) \, .
\]
Observe that the $ _\alpha \Delta _{k}$ still commute.  Since
automorphisms in $\Auto$ acts trivially on the center of $\Hei$, they
preserve the orthogonal decomposition
\eqref{orthogonal_decomposition}.

\paragraph{Laplacian in the Schr\" odinger representation.} The
infinitesimal Schr\" odinger representation $\sch ^\hbar$ (see
\eqref{Schrodinger_symplectic_base_hei}) sends the partial Hermite and
Laplace operators into the operators
\[
^\hbar \! H_k := \sch ^\hbar (H_k)= (2\pi )^2 \left( \hbar ^2 P_k ^2 +
  Q_k^2 \right) = -\hbar ^2 \frac{\partial ^2}{\partial x_k^2} + (2
\pi )^2 x_k^2 \qquad \qquad k=1,2,\ldots ,g
\]
and
\[
^\hbar \! \Delta_k := \sch ^\hbar (\Delta) = -\hbar ^2 \frac{\partial
  ^2}{\partial x_k^2} + (2 \pi )^2 x_k^2 + (2\pi)^2 \hbar^2 \qquad
\qquad k=1,2,\ldots ,n
\]
on $\SB (\R^g)$, respectively. Consequently, the Laplacian $\Delta$ is
sent into the operator
\begin{equation} \label{hbar_Delta} ^\hbar \! \Delta := \sch ^\hbar
  (\Delta ) = n (2\pi)^2 \hbar ^2 + \sum_{k=1}^g \, ^\hbar H_k \, .
\end{equation}
Finally,
\begin{equation} \label{hbar_K} ^\hbar \! K := \sch ^\hbar ( K) = (1+
  \, ^\hbar \Delta _{1}) (1+ \, ^\hbar \Delta _{2})\ldots (1+ \,
  ^\hbar \Delta_{g})
\end{equation}
The Hermite operators, hence the (partial) Laplacians and the operator
$K_g$, are diagonalized by the Hermite basis $\{ \phi_a \}_{a\in \N^g}
$ of $L^2 (\R^g)$, indexed by $a = ( a _1 , a _2, \dots , a _g ) \in
\N ^g $ (see for example \cite{MR657581}). We have
\[
^\hbar \! H_k \, \phi_a = 2\pi \hbar (2 a _k + 1) \phi_a \, , \qquad
^\hbar \Delta_k \, \phi_ a = 2\pi \hbar (2 a _k +1+2\pi \hbar ) \phi_a
\]
and, summing over $k=1,\dots,g$,
\[
^\hbar \! H \, \phi_a = 2\pi \hbar (2 | a | + g) \phi_a \, , \qquad
^\hbar \Delta \, \phi_ a = 2\pi \hbar (2|a| +g +2\pi \hbar g) \phi_a
\]
(above, $|a|=a_1+\dots + a_g$).  Finally, we observe that
\[
^\hbar \! K\, \phi _a = \left( \prod_{k=1}^g \left( 1+ 2\pi \hbar +
    (2\pi)^2 \hbar ^2 + 4 \pi \hbar a _k \right) \right) \phi _a \, .
\]

\paragraph{Sobolev spaces.} For any $\alpha \in \Auto$, the Laplacian
$_\alpha \Delta $ is an elliptic, positive and essentially
self-adjoint operator on $\SB (M)$, hence we may use the spectral
theorem to define $(1+ {} _\alpha \!  \Delta)^s$ for any $s \in
\R$. We define as usual, for all $s \in \R$, the Sobolev spaces
$\Sobolev _{\alpha}^s (M)= (1+{} _\alpha \!  \Delta)^{-s/2} L^2(M )$,
completion of $C^\infty (M)$ with respect to the norm
\[
\| \varphi \| _{\Sobolev_\alpha ^ s(M)} = \| (1+{} _\alpha \!\Delta)
^{s/2} \varphi \| _{L^2(M)} \, .
\]
They are Hilbert spaces if equipped with the inner product
\[
\langle \varphi , \phi \rangle _ {\Sobolev ^s_\alpha (M)} = \langle
(1+\,\!  _\alpha \Delta) ^{s/2} \varphi , (1+{} _\alpha \! \Delta)
^{s/2} \phi \rangle _{L^2(M)}
\]
Thus, $\Sobolev ^{-s}_\alpha (M)$ is the dual Hilbert space of
$\Sobolev ^{s}_\alpha (M)$.

For technical reasons (useful when solving the cohomological equation
in a Schr\" odinger representation) we also consider the operator $K =
(1+ \Delta _{1}) (1+ \Delta_{2})\ldots (1+ \Delta_{g})$ \cyan{(should
  we use, instead, $K_g ^{1/g}$, the geometric mean of the partial
  Laplacians?)}  and, for any $\alpha \in \Auto$, the operators
$_\alpha \! K $, which are positive and essentially self-adjoint,
although not elliptic.  For any $s \in \R$, we then define the
``modified Sobolev spaces'' $\Sob _{\alpha } ^s (M)$, completion of
$C^\infty (M)$ with respect to the norm
\[
\| \varphi \| _{\Sob_\alpha ^s(M)} = \| {}_\alpha \! K ^{s/2} \varphi
\| _{L^2(M)}
\]
They are Hilbert spaces if equipped with the inner product
\[
\langle \varphi , \phi \rangle _ {\Sob ^s_\alpha (M)} = \langle
{}_\alpha \! K ^{s/2} \varphi , \, _\alpha \! K ^{s/2} \phi \rangle
_{L^2(M)}
\]
Thus, $\Sob ^{-s}_\alpha (M)$ is the dual Hilbert space of $\Sob
^{s}_\alpha (M)$.  The obvious inequalities \footnote{Using Jensen
  inequality (or, less prosaically, the inequality $(x_1x_2\ldots
  x_g)^{1/g} \leq (x_1+x_2+\cdots +x_g)/g$ between geometric and the
  aritmetic means of nonnegative numbers $x_k$), one sees that
  \[
  (1+\Delta _1) (1+\Delta _{2}) \ldots (1+\Delta _g) \leq \left(
    1+\frac{1}{g}\Delta \right) ^g
  \]
  If we redefine the Laplacian as $\Delta = \frac{1}{g} \left(
    \Delta_1+\dots+\Delta_g \right)$, and then $K = \left(
    (1+\Delta_1) \dots (1+\Delta_g) \right) ^{1/g}$, the inequalities
  read $(1+\Delta)^{1/g} \leq K \leq (1+\Delta) $ } $1+ \Delta \leq K
\leq (1+\Delta ) ^g$ imply the inequalities

\begin{equation} \label{Sob_inclusions} \| \varphi \| _{\Sobolev
    _\alpha ^s(M)} \leq \| \varphi \| _{\Sob _\alpha ^s (M) } \leq \|
  \varphi \| _{\Sobolev _\alpha^{sg}(M)} \, ,
\end{equation}
so that we have the inclusions $\Sobolev _{\alpha }^{sg} (M) \subset
\Sob _{\alpha} ^s (M) \subset \Sobolev _{\alpha }^s(M) $ for $s \geq
0$. Observe that as topological vector spaces all the $\Sob ^s_\alpha
(M)$, for different $\alpha \in \Auto$, are the same, in particular
equal to the space $\Sob ^s(M)$ corresponding to the identity
automorphism $\alpha = 1$. What changes is just the norm!

\paragraph{Observation.} The operator $K$ is a sort of geometric mean
between the partial Laplacians $1+\Delta_k$. Inequality $K \leq
(1+\Delta)^g$ is a trivial instance of the inequality between the
geometric and the arithmetic mean (or of Jensen inequality).  The
other inequality $1+\Delta \leq K$ is trivial, but cannot be improved
to something like $(1+\Delta)^g \leq C \cdot K$ , which would yeld $\|
\varphi \| _{\Sob^s_\alpha(M)} \asymp \| \varphi \|
_{\Sobolev^{sg}_\alpha(M)}$, hence no exponential loss of derivatives
in the inductive argument used to solve the cohomological equation in
the Schr\"odinger representation!  This is clear taking a sequence of
Hermite eigenfunctions $\phi_a$ with $a_1=\dots = a_{g-1}=0$ and $a_g
=n \to \infty$. Then $(1+\Delta)^g \phi _a \sim n^g \phi_a$ while
$K\phi_a \sim n \phi_a$.



\section{Equidistribution}

In this section we consider only functional spaces ``built up'' from
the space of functions with zero  average along the fibers of the central
fibration of  the standard nilmanifold $\M$.  Thus, all smooth forms
have coefficients in $C_0^\infty(\M)$, all Sobolev forms and currents 
have coefficients in some $W^s_\alpha(\M)$, $s\in \R$ (see
definition~\ref{eq:sect4:4}).

\subsection{Birkhoff sums and renormalization} \label{sect:5.1}

Let $ (X_1^0,\dots , X_g^0,\Xi_1^0,\dots, \Xi_g^0,T)$ be the
``standard'' Heisenberg basis defined in section~\ref{par:heis_not}.

For $1\leq d \leq g$, we define the sub-algebra ${\mathfrak p}^{d,0}
\subset \hei$ generated by the first $d$ base elements $ 
X_1^0, \dots,  X_d^0$, and then the Abelian subgroup $
\Subgroup^{d,0} := \exp {\mathfrak p}^{d,0} $.

According to \eqref{eq:sect4:1}, we let $\Auto$ acts on the right on
subgroups,  and,  for $\alpha\in \Auto$, and we set
$(X^\alpha_i,\Xi_j^\alpha,T):=
(\alpha^{-1}(X_i^0),\alpha^{-1}(\Xi_j^0),T)$. 
Then  ${\mathfrak p}^{d,\alpha} := \alpha^{-1}({\mathfrak
  p}^{d,0})$ and $ \Subgroup^{d,\alpha} = \alpha^{-1}(\Subgroup^{d,0})
$ are respectively the algebra and the subgroup generated by
$(X^\alpha_i,\Xi_j^\alpha,T)$.  Every \isotropic subgroup of $\Heis$
is obtained in this way, i.e. given by some $\Subgroup^{d,\alpha} $
defined as above.

It is immediate that for every $\alpha,\beta\in \Auto$ we have
\[
\alpha^{-1}(\Subgroup^{d,\beta})= \Subgroup^{d,\beta\alpha};
\]
in particular, if $\beta$ belongs to the diagonal Cartan subgroup
$\Cartan$, then $ \Subgroup^{d,\beta\alpha}=\Subgroup^{d,\alpha}$.

We define a parametrization of $\Subgroup^{d,\alpha}$, hence a
$\R^d$-action on $\M$ subordinate to $\alpha$, by setting
\begin{equation}
  \label{eq:sect5:1}
  \Subgroup_x^{d,\alpha} := \exp (x_1 X_1^\alpha+\dots+x_d X_d^\alpha) \quad \text{ with }
  x=(x_1,\dots,x_d) \in \R^d.
\end{equation}


\paragraph{Birkhoff averages. } 
We define the bundle $A^j(\mathfrak p^d,\Bundle^{ s})\to \Mod$ of
$\mathfrak p$-forms of degree~$j$ and Sobolev order $s$ as the set of
pairs
\[
(\alpha , \omega), \quad \alpha \in \Auto, \quad \omega\in
A^j(\mathfrak p^{d,\alpha}, W^{s}_\alpha(\M)),
\]
modulo the equivalence relation $ (\alpha , \omega) \equiv
(\alpha\gamma ,\gamma^*\omega)$ for all $ \gamma \in \GMod $. The
class of $(\alpha , \omega)$ is denoted $[\alpha , \omega]$.  We also
define the dual bundle $A_j(\mathfrak p^d,\Bundle^{ -s})\to \Mod$ of
$\mathfrak p$-current of dimension~$j$ and Sobolev order $s$ as the
set of pairs
\[
(\alpha , \mathcal D), \quad \alpha \in \Auto, \quad \mathcal D\in
A_j(\mathfrak p^{d,\alpha}, W^{-s}_\alpha(\M)),
\]
modulo the equivalence relation $ (\alpha , \mathcal D) \equiv
(\alpha\gamma ,(\gamma_*)^{-1} \mathcal D)$ for all $ \gamma \in \GMod
$. The
class of $(\alpha , \mathcal D)$ is denoted $[\alpha , \mathcal D]$.

The bundles $A^j(\mathfrak p,\Bundle^{s})$ and $A_j(\mathfrak
p,\Bundle^{-s})$ are Hilbert bundles for the dual norms
\[
\| \, [\alpha , \omega ]\,\| _s := \| \omega \|_{s,\alpha},\qquad\| \,
[\alpha , \mathcal D]\,\| _{-s} := \| \mathcal D \|_{-s,\alpha}.
\]
In the following, it will be convenient to set $\omega^{d,\alpha}=
dX_1^\alpha \wedge \dots \wedge dX_d^\alpha$ and to identify
top-dimensional currents $\mathcal D$ with distributions by setting
$\lin \mathcal D, f\rin := \lin \mathcal D, f\omega^{d,\alpha}\rin$.

Given a Jordan region $U \subset \R^d$ and a point $m \in \M$, we
define a top-dimensional $\mathfrak p$-current
$\Birkhoff_U^{d,\alpha}m$ as the Birkhoff sums given by integration
along the chain $\Subgroup_U^{d,\alpha}m=\{\Subgroup^{d,\alpha}_x
m\mid x\in U\}$. Explicitely,  if $\omega = f \D X_1^\alpha \wedge \dots
\wedge \D X_d^\alpha$ is a top-dimensional $\mathfrak p$-form, then 
\begin{equation} \label{Birkhoff} \lin \Birkhoff^{d,\alpha}_{U}m,
  \omega \rin := \int _{\Subgroup^{d,\alpha}_Um} \omega= \int _{U}f (
  \Subgroup^{d,\alpha}_x m ) \, \D x_1\dots\D x_d.
\end{equation} 
Our goal is to understand the asymptotic of these distributions as $U
\nearrow \R^d$ in a F{\o}lner sense.  A particular case is obtained
when $U= Q(\radius) = [0,T]^d$.

We remark that the Birkhoff sums satisfy the following covariance property: 
\[
\gamma_*^{-1}\left( \Birkhoff ^{d,\alpha}_U {m} \right) = \Birkhoff
^{d,\alpha\gamma}_U (\gamma^{-1} m) ,\qquad \forall m\in \M,
\forall\gamma \in \GMod.
\]

\paragraph{Renormalization flows.} 
For each $1\leq i \leq g$, we denote by $\dcart_i := \left( \begin{smallmatrix} \delta_i & 0 \\ 0 & -
    \delta_i \end{smallmatrix} \right) \in \mathfrak a $ the
element of the Cartan subalgebra of diagonal symplectic defined by the diagonal matrix   $\delta_i =
\mathrm{diag}(d_1,\dots,d_g)$  with $d_i=1$
and $d_k=0$  if $k\neq i$. Any such $\dcart _i$ generates  a one
parameter group of automorphisms $r_i ^t := e^{t\dcart_i} \in
\Cartan$, with $t \in \R$.

Left multiplication by the one parameter group $(r_i ^t)$ yields a
flow on $\Auto$ that projects to moduli space $\Mod$ according to  $[
\alpha] \mapsto r_i ^t[\alpha]= [r_i ^t\alpha] $.

Above this flow, we consider its horizontal lift to the bundles
$A^j(\mathfrak p^d,\Bundle^{s})$ and $A_j(\mathfrak p^d,\Bundle^{-s})$
($s\in \R$), defined by
\[
r_i ^t [ \alpha , \omega ] := [r_i ^t \alpha , \omega ] \qquad r_i ^t
[ \alpha , \mathcal D] := [r_i ^t \alpha , \mathcal D ]
\]
for $\alpha\in \Auto$ and $ \omega \in A^j(\mathfrak
p^{d,\alpha},\Bundle^{ s})$ or $\mathcal D \in A_j(\mathfrak
p^{d,\alpha},\Bundle^{ -s})$. This is well defined since, as we
remarked before, $\mathfrak p^{d,\alpha}=\mathfrak p^{d,r_i
  ^t\alpha}$.

\begin{definition}
  \label{def:sect5:1}
  For $s>0$, let $\closedbundle{d}{d}{-s}$
  be the sub-bundle of the bundle $\currentbundle{d}{d}{-s}$ 
  consisting of elements $[\alpha, \mathcal D]$ with $ \mathcal D \in
  Z_d(\mathfrak p^{d,\alpha}, W^{-s}_\alpha(\M))$, i.e.  with $\mathcal D$ a closed
  $\mathfrak p^{d,\alpha}$-current of dimension~$d$ and Sobolev order
  $s$.
\end{definition} 

We remark that the definition is well posed. In fact, if $ \mathcal D$
is a closed $\mathfrak p^{d,\alpha}$-current of dimension~$d$ then,
from the identities $\lin \mathcal D, X^\alpha_i(f)\rin =0 $ for all
test functions $f$ and $i\in [1,d]$, we obtain $0 = \lin
\gamma_*\mathcal D, \gamma_* X^\alpha_i(f)\rin =\lin \gamma_*\mathcal
D, X^{\alpha\gamma^{-1}}_i(f)\rin$, which shows that $\gamma_*\mathcal
D$ is a closed $\mathfrak p^{d,\alpha\gamma^{-1}}$-current of
dimension~$d$.

Observe that, although the subgroup $\Subgroup^{d, (r_i^t\alpha)}$ and
$\Subgroup^{d, \alpha}$ coincide, the actions of $\R^d$ defined by
their parametrizations~\eqref{eq:sect5:1} differ by a constant
rescaling; in fact
\begin{equation}
  \label{eq:sect5:3}
  \Subgroup _{(x_1,\dots ,x_d)}^{d, (r_1^{t_1} \dots r_g^{t_g}\alpha)}=
  \Subgroup_{(e^{-t_1} x_1, \dots,e^{-t_d}x_d) }^{d, \alpha}.
\end{equation}
Consequently, denoting by $(e^{-t_1}, \dots, e^{-t_d})U$ the obvious
diagonal automorphism of $\R^d$ applied to the region $U$, the
Birkhoff sums satify the identities
\begin{equation}
  \label{eq:sect5:2}
  \Birkhoff _U^{d, (r_1^{t_1} \dots r_g^{t_g}\alpha)} m =
  e^{t_1+\dots +t_d} \,   \Birkhoff _{(e^{-t_1}, \dots, e^{-t_d})U}^{d, \alpha} m.
\end{equation}

\begin{proposition}
  \label{Lyap_renorm_inv_dist_bis}
  Let $s > d/2$.  The sub-bundle $\closedbundle{d}{d}{-s}$ is
  invariant under the renormalization flows $r_i^t$ with $1\leq i \leq d$. Furthermore, for
  every $(t_1,\dots,t_d)\in \R^d$ and any $ [\alpha, \mathcal D]\in
  \closedbundle{d}{d}{-s}$ and any $s> d/2$, we have
  \[
  \big\| \, r_1^{t_1} \dots r_d^{t_d} [\alpha, \mathcal D]\,
  \big\|_{-s} = e^{-(t_1+\dots +t_d)/2} \, \big\| \,[\alpha, \mathcal
  D] \,\big\|_{-s} .
  \]
\end{proposition}

\begin{proof} The invariance of the sub-bundle
  $\closedbundle{d}{d}{-s}$ is clear from~\eqref{eq:sect5:3}.

  Set, for simplicity, $r:=r_1^{t_1} \dots r_d^{t_d}$.  By definition
  $ \big\|r [\alpha, \mathcal D]\, \big\|_{-s}= \big\| \,[r\alpha,
  \mathcal D]\, \big\|_{-s}= \| \mathcal D \|_{-s,r \alpha}$ for any $
  [\alpha, \mathcal D]\in \currentbundle{d}{d}{-s}$.
    
  \begin{sloppypar}
      Without loss of generality we may assume that $\mathcal D$ belongs
  to the space $ A_d(\mathfrak p^{d, \alpha}, W^{-s}(\rho_h))$, where
  $\rho_h$ is an irreducible Schr\" odinger reprentation in which the
  basis $(X_i^\alpha, \Xi^\alpha, T)$ acts according to
  \eqref{eq:sect2:6}.  Let $\tilde L_\alpha = (\rho_h)_* L_\alpha$ and
  ${\widetilde L}_{r_d^t\alpha}= (\rho_h)_* {\widetilde
    L}_{r_d^t\alpha}$ the push-forward to $L^2(\R^g)$ of the operators
  defining the norms $\| \cdot\|_{s,\alpha}$ and $\|
  \cdot\|_{s,r_d^t\alpha}$.
  \end{sloppypar}

  By Proposition \ref{prop:sect:3:12}, the space of closed currents of
  dimension $d$ is spanned by $\Int_g$, if $d=g$, and by  the dense
  set of currents $\mathcal D = D_y \circ \Int_{d,g} $ with $D_y \in
  L^2(\R^{g-d}, dy)$, if $d < g$. Any such current is given, for any
  test function $f\in \mathcal S(\R^g)$, by $\lin \mathcal D , f \rin
  = \lin D_y , \int _{\R^d} f (x,y) \, dx \rin$. The unitary operator
  $U_t: L^2(\R^g) \to L^2(\R^g)$ defined, for $t = (t_1,\dots,t_d)$,
  by\footnote{This is a particular case of the {\em metaplectic
      representation}.  (See \cite{MR0165033, MR983366}).}
  \begin{equation} \label{intertwiner} U_t f(x,y) := e^{-(t_1+\dots
      +t_d)/2}\,f\big( (e^{t_1}, \dots, e^{t_d}) x,y\big)
  \end{equation} 
  ($x\in \R^d$, $y\in \R^{g-d}$), intertwines the differential
  operator ${\widetilde L}_\alpha$ with the operator ${\widetilde
    L}_{r \alpha}$, i.e.\ $U_t({\widetilde L}_\alpha f)= {\widetilde
    L}_{r\alpha}U_t f$ for any smooth $f$. Thus
  \begin{eqnarray*}
    \|\mathcal D   \| _{-s,  r \alpha}   
    & = & 
    \sup _{\| f  \| _{s, r \alpha } =1} 
    \left|  \lin  \mathcal D  , f \rin \right| =
    \sup _{\|  {\widetilde L}_{r \alpha}^{s/2} f \|  =1}  \left|  \lin
      \mathcal D  , f \rin \right|\\
    & = &
    \sup _{\|  {\widetilde L}_{\alpha}^{s/2} U_t^{-1}f \|  =1}  \left|  \lin
      \mathcal D  , f \rin \right| = \sup _{\|  {\widetilde L}_{\alpha}^{s/2} f \|  =1}  \left|  \lin
      \mathcal D  , U_t f \rin \right|
    \\
    & = & 
    \sup _{\|  (L_{\alpha} ) ^{s/2} f \|  =1} 
    \left|  \lin D_y , \int _{\R^g}e^{-(t_1+\dots +t_d)/2}\,f\big( (e^{t_1}, \dots, e^{t_d}) x,y\big)  \, dx \rin   \right|  \\
    & = & 
    \sup _{\|   (L_{\alpha}) ^{s/2} f \|  =1} e^{-(t_1+\dots +t_d)/2} 
    \left| \lin D_y ,  \int _{\R^g} f (x,y)   \, dx \rin  \right|  \\
    & = &  e^{-(t_1+\dots +t_d)/2}  \, \| \mathcal D  \| _{-s, \alpha } 
  \end{eqnarray*}
\end{proof}

\subsection{The renormalization argument}
\label{subSection:renormalization_argument}

\paragraph{Orthogonal splittings. }

For any exponent $s>d/2$, the sub-bundle $\closedbundle{d}{d}{-s}$ is
a closed subspace of the Hilbert bundle $\currentbundle{d}{d}{-s}$ 
and therefore induces an orthogonal decomposition
\begin{equation}
  \label{orthogonal_decomposition} 
  \currentbundle{d}{d}{-s}  = \closedbundle{d}{d}{-s} \oplus  \restbundle{d}{d}{-s} \,.
\end{equation}
where $\restbundle{d}{d}{-s} := \closedbundle{d}{d}{-s}^\perp$.  We
denote by $\mathcal Z^{-s}$ and $\mathcal R^{-s}$ the corresponding
orthogonal projections, and, given $\alpha \in \Symp_{2g}(\R)$, by
$\mathcal Z^{-s}_{\alpha}$ and $\mathcal R^{-s}_\alpha$ the
restrictions of these projections to the fiber over $[\alpha] \in
\Mod$.
In particular, we obtain a decomposition of the Birkhoff averages
$\mathcal D = \Birkhoff ^{d,\alpha}_{U} m$ as
\begin{equation}
  \label{orthogonal_decomposition_Birkhoff_bundle} 
  \begin{split} 
    [\alpha , \mathcal D] & = \mathcal Z^{-s} [\alpha , \mathcal D] + \mathcal R ^{-s} [\alpha , \mathcal D ] \\
    & = [\alpha , \mathcal Z ^{-s}_\alpha (\mathcal D)] + [\alpha ,
    \mathcal R ^{-s}_\alpha (\mathcal D) ]
  \end{split} \end{equation} with ``boundary term'' $ \mathcal
Z^{-s}_\alpha (\mathcal D) \in Z_d(\mathfrak p^{d,\alpha},
W^{-s}_\alpha(\M))$ and ``remainder term'' $ \mathcal R^{-s}_\alpha
(\mathcal D) \in R_d(\mathfrak p^{d,\alpha}, W^{-s}_\alpha(\M)) $.

We will also need an estimate for the distortion of the Sobolev norms
along the renormalization flow.  Below, $|t| $ denotes the sup norm of
a vector $t \in \R^d$.

\begin{lemma} \label{small_projection} Let $s > d/2+2$.  For
  $t=(t_1,\dots, t_d) \in \R^d$ and $\tau \in \R$, let $r^\tau =
  r^{-\tau t_1}_1\dots r^{-\tau t_d}_d $.  There exists a constant
  $C=C(s)$ such that if $|\tau t|$ is sufficiently small then the
  orthogonal projection
  \[
  \mathcal Z^{-s} _{r^{\tau}\alpha} : R_d(\mathfrak p^{d,\alpha},
  W^{-(s-2)}_{\alpha}(\M)) \to Z_d(\mathfrak p^{d,\alpha},
  W^{-s}_{r^\tau \alpha}(\M))
  \]
  has norm bounded by $C \,|\tau t|$.
\end{lemma}

\begin{proof} As in the proof of Proposition
  \ref{Lyap_renorm_inv_dist_bis}, we may restrict to a fixed Schr\"
  odinger representation $\rho _h$ in which the basis $(X_i^\alpha,
  \Xi_i^\alpha, T)$ acts according to \eqref{eq:sect2:6}.
  It is also clear from Lemma
  \ref{lem:prop_cohom_eq_uniform_h_Sobolev} that we may use the
  homogeneous Sobolev norm defined in \eqref{eq:sect3:3b}.  If
  $H_g=(\rho_h)_\ast {L}_\alpha$ denotes the sub-Laplacian inducing
  the Sobolev structure of $W^{-s}_\alpha (\R^g)$, then the Sobolev
  structure of $W^{-s} _{r^\tau \alpha} (\R^g)$ is induced by
  \[
  H_\tau = U'_{-\tau} H U'_\tau
  \]
  where $U'_\tau = U_{\tau t} $ is the one-parameter group of unitary
  operators of $L^2(\R^g)$ defined according to
  \eqref{intertwiner}. We denote by $\lin \phi, \psi \rin_{-s,\tau} =
  \lin \phi, H_\tau^{-s} \psi \rin$ the inner product in $W^{-s}
  _{r^\tau \alpha} (\R^g)$.  A computation shows that the
  infinitesimal generator of $U'_\tau$ is $i$ times the self-adjoint
  operator $A = (\rho_h)_\ast \left( \sum_{k=1}^d t_k (1/2 - X_k\Xi
    _k) \right)$.  Moreover, using the Hermite basis, one can show
  that there exists a constant $C$ such that $\| A \psi \| \leq C |t|
  \, \| H \psi \|$ for $\psi$ in the domain of $A$.

  Now, let $\mathcal R \in W^{-s+2}_\alpha (\R^g)$ be a distribution
  (we identify top-dimensional currents with distributions as
  explained in \ref{sect:5.1}) which is orthogonal to the subspace $Z$
  of closed distributions when $\tau=0$, i.e. such that
  \[
  \lin \mathcal R , \mathcal D \rin _{-s,0} = \lin \mathcal R , H^{-s}
  \mathcal D \rin = 0
  \]
  for all $\mathcal D \in Z$.  In order to bound the norm of its
  projection to $Z$ w.r.t. the Sobolev structure at $\tau$ we must
  bound the absolute values of the scalar products $ \lin \mathcal R ,
  \mathcal D \rin _{-s,\tau}$ for all $\mathcal D$ in $Z$. Now,
  \[
  \begin{split}
    \lin \mathcal R , \mathcal D \rin _{-s,\tau} & = 
    \lin \mathcal R, U'_{-\tau} H^{-s} U'_{\tau} \mathcal D \rin \\
    & = \lin U'_{\tau} \mathcal R, H^{-s} U'_{\tau} \mathcal D \rin
  \end{split}
  \]
  If $\mathcal R$ is in the domain of $A$, we may write
  \[
  U'_\tau \mathcal R = \mathcal R+ i \int_0^\tau U'_{u} A \mathcal R
  \, du
  \]
  According to Proposition \ref{Lyap_renorm_inv_dist_bis}, the group
  $U'_\tau$ preserves $Z$. Therefore, since $\mathcal R$ is orthogonal
  to $U'_{\tau}\mathcal D$ for all $\tau$, we may write
  \[
  \begin{split}
    \lin \mathcal R , \mathcal D\rin _{-s,\tau}
    &=i \int_0^\tau  \lin U'_{u} A \mathcal R, H^{-s} U'_{\tau} \mathcal D \rin \, du \\
    & = i\int_0^\tau  \lin  A \mathcal R, U'_{-u} H^{-s} U'_{\tau} \mathcal D \rin \, du \\
    &=i \int_0^\tau  \lin  A \mathcal R,  U'_{\tau-u} \mathcal D \rin _{-s,u} \, du \\
  \end{split}
  \]
  By Cauchy-Schwartz and Lemma \ref{lem:sect4:1}, if $|\tau t|$ is
  sufficiently small we have
  \[
  \begin{split}
    \left| \lin \mathcal R , \mathcal D \rin _{-s,\tau} \right| & \leq
    \left| \int_0^\tau \| A\mathcal R \|_{-s,u} \,
      \|   U'_{\tau-u} \mathcal D \| _{-s,u} \, du \right| \\
    & \leq C' \, \| A\mathcal R \|_{-s,0} \,
    \left| \int_0^\tau  \|   U'_{\tau-u} \mathcal D \| _{-s,u} \, du \right| \\
    & \leq C'' \,|t| \| \mathcal R \|_{-s+2,0} \, \left| \int_0^\tau
      \| U'_{\tau-u} \mathcal D \| _{-s,u} \, du \right|
  \end{split}
  \]
  But $\| U'_{\tau - u} \mathcal D \|_{-s,u} = \| \mathcal D
  \|_{-s,\tau}$.  There follows
  \[
  \left| \lin \mathcal R , \mathcal D \rin _{-s,t\tau} \right| \leq
  |\tau t| \, C'' \, \| \mathcal R \|_{-s+2,0} \, \| \mathcal D \|
  _{-s,\tau} \,
  \]
  This says that the orthogonal projection $Z_\tau (\mathcal R)$ of
  $\mathcal R$ onto $Z$ w.r. to the Sobolev structure at $\tau$ has
  norm
  \[
  \| Z_\tau (\mathcal R) \| _{-s,\tau} \leq | \tau t| \, C'' \, \|
  \mathcal R \| _{-s+2, 0}.
  \]
\end{proof}

\begin{notation} In order to shorten our formulas, in the proofs
of the following statements we drop the ``initial point'' $m\in M$ or
the automorphism $\alpha$ in the symbol $\Birkhoff ^{d,\alpha}_{U} m $
whenever the estimates are uniform in $m$, in $\alpha$ or both.
\end{notation}

From the Sobolev embedding theorem and the definition
\eqref{best_Sobolev_constant} of the  Best Sobolev Constant
 $B_s$  we have the following trivial bound.

\begin{lemma} \label{trivial_bound_Sobolev} For any Jordan region
  $U\subset \R^d$ with Lebesgue measure $|U|$,  for any $s> g+1/2$
  and all $m\in \M$ we have
  \[
  \left \| [\alpha, \Birkhoff ^{d,\alpha}_U m] \right\|_{-s} \leq
  B_{s} ([[\alpha]]) \, |U|.
  \]
\end{lemma}

For the remainder term we have the following estimate. Below, we
denote by $\partial \mathcal D$ the boundary of the current $\mathcal
D $, defined by $\lin \partial \mathcal D , \eta \rin = \lin \mathcal
D , \D \eta \rin$.

\begin{lemma}
  \label{rest_estimate}
  Let $s>g+d/2+1$. For any non-negative $s' < s-(d+1)/2$, there exists
  a constant $C=C(g,d,s',s) >0$ such that, for all $m\in \M$ and
  $\alpha \in \Auto$, we have
  \[
  \| \, \mathcal R^{-s} \, [\alpha, \Birkhoff_{U}^{d,\alpha} m ] \,\|
  _{-s} \leq {C} \, \| \, [\alpha,\partial ( \Birkhoff^{d,\alpha}_U m)
  ] \|_{-s'} \, .
  \]
\end{lemma}

\begin{proof}

  Let $\omega:[\alpha]\to \omega([\alpha])$ be a section of
  $\formbundle{d}{d}{s}$. Writing $\omega = \omega^s_Z + \omega^s_R$
  for its decomposition with $\omega^s_R$ in the annihilator of $
  \closedbundle{d}{d}{-s} $ and $\omega^s_Z$ in the annihilator of $
  \restbundle{d}{d}{-s}$, we have
  \[
  \lin \mathcal R ^{-s}_\alpha (\Birkhoff _{U}^{d,\alpha} ) , \omega
  \rin = \lin \mathcal R ^{-s}_\alpha (\Birkhoff _{U}^{d,\alpha} ) ,
  \omega^s_R \rin = \lin \Birkhoff_U^{d,\alpha} , \omega^s_R \rin \, .
  \]
  Since $s>(d+1)/2$ and since, by definition, $\lin T , \omega^s_R
  \rin = 0$ for any $T\in \closedbundle{d}{d}{-s} $, by Theorem
  \ref{thm:sect3:2}
  there exists a constant $C:= C(g,d,s',s)$ and a section of
  $(d-1)$-forms $\eta$ with $\D{}\eta = \omega^s_R$ and satisfying, for
  all $s' < s -(d+1)/2$, the estimate $ \|\eta([\alpha])\|_{s',\alpha}
  \le C \|\,\omega^s_R([\alpha]) \|_{s,\alpha}$ for all $\alpha$. It
  follows that
  \[
  \lin \Birkhoff_U^d , \omega^s_R \rin = \lin \partial \Birkhoff^d_U ,
  \eta \rin \, .
  \]
  Hence, for $s' < s -(d+1)/2$, for all $m\in \M$ and $\alpha \in
  \Auto$, we have
  \[
  | \lin \Birkhoff_U^d , \omega^s_R \rin | \le C \|\partial
  \Birkhoff^d_U\|_{-s'} \times \|\,\omega^s_R \|_{s}\le C \|\partial
  \Birkhoff^d_U\|_{-s'} \times \|\,\omega \|_{s}.
  \]
\end{proof}


To estimate the boundary term, we need the following recursive
estimate.

\begin{lemma} \label{recursive_inequality} Let $s > d/2+2$. There
  exists a constant $C_1=C_1(s)>0$ such that for all $t_1\ge 0$, \dots,
  $t_d\ge 0$ and all $[\alpha , \mathcal D] \in
  \currentbundle{d}{d}{-(s-2)}$ we have
  \[
  \begin{split}
    \| \, &\mathcal Z^{-s} [\alpha, \mathcal D] \, \| _{-s} \leq \, \,
    \, e^{-(t_1+\dots+t_d)/{2} } \,\| \, \mathcal Z^{-s}
    [r^{-t_1}_1\dots r^{-t_d}_d\alpha,  \mathcal D ] \,   \| _{-s} \\
    & + C_1\, |t_1+\dots+t_d| \int_{0}^{1}
    e^{-{u}(t_1+\dots+t_d)/{2} } \| \, \mathcal R^{-s} [
    r^{-ut_1}_1\dots r^{-ut_d}_d\alpha, \mathcal D ] \, \| _{-(s-2)}\,
    \D{u}.
  \end{split}
  \]
\end{lemma}

\begin{proof}
  Set for simplicity $r^u = r^{-ut_1}_1\dots r^{-ut_d}_d $ and
  $t=t_1+\dots+t_d$. Consider the orthogonal decomposition
  \[
  \mathcal D = \mathcal Z^{-s}_{r^{-u}\alpha} (\mathcal D) + \mathcal
  R^{-s}_{r^{-u}\alpha} (\mathcal D), \quad u\in [0,1].
  \]
  If we apply the projection $\mathcal Z^{-s}_{r^{\tau-u}\alpha}$,
  since, by Proposition~\ref{Lyap_renorm_inv_dist_bis}, $\mathcal
  Z^{-s}_{r^{\tau-u}\alpha} \mathcal Z^{-s}_{r^{-u}\alpha} (\mathcal
  D)= \mathcal Z^{-s}_{r^{-u}\alpha} (\mathcal D)$, we get
  \[
  \mathcal Z^{-s}_{r^{\tau-u}\alpha} (\mathcal D) = \mathcal
  Z^{-s}_{r^{-u}\alpha} (\mathcal D) + \mathcal
  Z^{-s}_{r^{\tau-u}\alpha} ( \mathcal R^{-s}_{r^{-u}\alpha} (\mathcal
  D) )
  \]
  and therefore we may write
  \[
  \begin{split}
    [r^{\tau-u}\alpha, \mathcal Z^{-s}_{r^{\tau-u}\alpha} (\mathcal D)
    ] & = [r^{\tau-u}\alpha, \mathcal Z^{-s}_{r^{-u}\alpha} (\mathcal
    D) ]+ [r^{\tau-u}\alpha, \mathcal Z^{-s}_{r^{\tau-u}\alpha}
    ( \mathcal R^{-s}_{r^{-t}\alpha} (\mathcal D) ) ] \\
    & = r^{\tau} \, \mathcal Z^{-s} [r^{-u}\alpha, \mathcal D ]+
    \mathcal Z^{-s} [r^{\tau-u}\alpha, \mathcal R^{-s}_{r^{-u}\alpha}
    (\mathcal D) ]
  \end{split}
  \]
  Now, we compute the norm with exponent $-s$.  By
  Proposition~\ref{Lyap_renorm_inv_dist_bis}, the first term on the
  right has norm
  \[
  \| r^{\tau} \, \mathcal Z^{-s} [r^{-u}\alpha, \mathcal D ] \|_{-s} =
  e^{-\tfrac{t}{2} \tau } \| \mathcal Z^{-s} [r^{-u}\alpha, \mathcal D
  ] \|_{-s} \, .
  \]
  To estimate the norm of the second term on the right, we observe
  that $\mathcal Z^{-s}_{r^{\tau -u}}$ is an orthogonal projection,
  and that by Lemma~\ref{small_projection} the projection
  \[ R_d(\mathfrak p^{d,\alpha}, W^{-(s-2)}_{r^{\tau-u}\alpha}(\M))
  \to Z_d(\mathfrak p^{d,\alpha}, W^{-s}_{r^{-u}\alpha}(\M))
  \]
  has norm bounded by $C(s) \,t\,\tau$. Therefore
  \[
  \begin{split}
    \| \mathcal Z^{-s}[ {r^{\tau-u}\alpha} , \mathcal D] \|_{-s} &\leq
    e^{-\tfrac{t}{2}\tau} \, \| \mathcal Z^{-s} [ {r^{-u}\alpha} ,
    \mathcal D] \|
    _{-s} \\
    &\qquad+ C(s)\,t\,\tau \, \| \mathcal R^{-s} [{r^{-u}\alpha} ,
    \mathcal D) \| _{-(s-2)}\,.
  \end{split}
  \]
  Let $n\in \N^+$, and set $\tau =1/n$, $u=k \tau$, with $k\in \N\cap
  [0,n]$.  By finite induction on $k$ we obtain
  \[
  \begin{split}
    \| \, \mathcal Z^{-s} [\alpha, \mathcal D] \, \| _{-s} \leq
    & \, \, \, e^{-\frac{t}{2} }  \| \, \mathcal Z^{-s} [r^{-1}\alpha,  \mathcal D ] \,   \| _{-s} \\
    & + C(s)\,\frac t n \sum_{k=1}^{n} e^{-\frac{ tk}{2n} } \| \,
    \mathcal R^{-s} [r^{- k/n}\alpha, \mathcal D ] \, \| _{-(s-2)}\,.
  \end{split}
  \]
  The Lemma follows by taking the limit as $n\to\infty$.
\end{proof}

Next, we consider the case $d =1$.

\begin{theorem}
  \label{result_one_dim}
  Let $\alpha \in \Symp_{2g}(\R)$ and $s>g+7/2$. Let
  $\Subgroup^{1,\alpha}$ be the $1$-dimensional abelian subgroup of
  $\Heis$ generated by the base vector field $X_1^\alpha\in \hei$. 
  Let $U_T=[0,T]$ and $\Birkhoff_{U_T}^{1,\alpha}
  m$ the Birkhoff sum associated to some $m\in \M$ for the action of 
  $\Subgroup^{1,\alpha}_{x}$ $(x\in \R)$.  There exist a constant
  $C_2=C_2(s)>0$ such that for all $\radius \ge 1$ and all $m\in \M$ we
  have
  \[
  \begin{split}
    \left\| \, [\alpha, \Birkhoff _{U_T}^{1,\alpha} m ] \, \right\|
    _{-s} \leq
    & \quad C_2 \, T^{1/2}  \, \Height \left( \class{r^{-\log T}_1 \alpha} \right) ^{1/4} \\
    & \quad + C_2 \, \int_0^{\log T} e^{ u /2} \, \Height \left(
      \class{r^{-u}_1 \alpha} \right) ^{1/4} \, \D{ u}.
  \end{split}
  \]
\end{theorem}

\begin{proof}
  For simplicity we set $r^t=r_1^t$. To start, we observe that,
  according to \eqref{eq:sect5:2} and Lemma \ref{rest_estimate}, we
  have
  \[
  \begin{split} \| \, \mathcal R^{-s} [r^{-t} \alpha ,
    \Birkhoff^{1,\alpha}_{U_{e^tT}} ] \, \| _{-(s-2)}
    & = e^t \, \| \, \mathcal R^{-s} [r^{-t} \alpha ,\Birkhoff ^{1, r^{-t}\alpha}_{U_{T}}  ] \, \| _{-(s-2)} \\
    & \le e^t \, \big\| \, [r^{-t}\alpha, \partial ( \Birkhoff ^{1,
      r^{-t}\alpha}_{U_{T}} ) ] \,\big\|_{-s'} \end{split}
  \]
  provided $g+1/2<s'<s-3$.  The boundary $\partial (\Birkhoff ^{1,
    r^{-t}\alpha}_{U_{T}})$ is a 0-dimensional current given by
  \[
  \lin \partial (\Birkhoff ^{1, r^{-t}\alpha}_{U_{T}} , f \rin =
  f(\Subgroup^{r^{-t}\alpha}_{T}(m)) - f(m) \, ,
  \]
  hence, by the Sobolev embedding theorem and the definition
  \eqref{best_Sobolev_constant} of the Best Sobolev Constant, we have
  \[
  \big\| \, [r^{-t} \alpha, \partial (\Birkhoff ^{1,
    r^{-t}\alpha}_{U_{T}})] \,\big\|_{-s'} \le 2 \, B_{s'}( \class{
    r^{-t}\alpha }) \, .
  \]
  There follows from Proposition \ref{best_Sobolev_constant2} that
  \[
  \big\| \, \mathcal R^{-s} [r^{-t}\alpha,
  \Birkhoff^{1,\alpha}_{U_{e^tT}} ] \,\big\|_{-(s-2)} \le 2 \,e^t \,
  B_{s'}( \class{r^{-t} \alpha}) \le C(s') \, e^t \, \Height
  (\class{r^{-t} \alpha })^{1/4} \, .
  \]
  Using Lemma \ref{recursive_inequality} with $\mathcal D =
  \Birkhoff^{1,\alpha}_{U_{e^tT}}m$ and $t = n \step$, we may estimate
  the boundary term in the decomposition
  \eqref{orthogonal_decomposition_Birkhoff_bundle} as
  \[
  \begin{split}
    \left\|\, \mathcal Z^{-s} [ \alpha,
      \Birkhoff_{U_{e^tT}}^{1,\alpha} ]\, \right\| _{-s} \leq & \,
    e^{-t/2 } \left\|\, \mathcal Z^{-s} [r^{-t} \alpha,
      \Birkhoff_{U_{e^tT}}^{1,\alpha} ]\, \right\| _{-s} \\ & \, +
    C(s,s') \, \int_{0}^{t} e^{ u /2} \Height (\class{r^{-u} \alpha})
    ^{1/4}\,\D{u} \, .
  \end{split}
  \]
  By the covariance \eqref{eq:sect5:2}, Proposition
  \ref{best_Sobolev_constant2} and Lemma \ref{trivial_bound_Sobolev},
  we have
  \[
  \begin{split}
    \left\|\, \mathcal Z^{-s} [ r^{-t} \alpha, \Birkhoff
      _{U_{e^tT}}^{1,\alpha} ]\, \right\| _{-s} & =
    e^{t} \, \left\|\, \mathcal Z^{-s}  [ r^{-t} \alpha, \Birkhoff_{U_{T}}^{1,r^{-t} \alpha}  ]\, \right\|  _{-s} \\
    & \leq e^{t} \, C(s) \, T \, \Height( \class{r^{-t} \alpha} )
    ^{1/4} \, .
  \end{split}
  \]
  There follows that
  \[
  \begin{split}
    \left\|\, \mathcal Z^{-s} [ \alpha,
      \Birkhoff_{U_{e^tT}}^{1,\alpha} ]\, \right\| _{-s} & \leq
    e^{t/2} \, C(s) \, T \, \Height( \class{r^{-t}_1 \alpha} ) ^{1/4} \\
    & \quad + C(s,s') \, \int_{0}^{t} e^{ u /2} \Height (\class{r^{-u}
      \alpha}) ^{1/4}\,\D{u}\,.
  \end{split}
  \]
  If we take first $\radius=1$, then rename $e^{t} := \radius \ge 1$,
  we finally get
  \[
  \begin{split} \left\| \, \mathcal Z^{-s} [\alpha, \Birkhoff
      _{U_T}^{1,\alpha} m ] \, \right\| _{-s} \leq
    & \quad C (s)\, T^{1/2}  \, \Height ( \class{r^{-\log T} \alpha} ) ^{1/4} \\
    & \quad + C(s,s') \, \int_0^{\log T} e^{ t /2} \, \Height (
    \class{r^{-t} \alpha} ) ^{1/4} \, \D{t} \, .
  \end{split}
  \]
  The reminder term in the decomposition
  \eqref{orthogonal_decomposition_Birkhoff_bundle} is estimated as at
  the beginning of the proof, using Lemma \ref{rest_estimate},
  Proposition \ref{best_Sobolev_constant2} and Lemma
  \ref{trivial_Diophantine_bound}, and is bounded by
  \[
  \begin{split} \left\| \, \mathcal R^{-s} [\alpha, \Birkhoff
      _{U_T}^{1,\alpha} ] \, \right\| _{-s} & \leq C(s) \,
    \Height([\alpha]) ^{1/4} \\ & = C(s) \, \Height( \class{ r^{ \log
        \radius} r^{ -\log\radius} \alpha} ) ^{1/4} \\ & \leq C(s) \,
    \radius ^{1/2} \, \Height(\class{r^{ -\log\radius} \alpha}) ^{1/4}
    \, .
  \end{split}
  \]
  The Theorem follows.
\end{proof}

The next result follows immediatly from the above Theorem
\ref{result_one_dim} and the Kleinbock-Margulis logarithm law, i.e
from Proposition \ref{Sobolev_constant_estimate}.

\begin{proposition} \label{result_one_dim_cor} Let the notation as in
  Theorem \ref{result_one_dim}.  There exists a full measure set
  $\Omega_g (\dcart_1)\subset \Mod$ such that if $[\alpha] \in \Omega
  _g(\dcart_1)$ and $\varepsilon >0$ there exists a constant $C =C
  (s,\varepsilon)>0$ such that for all $\radius \gg 1$ and all $m\in
  \M$ we have
  \[
  \left\| \, [\alpha, \Birkhoff _{U_T}^{1,\alpha} m] \, \right\| _{-s}
  \leq C \, {\radius^{1/2}} \, (\log \radius )^{1/(2g+2) +
    \varepsilon} .
  \]
\end{proposition}


Now we may use induction on the dimension of the isotropic group
$\Subgroup^d \subset \Heis$.  Let $(s_d)_{d\in \N}$ be the solution of
the recusive equation $s_{d+1} = s_d+3+d/2$ with initial condition
$s_1 = g+7/2$, that is, $s_d = d(d+11)/4+g+1/2 $.

\begin{theorem}
  \label{result_d_dim}
  Let $s>s_d$.  There exists a constant $C_3=C_3(s,d) >0$ such the 
  following holds true. Let $\alpha \in \Symp_{2g}(\R)$ and let $
  \Subgroup ^{d,\alpha} \subset \Heis$ be the $d$-dimensional Abelian
  subgroup of $\Heis$ generated by the base vector fields $X_1^\alpha,
  \dots , X_d^\alpha \in \hei$. Let $U_d(t):=[0,e^{t}]^d$. Let
  $\Birkhoff_{U_d(t)}^{d, \alpha}=\Birkhoff_{U_d(t)}^{d, \alpha} m$ be
  the Birkhoff sum associated to some $m\in M$ for the action of $
  \Subgroup ^{d,\alpha}_x$ , $(x\in \R^d)$.  Then, for all $t>0$ and
  all $m\in \M$, we have
  \begin{equation} \label{inductive_hypothesis}
    \begin{split}
      &\left\| \, [\alpha, \Birkhoff_{U_d(t)}^{d,\alpha} m] \, \right\| _{-s} \\
      & \leq C_3 \,  \sum_{k=0}^d\, \sum_{1\le i_1< \dots <i_k \le
        d}\,\,\, \int_0^{t}\dots\int_0^{t}
      \exp\Big(\tfrac {d} 2 t - {\tfrac 1 2\sum_{\ell=1}^k u_\ell }\Big)\\
      &\qquad\qquad\qquad\qquad\times \Height \Big( \class{ \prod_{{
            1\le j\le d}} r_j^{-t} \prod_{\ell=1}^k
        r_{i_\ell}^{u_\ell} \alpha}\Big)^{1/4}\,\D{u_1}\dots\D{u_k}.
    \end{split}
  \end{equation} 
\end{theorem}

\begin{proof} We argue by induction. The case $d=1$ is Theorem
  \ref{result_one_dim}. We assume the result holds for $d-1 \geq 1$.

  Set for simplicity $r^u = r^{u}_1\dots r^{u}_d $.

  Decomposing the current $ \Birkhoff^{d,\alpha}_{U_d(t)}m$ as
  in~\eqref{orthogonal_decomposition_Birkhoff_bundle} as a sum of a
  current $\mathcal Z^{-s} [ \alpha,\Birkhoff_{U_d(t)}^{d,\alpha} ]$
  and a current $\mathcal R^{-s} [
  \alpha,\Birkhoff_{U_d(t)}^{d,\alpha} ]$, we first estimate the
  boundary term $\big\|\, \mathcal Z^{-s} [
  \alpha,\Birkhoff_{U_d(t)}^{d,\alpha} ]\,\big\| _{-s}$. Using Lemma
  \ref{recursive_inequality} we have:
  \begin{equation}
    \label{eq:sect5:5}
    \begin{split}
      \left\|\, \mathcal Z^{-s} [ \alpha,\Birkhoff_{U_d(t)}^{d,\alpha}
        ]\, \right\| _{-s} \leq & \, e^{-dt/2 } \left\|\, \mathcal
        Z^{-s} [r^{-1} \alpha, \Birkhoff_{U_d(t)}^{d,\alpha} ]\,
      \right\| _{-s}
      \\
      & + C_1(s)\, \int_{0}^{t} e^{-{u} d/{2} } \| \, \mathcal R^{-s}
      [ r^{-u}\alpha, \Birkhoff_{U_d(t)}^{d,\alpha}] \, \|
      _{-(s-2)}\, \D{u}\\
      &= I +II.
    \end{split}
  \end{equation}

  By the covariance \eqref{eq:sect5:2}, Lemma
  \ref{trivial_bound_Sobolev} and Proposition
  \ref{best_Sobolev_constant2}, we have
  \[
  \begin{split}
    \left\|\, \mathcal Z ^{-s} [ r^{-1} \alpha,
      \Birkhoff_{U_d(t)}^{d,\alpha} ]\, \right\| _{-s} & = e^{dt} \,
    \left\|\, \mathcal Z^{-s} [ r^{-1} \alpha,
      \Birkhoff_{U_d(0)}^{d,r^{-t}
        \alpha}  ]\, \right\|  _{-s}  \\
    & \leq C \,e^{dt} \, \Height( \class{r^{-t} \alpha} )^{1/4}
  \end{split}
  \]
  Hence
  \begin{equation}
    \label{eq:sect5:8}
    I\le  C \,e^{dt/2} \, \Height( \class{r^{-t} \alpha} )^{1/4}
  \end{equation}
  corresponding to the term with $k=0$ in the statement of the
  theorem.

  To estimate the term $II$, we start observing that, using
  \eqref{eq:sect5:2} and Lemma~\ref{rest_estimate}, provided $s'<s
  -2-(d+1)/2$, we have
  \begin{equation}
    \label{eq:sect5:6}
    \begin{split}
      \big\| \mathcal R^{-s} [r^{-u} \alpha ,
      \Birkhoff^{d,\alpha}_{U_d(t)} ] \,\big\|_{-(s-2)}& = \big\|
      e^{ud} \, \mathcal R^{-s} [r^{-u} \alpha ,\Birkhoff ^{d,
        r^{-u}\alpha}_{U_d(t-u)}  ] \,\big\|_{-(s-2)} \\
      & \le C(s',s) \,e^{ud} \, \big\| \, [r^{-u}\alpha, \partial (
      \Birkhoff ^{d, r^{-u}\alpha}_{U_d(t-u)} ) ] \,\big\|_{-s'}.
    \end{split}
  \end{equation}

  The boundary $ \partial ( \Birkhoff ^{d, r^{-u}\alpha}_{U_d(t-u)} )
  $ is the sum of $2d$ currents of dimension $d-1$. These currents are
  Birkhoff sums of $d$ ``face'' subgroups $\Subgroup^{d-1,
    r^{-u}\alpha}_j$, $(j=1,\dots,d)$, obtained from $\Subgroup ^{d,
    r^{-u}\alpha}$ by omitting one of the base vector fields
  $X_1^\alpha,\dots,X_d^\alpha$. For each $j=1,\dots,d$ there are two
  Birkhoff sums of $\Subgroup^{d-1, r^{-u}\alpha}_j$ for points
  $m_{\pm j}$ along the $(d-1)$-dimensional cubes $U_{d-1,j}(t-u)$
  obtained from $U_{d}(t-u)$ by omitting the $j$-th factor interval
  $[0,e^{t-u}]$.

  If $s'>s_{d-1}$ (and therefore $s > s_{d-1}+(d+1)/2 +2=s_d$),
  denoting by $ \Birkhoff ^{d-1, r^{-u}\alpha}_{U_{d-1}(t-u)}$  the
  generic summand of $ \partial ( \Birkhoff ^{d,
    r^{-u}\alpha}_{U_d(t-u)} ) $, we may estimate the norm of each such
  boundary term using the inductive
  hypothesis~\eqref{inductive_hypothesis}. For the $j$-face we obtain
  \[
  \begin{split}
    &\left\| \, [r^{-u} \alpha , \mathcal P_{U_{d-1}(t-u)}^{d-1,r^{-u}
        \alpha}] \, \right\| _{-s'}
    \le  C_3(s',d-1)\, \sum_{k=0}^{d-1} \quad \sum_{\substack{1\le i_1< \dots <i_{k} \le d\\i_\ell\neq j}}\\
    & \quad\times \int_0^{t-u}
    \D{u_{i_1}}\cdots\int_0^{t-u}\!\!\!\D{u_{i_k}}
    \exp\Big({\tfrac {d-1} 2 (t-u)-\tfrac 1 2\sum_{\ell=1}^{k} u_{i_\ell} }\Big)\\
    &\qquad\qquad\qquad\qquad\times \Height \Big( \class{
      \prod_{\substack{{1\le \ell\le d}\\{\ell\neq j}}}
      r_\ell^{-(t-u)} \prod_{\ell=1}^{k} r_{i_\ell}^{u_{i_\ell}}
      \,r^{-u}\alpha}\Big)^{1/4}.
  \end{split} \] From~\eqref{eq:sect5:5} and~\eqref{eq:sect5:6} we
  obtain the following estiamate for the term $II$:
  \begin{equation}
    \label{eq:sect5:7}
    \begin{split} 
      &II \le C_4(s,d)\,\sum_{j=1}^{d}\,
      \sum_{k=0}^{d-1}\quad\sum_{\substack{1\le i_1< \dots <i_{k} \le
          d\\i_\ell\neq
          j}} \\
      & \quad\quad \times \int_{0}^t \D{u}
      \int_0^{t-u}\D{u_{i_1}}\cdots\int_0^{t-u}\!\!\!\D{u_{i_k}}
      \exp\Big(\tfrac {d-1} 2 t +\tfrac {1} 2 u -\tfrac 1 2\sum_{\ell=1}^{k} u_{i_\ell} \Big)\\
      &\qquad\qquad\qquad\times \Height \Big( \class{ \prod_{ 1\le
          \ell\le d} r_\ell^{-t} \prod_{\ell=1}^{k}
        r_{i_\ell}^{u_{i_\ell}} \,r_j^{-u+t}\alpha}\Big).
    \end{split}
  \end{equation}
  Applying the change of variable $u_j=t-u$, majorizing the integrals
  $\int_0^{t-u}$ with integrals $\int_0^{t}$ and observing that there
  are at most $k+1$ integer intervals $]i_t,i_{t+1}[$ in which the
  integer $j$ in the above sum may land, we obtain
  \begin{equation}
    \label{eq:sect5:9}
    \begin{split} 
      & II \le C_4(s,d)\,\sum_{j=1}^{d}\,
      \sum_{k=0}^{d-1}\quad\sum_{\substack{1\le i_1< \dots <i_{k} \le
          d\\i_\ell\neq j}}
      \\
      & \quad\quad \times \int_{0}^t \D{u_j}
      \int_0^{t-u}\D{u_{i_1}}\cdots\int_0^{t-u}\!\!\!\D{u_{_{i_k}}}
      \exp\Big(\tfrac {d} 2 t -\tfrac {1} 2 u_j -\tfrac 1
      2\sum_{\ell=1}^{k} u_{i_\ell} \Big)
      \\
      &\qquad\qquad\qquad\times \Height \Big( \class{ \prod_{ 1\le
          \ell\le d} r_\ell^{-t} \prod_{\ell=1}^{k}
        r_{i_\ell}^{u_{i_\ell}} \,r_j^{-u_j}\alpha}\Big).
      \\
      & \quad\le C_5(s,d)\sum_{k=1}^{d}\quad\sum_{1\le i_1< \dots
        <i_{k} \le d} \quad
      \int_0^{t}\D{u_{i_1}}\cdots\int_0^{t}\D{u_{_{i_k}}}
      \\
      &\quad\quad\quad\quad\quad\quad\quad\quad\times \exp\Big(\tfrac
      {d} 2 t -\tfrac 1 2\sum_{\ell=1}^{k} u_{i_\ell} \Big)\Height
      \Big( \class{ \prod_{ 1\le \ell\le d} r_\ell^{-t}
        \prod_{\ell=1}^{k} r_{i_\ell}^{u_{i_\ell}} \alpha}\Big).
    \end{split}
  \end{equation}   
  The reminder term $\mathcal R^{-s} [
  \alpha,\Birkhoff_{U_d(t)}^{d,\alpha} ]$ in the decomposition
  \eqref{orthogonal_decomposition_Birkhoff_bundle} is estimated using
  Lemma \ref{rest_estimate}, Proposition \ref{best_Sobolev_constant2}
  and Lemma \ref{trivial_Diophantine_bound}. We have:
  \begin{equation}
    \label{eq:sect5:10}
    \begin{split} \left\| \,\mathcal R^{-s} [ \alpha,\Birkhoff_{U_d(t)}^{d,\alpha} ] \, \right\| _{-s} & \leq C(s) \, \Height([\alpha]) ^{1/4} \\ & = C(s) \,
      \Height( \class{ r^{t} r^{ -t} \alpha} ) ^{1/4} \\ &
      \leq C(s) \,e ^{td/2} \, \Height(\class{r^{ -t} \alpha})
      ^{1/4} \, , 
    \end{split}
  \end{equation}
  producing one more term like \eqref{eq:sect5:8}.  The theorem
  follows from the estimates \eqref{eq:sect5:8} and
  \eqref{eq:sect5:9}, for the terms I and II, and \eqref{eq:sect5:10}
  for the remainder.
\end{proof}

Different possible asympthotics are then consequences of the
Diophantine conditions \eqref{Diophantine}, \eqref{Roth} and
\eqref{bounded}, or the Kleinbock-Margulis logarithm law (Proposition
\ref{Sobolev_constant_estimate}).

\paragraph{Proof of  Theorem \ref{equidistribution_intro}.} 
Let the notations as in Theorem \ref{result_d_dim}, and consider the
integrals in \eqref{inductive_hypothesis}.  It follows from Lemma
\ref{trivial_Diophantine_bound} that, for any $0 \leq k \leq d$,
\[
\Height \Big( \class{ \prod_{{ 1\le j\le d}} r_j^{-t} \prod_{\ell=1}^k
  r_{i_\ell}^{u_\ell} \alpha}\Big)^{1/4} \leq e^{\frac{1}{2} \sum
  _{\ell =1}^k u_k } \Height \Big( \class{ \prod_{{ 1\le j\le d}}
  r_j^{-t} \alpha}\Big)^{1/4} \, . 
\]
There follows from \eqref{inductive_hypothesis} that
\begin{equation} \label{final} \left\| \, [\alpha,
    \Birkhoff_{U_d(t)}^{d,\alpha} ] \, \right\| _{-s} \leq C \, t^d \,
  e^{\frac{d}{2}t} \Height \Big( \class{ \prod_{{ 1\le j\le d}}
    r_j^{-t} \alpha}\Big)^{1/4}
\end{equation} 
for some constant $C=C(s,d)$. Therefore the norms of our currents
depend on the Diophantine properties of $\alpha$ in the direction of
$\dcart(d) := \dcart_1+\dots +\dcart_d \in \cartan$ (recall that $r_i^t
= e^{t\dcart_i}$), defined in \ref{def:Diophantine_properties}. For
example, if $\alpha$ satisfies a $\dcart(d)$-Diophantine condition
\eqref{Diophantine} of exponent $\sigma>0$, we get
\[
\left\| \, [\alpha, \Birkhoff_{U_d(t)}^{d,\alpha} ] \, \right\| _{-s}
\leq C \, t^d \, e^{d(1-\sigma /2)t} \leq C' e^{d(1-\sigma'/2)t}
\]
for all $\sigma ' <\sigma$.  If $\alpha$ satisfies a $\dcart(d)$-Roth
condition \eqref{Roth}, we get
\[
\left\| \, [\alpha, \Birkhoff_{U_d(t)}^{d,\alpha} ] \, \right\| _{-s}
\leq C \, e^{(d/2+\varepsilon)t}
\]
for all $\varepsilon>0$.
If $\alpha$ is of bounded type, i.e. satisfies \eqref{bounded}, then
all the ``Height'' terms inside the integrals of
\eqref{inductive_hypothesis} are bounded, and we get
\[
\left\| \, [\alpha, \Birkhoff_{U_d(t)}^{d,\alpha} ] \, \right\| _{-s}
\leq C \, e^{(d/2)t}.
\]
On the other side, according to the easy part of Kleinbock and
Margulis theorem \ref{Khinchin-Kleinbock-Margulis_law}, there exists a
full measure set $\Omega_g (\dcart(d)) \subset \Sigma _g$ such that if
$\class{\alpha} \in \Omega _g(\dcart(d))$ and $\varepsilon >0$ then
\[
\Height \Big( \class{\prod_{{ 1\le j\le d}} r_j^{-t}
  \alpha}\Big)^{1/4} \leq C t^{1/(2g+2) + \varepsilon} \, .
\]
There follows from \eqref{final} that for such $\alpha$'s \[ \left\|
  \, [\alpha, \Birkhoff_{U_d(t)}^{d,\alpha} ] \, \right\| _{-s} \leq C
\, t^{d+1/(2g+2) + \varepsilon } \, e^{(d/2)t} \, . \]

\subsection{Birkhoff averages and Theta sums} \label{sec_theta_sums}

\paragraph{First return map.} 
Here it is convenient to work with the ``polarized'' Heisenberg group,
the set $\Heis_{\mathrm{pol}} \approx \R^g \times \R^g \times \R$
equipped with the group law $(x,\xi,t) \cdot (x',\xi ', t') =
(x+x',\xi+\xi ', t+t'+\xi x')$.  The homomorphism $\Heis \to \Heis
_{\mathrm{pol}}$, as well as the exponential map $\exp : \hei \to
\Heis _{\mathrm{pol}}$, is $(x,\xi, t) \mapsto (x, \xi, t +
\frac{1}{2}\xi x )$.  Define the ``reduced standard Heisenberg group''
$\Heis_{\mathrm{red}} := \Heis_{\mathrm{pol}} / (\{ 0 \} \times \{ 0
\} \times \tfrac{1}{2} \Z) \approx \R^g \times \R^g \times (\R /
\tfrac{1}{2}\Z)$, and then the ``reduced standard lattice''
$\Lattice_{\mathrm{red}} := \Z^g \times \Z^g \times \{ 0 \} \subset
\Heis_{\mathrm{red}}$. It is clear that the quotient $
\Heis_{\mathrm{red}} / \Lattice_{\mathrm{red}} \approx \Heis /
\Lattice$ is the standard nilmanifold.  The subgroup $\mathsf N =
\{(0, \xi , t ) \, \text{with } \, \xi \in \R^g \, , \, t \in \R /
\tfrac{1}{2}\Z \}$ is a normal subgroup of $\Heis_{\mathrm{red}}$. The
quotient $\Heis_{\mathrm{red}} / \mathsf N$ is isomorphic to the
Lagrangian subgroup $ \Subgroup = \{ (x, 0, 0) \, \text{with }\, x \in
\R^g \}$, and we have an exact sequence $ 0 \rightarrow \mathsf N
\rightarrow \Heis_{\mathrm{red}} \rightarrow \Subgroup \rightarrow 0$.
Therefore $\Heis_{\mathrm{red}} \approx \Subgroup \ltimes \mathsf N$,
and in particular any $(x, \xi , t) \in \Heis_{\mathrm{red}}$ may be
uniquely written as the product
\[
(x, \xi , t ) = \exp (x_1 X_1+\dots +x_gX_g) \cdot (0,\xi,t) = (x, 0,
0) \cdot (0, \xi , t ) \, .
\]
Given a symmetric $g \times g$ real matrix $\fq$, we consider the
symplectic matrix $\alpha= \left( \begin{smallmatrix} I & 0 \\ \fq &
    I \end{smallmatrix} \right) \in \Symp_{2g} (\R)$.  Then $ \exp
(x_1 X^\alpha_1+\dots +x_gX_g^\alpha) = (x, -\fq x , - x^\top \fq x
)$, and any element of $\Heis_{\mathrm{red}}$ can be written uniquely
as a product
\begin{eqnarray*}
  \exp (x_1 X^\alpha_1+\dots +x_gX_g^\alpha) \cdot  (0,\xi, t)  = (x, \xi - \fq x, t - \tfrac{1}{2} x ^\top  \fq x )
\end{eqnarray*}
for some $x \in \R^g$, $\xi \in \R^g$ an $t \in \R / \tfrac{1}{2} \Z$.
Given $n \in \Z^g$, $m \in \Z^g$, hence $(n,m,0 ) \in
\Lattice_{\mathrm{red}}$, then
\begin{equation} \label{first_return_map} \exp (x_1 X^\alpha_1+\dots
  +x_gX^\alpha_g) \cdot (0,\xi, t ) \cdot (n,m,0) = \exp (x'_1
  X^\alpha_1+\dots +x'_gX^\alpha_g) \cdot (0,\xi', t') \end{equation}
if and only if $x' = x+ n $, $ \xi '= \xi + m +\fq n $ and $ t' = t +
\xi^\top n +\tfrac{1}{2} n^\top \fq n +\tfrac{1}{2} \Z$.

\paragraph{Birkhoff averages of certain functions on the circle.} Let
$\varphi \in \SB \left( \R / \tfrac{1}{2}\Z \right)$, and let $\psi
\in \mathcal E (\R^g)$ be a smooth function with compact support.
Define a function $\phi : \Heis_{\mathrm{red}}\approx \alpha^{-1}
(\Subgroup) \ltimes \mathsf N \to \CC$ as the product
\[
\phi ( \exp(x_1 X^\alpha_1+\dots +x_gX^\alpha_g) \cdot (0, \xi , t ))
:= \psi (x) \cdot \varphi (t)
\]
and then a function $\tilde{\phi} :\M \to \CC$ on the quotient
standard nilmanifold summing over the lattice
$\Lattice_{\mathrm{red}}$. Namely, if $m = \exp(x_1 X^\alpha_1+\dots
+x_gX^\alpha_g) \cdot (0,\xi,t) \cdot \Lattice_{\mathrm{red}} \in \M$,
we set
\[
\begin{split} \tilde{\phi} ( m ) & : =
  \sum _{(n,m,0) \in \Lattice_{\mathrm{red}}}  \phi  ( \exp(x_1 X^\alpha_1+\dots +x_gX^\alpha_g) \cdot (0, \xi , t ) \cdot (n,m,0) ) \\
  & = \sum _{n \in \Z^g} \psi \left( x+ n \right) \cdot \varphi \left(
    t + \xi^\top n +\tfrac{1}{2} n ^\top \fq n \right)
\end{split}
\]
where we used \eqref{first_return_map}.  Since $\psi$ has compact
support, this sum is finite, so that $\tilde{\phi}$ is indeed a smooth
function.  The Birkhoff average of $\omega = \tilde{\phi} \,
dX_1^\alpha \wedge \dots \wedge dX_g^\alpha$ along the current
$\Subgroup ^{g,\alpha}_Um$ with $m \in \M$ as above is, according to
\eqref{Birkhoff},
\begin{eqnarray*}
  \lin \mathcal P^{g,\alpha}_U m , \omega \rin 
  = \sum _{n \in \Z^g } 
  \left( \varphi \left( t  + \xi^\top n  +\tfrac{1}{2} n^\top \fq n \right)   \cdot \int _U \psi  (y+x+ n ) \, dy \right) \, . 
\end{eqnarray*}
Let $0<\delta <1/2$, and choose a test function $\psi \in \mathcal E
(\R^g)$ with support in a small ball $B_\varepsilon (0) = \{x \in \R^g
\, \, \text{s.t. }\, |x|_\infty \leq \varepsilon \}$ of radius
$0<\varepsilon <\delta $, and unit mass $\int _{\R^g} \psi (x) \, dx =
1$. For $N$ a positive integer, $U = [-\delta , N+\delta ]^g$ and
$x=0$, we have
\begin{equation} \label{Birkhoff_sums_theta} \lin \mathcal
  P^{g,\alpha}_U m , \omega \rin = \sum_{n \in \Z^g \cap [0, N]^g}
  \varphi \left( t + \xi^\top n +\tfrac{1}{2} n^\top \fq n \right)
\end{equation}
There follows from Theorem \ref{equidistribution_intro} in the
Introduction and the above discussion (i.e. formula
\ref{Birkhoff_sums_theta}) that

\begin{theorem}\label{pretheta_sums} 
  Let $\fq [x] =  x^\top \fq x$ be the quadratic forms defined by the  symmetric $g \times g$ real matrix $\fq$, $\alpha =
  \left( \begin{smallmatrix} I & 0 \\ \fq & I \end{smallmatrix} \right)
  \in \Symp_{2g}(\R)$,  $\fl (x) = \fl ^\top x$ be the linear form defined by $\fl \in \R^g$, and $t \in \R$.  Then,

  \begin{itemize}
  \item there exists a full measure set $\Omega_g \subset \Mod$ such
    that if $[\alpha] \in \Omega _g$ and $\varepsilon >0$ then
    \[
    \sum _{n \in \Z^g \cap [0, N]^g} \varphi \left( t + \fl (n)
      +\tfrac{1}{2}  \fq [n] \right)= \mathcal O \left( (\log N
      )^{g+1/(2g+2) + \varepsilon} \, N^{g/2} \right) \]
  \item if $[\alpha] \in \Mod$ satisfies a $\dcart(g)$-Roth condition, then
    for any $\varepsilon >0$
    \[
    \sum _{n \in \Z^g \cap [0, N]^g} \varphi \left( t +  \fl (n)
      +\tfrac{1}{2} \fq [n] \right)= \mathcal O \left( N^{g/2
        +\varepsilon} \right)
    \]
  \item if $[\alpha] \in \Mod$ is of bounded type, then
    \[
    \sum _{n \in \Z^g \cap [0, N]^g} \varphi \left( t +  \fl (n)
      + \tfrac{1}{2} \fq [n] \right)= \mathcal O \left( N^{g/2}
    \right)
    \]
  \end{itemize}
  as $N \to \infty$, for any test function $\varphi \in \Sobolev ^s(\R
  / \tfrac{1}{2}\Z)$ with Sobolev order $s>s_g $ and zero
  average $ \int_0^{1/2}\varphi (t ) \, dt=0$.
\end{theorem}

Corollary \ref{theta_sums} in the Introduction follows if we take
$\varphi (t) =e^{4\pi it }$.

\bibliography{heisenberg} \bibliographystyle{amsalpha}

\end{document}